\documentclass[reqno,11pt]{article}
\usepackage{a4wide}
\usepackage{color,eucal,enumerate,mathrsfs, amsthm}
\usepackage[normalem]{ulem}
\usepackage{amsmath,amssymb,epsfig,bbm}
\numberwithin{equation}{section}

\usepackage[pdfborder={0 0 0}]{hyperref}  

\usepackage[latin1]{inputenc}
\usepackage[english]{babel}

\newcommand{\D}{\mathbb{D}}

\renewcommand{\H}{\mathbb{H}}

\newcommand{\N}{\mathbb{N}}

\newcommand{\R}{\mathbb{R}}
\newcommand{\Z}{\mathbb{Z}}

\newcommand{\mm}{{\mbox{\boldmath$m$}}}

\newcommand{\ppi}{{\mbox{\boldmath$\pi$}}}

\newcommand{\ssigma}{{\mbox{\boldmath$\sigma$}}}

\newcommand{\sfd}{{\sf d}}

\newcommand{\sfh}{{\sf h}}

\newcommand{\Id}{{\rm Id}}                          
\newcommand{\Kliminf}{K\kern-3pt-\kern-2pt\mathop{\rm lim\,inf}\limits}  
\newcommand{\supp}{\mathop{\rm supp}\nolimits}   
\newcommand{\Lip}{\mathop{\rm Lip}\nolimits}          
\renewcommand{\d}{{\mathrm d}}

\newcommand{\restr}[1]{\lower3pt\hbox{$|_{#1}$}}
\newcommand{\la}{\langle}                  
\newcommand{\ra}{\rangle}
\newcommand{\eps}{\varepsilon}  
\newcommand{\nchi}{{\raise.3ex\hbox{$\chi$}}}
\newcommand{\weakto}{\rightharpoonup}

\newcommand{\limi}{\varliminf}
\newcommand{\lims}{\varlimsup}

\newcommand{\fr}{\penalty-20\null\hfill$\blacksquare$}                      

\newcommand{\prob}[1]{\mathscr P(#1)}                   
\newcommand{\probt}[1]{\mathscr P_2(#1)}                   
\newcommand{\e}{{\rm{e}}}                           

\renewcommand{\mm}{\mathfrak m}                                

\renewenvironment{proof}{\removelastskip\par\medskip   
\noindent{\em proof} \rm}{\penalty-20\null\hfill$\square$\par\medbreak}

\newtheorem{theorem}{Theorem}[section]

\newtheorem{corollary}[theorem]{Corollary}
\newtheorem{lemma}[theorem]{Lemma}
\newtheorem{proposition}[theorem]{Proposition}
\newtheorem{assumption}[theorem]{Assumption}

\newtheorem{definition}[theorem]{Definition}

\newtheorem{remark}[theorem]{Remark}

\newcommand{\bd}{{\mathbf\Delta}}

\newcommand{\test}[1]{{\rm Test}(#1)}

\renewcommand{\b}{{\rm b}}
\newcommand{\X}{{\rm X}}
\newcommand{\Y}{{\rm Y}}
\renewcommand{\Z}{{\rm Z}}

\newcommand{\h}{{\sfh}}
\newcommand{\mau}{{\sf T}}
\newcommand{\mad}{{\sf S}}

\newcommand{\CD}{{\sf CD}}
\newcommand{\RCD}{{\sf RCD}}

\newcommand{\lip}{{\rm lip}}

\newcommand{\HS}{{\lower.3ex\hbox{\scriptsize{\sf HS}}}}
\renewcommand{\H}[1]{{\rm Hess}(#1)}

\renewcommand{\O}{{\sf O}}
\renewcommand{\D}{{\rm D}}

\newcommand{\fsm}[1]{\test#1}

\newcommand{\fl}{{\sf G}}
\newcommand{\gl}{{\sf Fl}}
\newcommand{\pr}{{\sf Pr}}

\newcommand{\Rb}{{\sf R}}
\newcommand{\rb}{{\sf r}}
\newcommand{\Rbt}{\bar{\sf R}}

\title{From volume cone to metric cone in the nonsmooth setting}
\begin{document}

\author{
   Nicola Gigli
   \thanks{SISSA, ngigli@sissa.it}\and  Guido de Philippis\ \thanks{Unit\'e de Math\'ematiques Pure et Appliqu\'ees -- ENS de Lyon. guido.de-philippis@ens-lyon.fr} }

\maketitle

\begin{abstract}
We prove that `volume cone implies metric cone' in the setting of $\RCD$ spaces, thus generalising to this class of spaces a well known result of Cheeger-Colding valid in Ricci-limit spaces.

\end{abstract}


\tableofcontents

\section{Introduction}

In the study of measured-Gromov-Hausdorff limits of Riemannian manifolds with Ricci curvature uniformly bounded from below, Ricci-limit spaces in short, as developed by Cheeger and Colding (\cite{Cheeger-Colding96}, \cite{Cheeger-Colding97I}, \cite{Cheeger-Colding97II}, \cite{Cheeger-Colding97III}), two almost rigidity results play a key role: the almost splitting theorem and the almost volume cone implies almost metric cone. By nature, both these results imply corresponding rigidity results for Ricci-limit spaces and in fact also the converse implication holds provided one is willing to give up the precise quantification given by the almost rigidity versions.

In the seminal papers \cite{Lott-Villani09} and \cite{Sturm06I}, \cite{Sturm06II}, Lott-Villani and Sturm proposed a synthetic definition of lower Ricci curvature bounds for metric-measure spaces based on optimal transport: according to their approach, spaces with Ricci curvature bounded from below by $K$ and dimension bounded from above by $N$ are called $\CD(K,N)$  spaces. Later on, mostly for technical reasons related to the local-to-global property, Bacher-Sturm introduced { in \cite{BacherSturm10}} a variant of the $\CD(K,N)$ condition, called reduced curvature dimension condition and denoted $\CD^*(K,N)$.

Key features of both the $\CD$ and $\CD^*$ conditions are the compatibility with the Riemannian case and the stability w.r.t.\ measured-Gromov-Hausdorff convergence. In particular, they include Ricci-limit spaces and it is natural to wonder whether the aforementioned geometric rigidity result hold for these structures. However, this is not the case, as both  $\CD$ and  $\CD^*$ structures include Finsler geometries (see the last theorem in  \cite{Villani09} and \cite{Ohta09}) and it is therefore natural to look for stricter conditions which, while retaining the crucial stability properties of Lott-Sturm-Villani spaces, rule out Finsler structures. 

A first step in this direction has been made by Ambrosio, Savar\'e and the first author in \cite{AmbrosioGigliSavare11-2}, where the notion of $\RCD(K,\infty)$ spaces (the `${\sf R}$' stands for Riemannian) have been introduced via means related to the study of the heat flow. Partly motivated by this approach the first author in \cite{Gigli12} proposed a strengthening of the $\CD/\CD^*$ conditions based solely on properties of Sobolev functions: the added requirement is that the Sobolev space $W^{1,2}$ is an Hilbert space, a condition called infinitesimal Hilbertianity, and the resulting classes of spaces are denoted $\RCD/\RCD^*$.

It turns out that the a priori purely analytic notion of infinitesimal Hilbertianity grants geometric properties, the reason being that it allows to make computations mimicking the calculus in Riemannian (as opposed to Finslerian) manifolds. The  first example in this direction has been the Abresch-Gromoll inequality proved by the first author and Mosconi in \cite{Gigli-Mosconi12}. Other relevant geometric properties, both on their own and for the purposes of the current paper, are the splitting theorem, proved in \cite{Gigli13} by the first author, and the maximal diameter theorem, proved by Ketterer in \cite{Ketterer13}. Another remarkable result has been established by Erbar-Kuwada-Sturm in \cite{Erbar-Kuwada-Sturm13}: they proved that, in a suitable sense,  $\RCD^*(K,N)$ spaces can be characterized as those spaces where the Bochner inequality with parameters $K,N$ holds (the case $N=\infty$ was already known by \cite{Gigli-Kuwada-Ohta10} and \cite{AmbrosioGigliSavare11-2}, see also \cite{AmbrosioGigliSavare12}).

\bigskip

The focus of this paper it to prove the non-smooth version of the `volume cone implies metric cone', our result being:
\begin{theorem}\label{thm:mainthm}
Let $K\in\R$, {$N\in (0,\infty)$}, $(\X,\sfd,\mm)$ a $\RCD^*(0,N)$ space with $\supp(\mm)=\X$, $\O\in \X$ and $\Rb>\rb>0$ such that
\[
\mm(B_{\Rb}(\O))=\Big(\frac{\Rb}{\rb}\Big)^N\mm(B_{\rb}(\O)).
\]
Then exactly one of the following holds:
\begin{itemize}
\item[1)] $S_{\Rb/2}(\O)$ contains only one point. In this case $(\X,\sfd)$ is isometric to $[0,{\rm diam}(\X)]$ ($[0,\infty)$ if $\X$ is unbounded) with an isometry which sends $\O$ in $0$ and the measure $\mm\restr{B_\Rb}(\O)$ to  the measure $c\, x^{N-1}\d x$ for  $c:=N \mm(B_\Rb(\O))$.
\item[2)] $S_{\Rb/2}(\O)$ contains  two points. In this case  $(\X,\sfd)$ is a 1-dimensional  Riemannian manifold,  possibly with boundary, and there is a bijective local isometry (in the sense of distance-preserving maps) from $B_{\Rb}(\O)$ to $(-\Rb,\Rb)$ sending $\O$ to $0$ and the measure  $\mm\restr{B_\Rb(\O)}$ to the measure  $c\,|x|^{N-1} \d x$ for $c:=\frac12N\mm(B_\Rb(\O))$. Moreover, such local isometry is an  isometry when restricted to $\bar B_{\Rb/2}(\O)$.
\item[3)] $S_{\Rb/2}(\O)$ contains more than two points. In this case { \(N\ge 2\) and there exists a a \(\RCD^*(N-2,N-1)\) space $(\Z,\sfd_\Z,\mm_\Z)$ with \({\rm diam } (\Z)\le\pi\) such that the ball \(B_{\Rb}(\O)\) is locally isometric to the ball \(B_{\Rb}(\O_\Y)\) of the cone \(Y\) built over \(\Z\). Moreover, such local isometry is an  isometry when restricted to $\bar B_{\Rb/2}(\O)$.}

\end{itemize} 
\end{theorem}

{ Some remarks are in order. First of all 
notice that much like the non-smooth splitting, this result gives some new information also about the smooth world: the space in now known to be locally a cone over a $\RCD^*(N-2,N-1)$ space, rather than simply over a length space. We remark that such information about the structure of the sphere will come after we proved the cone structure of the original space via a direct use of Ketterer's results \cite{Ketterer13} about lower Ricci bounds for  cones.

Actually the the \(\RCD^*(N-2,N-1)\) space $(\Z,\sfd_\Z,\mm_\Z)$ will be given by an appropriate rescaling of the sphere $S_{\Rb/2}(\O)$ once this is endowed with the ``natural'' induced distance and measure, see Section \ref{se:xp} for more details.

As a side remark, let us also point out that the classical smooth version of the equality cases in Bishop-Gromov inequality (i.e. the smooth version of Theorem \ref{thm:mainthm}) is usually stated  saying that \(B_{\Rb}(\O)\subset \X\)  and \(B_{\Rb}(\O_\Y)\subset\Y\) are {\em isometric}. Hence Theorem  \ref{thm:mainthm} seems a priori weaker of the corresponding smooth version. This is actually not the case:  indeed it is classical in Differential Geometry to say that two Riemannian manifolds are isometric if the pull  back the metric  tensors coincide. It is easy seen that  this however implies only locally isometry as metric spaces.  Optimality of our result (in the non smooth as well as in the smooth setting) can be checked by looking at the ball of radius \(1\) on the flat torus \(\mathbb T^N\), \(N\in \N\).} 

We focus on the case $K=0$  for simplicity, but in fact our techniques can be easily adapted to general $K$'s as well as to the case of `volume annulus'. We shall briefly mention this in the last section.

Let us now discuss the proof of our result. The structure and techniques used closely resemble that of the splitting theorem (compare with \cite{Gigli13} and \cite{Gigli13over}) with all the metric informations being read at the level of Sobolev functions and only at the very end translated back into metric properties.  We shall take advantage of this analogy in order to skip some lengthy detail whenever a closely related lemma has been already proved for the case of the splitting. Still, there are two crucial differences which need non-trivial adaptations. 

The first is that while the splitting is a result about the structure of the whole space, here we only deal with a portion of it. This means that the `regular' part of our space is confined in the ball of interest and appropriate cut-off arguments have to be used in order to justify the computations needed. Much part of this discussion will be in Section \ref{se:gfdir}.

The second, and most important, difference is that in the splitting the level set of the Busemann function is (proved to be) isometrically embedded in the original space and this fact can be used to quickly gain crucial informations about its structure. In the current framework, instead, such  level set is a sphere that we should consider with the induced intrinsic distance. As such, it is not really its metric structure being inherited by the ambient space, but rather its differential one and this shift from `metric' to `differential' calculus requires appropriate  tools to be handled because concepts like `vector fields' and `Hessian' must come into play. This non-trivial analysis is carried over in Section \ref{se:speedproj}, which amount to one third of the body of the paper, and makes use of the vocabulary proposed by the first author in \cite{Gigli14}.

\section{Preliminaries} \label{se:prel}
We assume the reader familiar with calculus on metric measure spaces and the $\RCD$ condition, here we recall only few basic things as reference for what comes next, mainly in order to fix the notation. Other results will be recalled in the body of the paper, whenever needed.

Let $(\X,\sfd,\mm)$ be a complete and separable metric space equipped with a non-negative Borel measure. The local Lipschitz constant $\lip(f):\X\to[0,\infty]$ of a function $f:\X\to\R$ is defined as
\[
\lip(f)(x):=\lims_{y\to x}\frac{|f(x)-f(y)|}{\sfd(x,y)},
\]
if $x$ is not isolated, 0 otherwise.

A Borel probability measure $\ppi$ on $C([0,1],\X)$ is said of bounded compression provided for some $C>0$ it holds
\[
(\e_t)_*\ppi\leq C\mm,\qquad\forall t\in[0,1],
\]
where $\e_t:C([0,1],\X)\to \X$ is the evaluation map sending a curve $\gamma$ to $\gamma_t$. By kinetic energy of a curve $\gamma$ we intend $\frac12\int_0^1|\dot\gamma_t|^2\,\d t$, the integral being intended $+\infty$ if the curve is not absolutely continuous. Then the kinetic energy of a probability measure $\ppi$ on $C([0,1],\X)$ is $\frac12\iint_0^1|\dot\gamma_t|^2\,\d t\,\d\ppi(\gamma)$. 

If $\ppi$ has finite kinetic energy and is of bounded compression, then it is called \emph{test plan}. 

For $t_0,t_1\in[0,1]$ the map ${\rm Restr}_{t_0}^{t_1}:C([0,1],\X)\to C([0,1],\X)$ is defined as
\[
{\rm Restr}_{t_0}^{t_1}(\gamma)_t:=\gamma_{(1-t)t_0+tt_1}.
\]
Notice that if $\ppi$ is a test plan, $t_0\neq t_1\in [0,1]$ and $\Gamma\subset C([0,1],\X)$ is such that $\ppi(\Gamma)>0$, then the plan $\ppi(\Gamma)^{-1}({\rm Restr}_{t_0}^{t_1})_*(\ppi\restr{\Gamma})$ is also a test plan.

The Sobolev class $S^2(\X)$ is the collection of all Borel functions $f:\X\to\R$ for which there is $G\in L^2(\X)$, $G\geq 0$, such that
\[
\int|f(\gamma_1)-f(\gamma_0)|\,\d \ppi(\gamma)\leq\iint_0^1G(\gamma_t)|\dot\gamma_t|\,\d t\,\d\ppi(\gamma),
\]
for  every test plan $\ppi$. Any such $G$ is called weak upper gradient for $f$ and it turns out that for $f\in S^2(\X)$ there exists a minimal $G$ in the $\mm$-a.e.\ sense, which is called minimal weak upper gradient and denoted by $|\D f|$. Among other properties, the minimal weak upper gradient is local in the sense that for every $f,g\in S^{2}(\X)$ we have
\begin{equation}
\label{eq:locwug}
|\D f|=|\D g|,\qquad\mm-a.e.\ on\ \{f=g\},
\end{equation}
see \cite{AmbrosioGigliSavare11} and \cite{Gigli12} for more details and for the proof of the following equivalent characterization:
\begin{proposition}\label{prop:sobf}
The following are equivalent:
\begin{itemize}
\item[i)] $f\in S^2(\X)$ and $G$ is a weak upper gradient.
\item[ii)] for every test plan $\ppi$ the following holds. For $\ppi$-a.e.\ $\gamma$ the function $t\mapsto f(\gamma_t)$ is equal a.e.\ on $[0,1]$ and on $\{0,1\}$ to an absolutely continuous function $f_\gamma$  such that
\[
|\partial_tf_\gamma|(t)\leq G(\gamma_t)|\dot\gamma_t|,\qquad\ a.e.\ t\in[0,1].
\]
\end{itemize}
\end{proposition}

The Sobolev space $W^{1,2}(\X)$ is defined as $L^2\cap S^2(\X)$ and is equipped with the norm
\[
\|f\|_{W^{1,2}}^2:=\|f\|_{L^2}^2+\||\D f|\|_{L^2}^2.
\]
The locality property \eqref{eq:locwug} allows to introduce the space $S^{2}_{loc}(\X)$ (resp. $W^{1,2}_{loc}(\X)$) as the space of those functions locally equal to some function in $S^{2}(\X)$  (resp. $W^{1,2}(\X)$). These functions come with a natural weak upper gradient which belongs to $L^2_{loc}(\X)$. Similarly, for $\Omega\subset \X$ open, the space $S^2(\Omega)$ (resp. $W^{1,2}(\Omega)$) is defined as the space of those functions locally in $\Omega$ equal to some function in $S^{2}(\X)$ (resp. $W^{1,2}(\X)$) such that $|\D f|\in L^2(\Omega)$ (resp. $f,|\D f|\in L^2(\Omega)$). Notice that Lipschitz functions $f$ always belong to $W^{1,2}_{loc}(\X)$ and that
\begin{equation}
\label{eq:lipwug}
|\D f|\leq \lip(f),\quad \mm-a.e..
\end{equation}
A test plan $\ppi$ is said to \emph{represent} the gradient of $f\in S^{2}(\X)$ provided 
\[
\limi_{t\downarrow0}\int\frac{f(\gamma_t)-f(\gamma_0)}{t}\,\d\ppi(\gamma)\geq\frac12\int|\D f|^2(\gamma_0)\,\d\ppi(\gamma)+\frac12\lims_{t\downarrow0}\iint_0^t|\dot\gamma_s|^2\,\d s\,\d\ppi(\gamma).
\]
Notice that the opposite inequality always holds. If $f$ is only in $S^2(\Omega)$ for some open set $\Omega$, we add the requirement that $(\e_t)_*\ppi$ is concentrated on $\Omega$ for every $t\in[0,1]$ sufficiently small.

Spaces $(\X,\sfd,\mm)$ such that $W^{1,2}(\X)$ is an Hilbert space are called \emph{infinitesimally Hilbertian} (see \cite{Gigli12}, to which we also refer for the differential calculus recalled below).

It turns out that
\begin{equation}
\label{eq:hilbloc}
\text{ on inf. Hilb. spaces the map $S^2(\X)\ni f\mapsto |\D f|^2\in L^1(\X)$ is a quadratic form.}
\end{equation} 
By polarization, it induces a bilinear and symmetric map
\[
S^2(\X)\ni f,g\qquad\mapsto\qquad\la\nabla f,\nabla g\ra\in L^1(\X),
\] 
which satisfies $\la\nabla f,\nabla f\ra=|\D f|^2$, the `Cauchy-Schwarz' inequality $|\la\nabla f,\nabla g\ra|\leq|\D f||\D g|$ and the chain and Leibniz rules
\begin{equation}
\label{eq:calcrules}
\begin{split}
\la\nabla(\varphi\circ f),\nabla g\ra&=\varphi'\circ f\la\nabla f,\nabla g\ra,\\
\la\nabla f,\nabla(g_1g_2)\ra&=g_1\la\nabla f,\nabla g_2\ra+g_2\la\nabla f,\nabla g_1\ra,
\end{split}
\end{equation}
for any $f,g\in S^2(\X)$, $g_1,g_2\in S^2\cap L^\infty(\X)$, $\varphi\in C^1(\R)$ with bounded derivative.

A simple yet crucial result is the following \emph{first order differentiation formula}: if $(\X,\sfd,\mm)$ is infinitesimally Hilbertian, $f,g\in S^2(\Omega)$ for some open set $\Omega$ and $\ppi$ represents the gradient of $f$ and is such that $(\e_t)_*\ppi$ is concentrated on $\Omega$ for every $t\in[0,1]$ sufficiently small, then
\begin{equation}
\label{eq:firstorder}
\lim_{t\downarrow0}\int\frac{g(\gamma_t)-g(\gamma_0)}t\,\d\ppi(\gamma)=\int\la\nabla f,\nabla g\ra(\gamma_0)\,\d\ppi(\gamma).
\end{equation}
The space $D(\Delta)\subset W^{1,2}(\X)$ is the space of functions $f$ for which there is a function in $L^2(\X)$, called the Laplacian of $f$ and denoted by $\Delta f$, such that
\[
\int g\Delta f\,\d\mm=-\int\la\nabla f,\nabla g\ra\,\d\mm\qquad\forall g\in W^{1,2}(\X).
\]
The Laplacian is local in the sense that for every $f,g\in D(\Delta)$ we have
\[
\Delta f=\Delta g,\qquad\mm-a.e.\ \text{ on the interior of }\{f=g\},
\]
and satisfies the natural chain and Leibniz rules.

Choosing to test with Lipschitz functions with prescribed support yields the notion of measure valued Laplacian: for $f\in W^{1,2}(\X)$ and $\Omega\subset \X$ open, we say that $f$ has a measure valued Laplacian in $\Omega$ provided there is a measure, denoted by $\bd f\restr{\Omega}$, such that
\[
\int g\,\d\bd  f\restr\Omega=-\int\la\nabla f,\nabla g\ra\,\d\mm\qquad\forall g\text{ bounded and Lipschitz with compact support in }\Omega.
\]
If $\Omega=\X$ we shall simply write $\bd f$ in place of $\bd f\restr \X$. The `measure-valued' and `$L^2$-valued' notions of Laplacian are tightly linked, as for a given $f\in W^{1,2}(\X)$ we have $f\in D(\Delta)$ if and only if $f$ has a measure-valued Laplacian absolutely continuous w.r.t.\ $\mm$ and with density in $L^2$. In this case $\Delta f$ coincides with such density.

The \emph{heat flow} on an infinitesimally Hilbertian space is the $L^2$-gradient flow of the Dirichlet energy $\mathcal E:L^2(\X)\to[0,\infty]$ defined as
\[
\mathcal E(f):=\left\{
\begin{array}{ll}
\displaystyle{\frac12\int|\D f|^2\,\d\mm},&\qquad\text{ if }f\in W^{1,2}(\X),\\
+\infty,&\qquad\text{ otherwise}.
\end{array}
\right.
\]
It will be denoted by $\h_t:L^2(\X)\to L^2(\X)$ so that for $f\in L^2(\X)$ the curve $t\mapsto \h_tf\in L^2(\X)$ is the unique gradient flow trajectory on $L^2(\X)$ of $\mathcal E$ starting from $f$. The fact that $\mathcal E$ is a quadratic form ensures that the heat flow is linear.

The functional $\mathcal U_N:\prob\X\to\R$ is defined as
\[
\mathcal U_N(\mu):=-\int\rho^{1-\frac1N}\,\d\mm,\qquad\text{for }\quad \mu=\rho\mm+\mu^s,\quad \mu^s\perp\mm.
\]
$(\X,\sfd,\mm)$ is said to be a $\CD^*(0,N)$ space provided $\mathcal U_N$ is geodesically convex on the space $(\probt{\supp(\mm)},W_2)$ (see \cite{BacherSturm10} and the original papers \cite{Lott-Villani09}, \cite{Sturm06II}). If $(\X,\sfd,\mm)$ is $\CD^*(0,N)$, then $(\supp(\mm),\sfd)$ is proper and geodesic (\cite{Sturm06II}).

A space which is both $\CD^*(0,N)$ and infinitesimally Hilbertian will be called $\RCD^*(0,N)$ (see \cite{Gigli12} and \cite{AmbrosioGigliSavare11-2}). We recall that a metric measure space $(\X,\sfd,\mm)$ has the \emph{Sobolev-to-Lipschitz} property (see \cite{Gigli13}) provided any $f\in W^{1,2}(\X)$ with $|\D f|\leq 1$ $\mm$-a.e., admits a 1-Lipschitz representative. It is an important fact about $\RCD^*(0,N)$ spaces that they have the Sobolev-to-Lipschitz property (see  \cite{AmbrosioGigliSavare11-2}).

On a $\RCD^*(0,N)$ space we consider the following space of \emph{test functions} (see \cite{Savare13} and \cite{Gigli14}):
\[
\fsm \X:=\Big\{f\in  D(\Delta)\ :\ f,|\D f|\in L^\infty(\X),\ \Delta f\in W^{1,2}(\X) \Big\}.
\]
In particular, test functions have a Lipschitz representative and we shall always consider such representative when working with them. By simple truncation and mollification via the heat flow we see that $\fsm\X$ is dense in $W^{1,2}(\X)$ and it easy to check that it is stable by application of the heat flow. A slightly more refined argument grants that 
\begin{equation}
\label{eq:densetest}
\begin{split}
&\text{if $f\in W^{1,2}(\X)$ (resp. $L^2(\X)$) has support on a given open set $\Omega$, then there exists}\\
&\text{a sequence of test functions with support in $\Omega$ converging to $f$ in $W^{1,2}(\X)$ (resp. $L^2(\X)$).}
\end{split}
\end{equation}

It is a remarkable property of test functions $f$ the fact that $|\D f|^2\in W^{1,2}(\X)$ (see \cite{Savare13}) and this fact  grants that $\fsm\X$ is an algebra.

Finally, on $\RCD^*(0,N)$ spaces the \emph{Bochner inequality} holds (\cite{Erbar-Kuwada-Sturm13}) in the sense that for $f\in\fsm\X$ the function $|\D f|^2$ has a measure valued Laplacian and
\begin{equation}
\label{eq:Bochner}
\bd\frac{|\D f|^2}{2}\geq \Big(\frac{(\Delta f)^2}{N}+\la\nabla f,\nabla\Delta f \ra\Big)\mm.
\end{equation}

\section{Main}
Throughout all the paper we shall make the following assumption:
\begin{assumption}\label{ass:base} $(\X,\sfd,\mm)$ is a $\RCD^*(0,N)$ space with $\supp(\mm)=\X$, $\O\in \X$ is a given point and $\Rb>\rb>0$ are radii such that
\[
\mm(B_{\Rb}(\O))=\Big(\frac{\Rb}{\rb}\Big)^N\mm(B_{\rb}(\O)).
\]
\end{assumption}
{
The proof of Theorem \ref{thm:mainthm} will be based in the study of the {\em gradient flow} of the  ``Busemann'' function $\b:\X\to\R^+$ given by
\begin{equation*}\label{eq:bus}
\b(x):=\frac{\sfd^2(x,\O)}{2},\qquad\forall x\in \X.
\end{equation*}
}
\subsection{Gradient flow of $\b$: effect on the measure}\label{se:gf}
We start recalling the following basic fact about geodesics on $\RCD^*(0,N)$ spaces, which directly follows from the existence of optimal maps established in \cite{GigliRajalaSturm13} when one of the two  measures considered  is $\delta_{\O}$:
\begin{proposition}
There exists a Borel $\mm$-negligible set $\mathcal N\subset \X$ and a Borel map $\fl:[0,1]\times \X\to \X$ such that for every $x\in \X\setminus\mathcal N$ the curve $[0,1]\ni t\mapsto\fl_t(x)$ is the unique constant speed geodesic from $\O$ to $x$. Moreover,  for every $t\in(0,1]$ the map $\fl_t:\X\setminus\mathcal N\to \X$ is injective and the measure $(\fl_t)_*\mm$ is absolutely continuous w.r.t.\ $\mm$ and its density $\rho_t$ satisfies
\begin{equation}
\label{eq:densinter}
\rho_t\circ\fl_t \leq t^{-N},\qquad\mm-a.e..
\end{equation}
\end{proposition}
Coupling this proposition with   Assumption \ref{ass:base} we get the following rigidity result:
\begin{proposition}\label{prop:mpf}
For every $t\in(0,1]$ we have 
\begin{equation}
\label{eq:mispalle}
\mm(B_{t\Rb}(\O))=t^N\mm(B_\Rb(\O)),\qquad\forall t\in[0,1].
\end{equation}
and
\begin{equation}
\label{eq:misura}
(\fl_t)_*(\mm\restr{B_{\Rb}(\O)})=\frac{1}{t^N}\mm\restr{B_{t\Rb }(\O)}.
\end{equation}
\end{proposition}
\begin{proof} We start with \eqref{eq:mispalle}. Set $v(s):=\mm(B_{s}(\O))$ and  notice that the Bishop-Gromov inequality (see Theorem 2.3 in \cite{Sturm06II}) ensures that $v$ is continuous, locally semiconcave on $[0,\Rb]$ and that the map $\frac{v'(s)}{s^{N-1}}$ is decreasing. It follows that the map $g(s):=v(s^{\frac1N})$ is locally semiconcave on $(0,\Rb^{N}]$, continuous on $[0,\Rb^N]$ and from
\[
g'(s)=v'(s^{\frac1N})\frac1Ns^{\frac1N-1}=\frac1N\frac{v'(s^{\frac1N})}{(s^{\frac1N})^{N-1}},\qquad a.e.\ s\in[0,\Rb^{N}],
\]
we see that $g$ has decreasing derivative, i.e.\ that $g$ is concave and in particular
\[
g(t\Rb^N)\geq t  g(\Rb^{N})+(1-t)g(0),\qquad\forall t\in[0,1].
\]
Since $g(0)=0$, we see that Assumption \ref{ass:base} gives that the inequality is an equality for $t:=\frac{\rb^N}{\Rb^N}$. Then the concavity of $g$ forces the equality for every $t\in[0,1]$, which is \eqref{eq:mispalle}.

To conclude, let $\mu_0:=\mm(B_\Rb(\O))^{-1}\mm\restr{B_\Rb(\O)}$ and put $\mu_t:=(\fl_{1-t})_*\mu_0$ so that $t\mapsto\mu_t$ is the only $W_2$-geodesic connecting $\mu_0$ to $\mu_1=\delta_{\O}$. Put also  $\nu_t:=\mm(B_{(1-t)\Rb}(\O))^{-1}\mm\restr{B_{(1-t)\Rb}(\O)}$, notice that  $\mu_t$ is concentrated on $B_{(1-t)\Rb}(\O)$ and thus for every $t\in[0,1)$
\begin{align*}
-(1-t)\mm(B_\Rb(\O))^{\frac1N}&=(1-t)\mathcal U_N(\mu_0)&&\text{ by computation}\\
&\geq \mathcal U_N(\mu_t)&&\text{ from $\CD^*(0,N)$   and  $\mathcal U_N(\delta_{\O})=0$}\\
&\geq \mathcal U_N(\nu_t)&&\text{ by Jensen's inequality}\\
&=-\mm(B_{(1-t)\Rb}(\O))^{\frac1N}&&\text{ by computation}\\
&=-(1-t)\mm(B_\Rb(\O))^{\frac1N}&&\text{ by \eqref{eq:mispalle}}.
\end{align*}
Therefore we have the equality in Jensen's inequality, which forces $\mu_t=\nu_t$  for every $t\in[0,1)$, which is the claim.
\end{proof}
{
We now introduce the 
the reparametrized flow $\gl:[0,\infty)\times (\X\setminus\mathcal N)\to \X$ defined by
\[
\gl_s(x):=\fl_{e^{-s}}(x),\qquad\forall s\geq 0,\ x\in \X\setminus\mathcal N.
\]
Notice that in the smooth setting, the flow $\gl$ would be the gradient flow of $\b$. In our context, some basic properties of $\gl$ follow from those of $\fl$:
}
\begin{corollary}
\label{cor:gl}
The following holds:
\begin{itemize}
\item[i)] For every $x\in \X\setminus \mathcal N$ the curve $[0,\infty)\ni t\mapsto\gamma_t:=\gl_t(x)$ is locally Lipschitz and satisfies
\[
\b(\gamma_t)=\b(\gamma_s)+\frac12\int_t^s|\dot\gamma_r|^2+\lip(\b)^2(\gamma_r)\,\d r,\qquad\forall 0\leq t\leq s,
\]
and in particular the metric speed of $t\mapsto \gl_t(x)$ is equal to $\lip(\b)(\gl_t(x))\leq \lip(\b)(x)$.
\item[ii)] For every $t\geq 0$, $\gl_t$ is an essentially invertible map from $B_{\Rb}(\O)$ to $B_{e^{-t}\Rb}(\O)$, i.e.\ there exists a map $\gl_t^{-1}:B_{e^{-t}\Rb}(\O)\to B_{\Rb}(\O)$ such that $\gl_t\circ\gl_t^{-1}=\Id$ $\mm$-a.e.\ on $B_{e^{-t}\Rb}(\O)$ and $\gl_t^{-1}\circ\gl_t=\Id$ $\mm$-a.e.\ on $B_{\Rb}(\O)$.
\item[iii)] For every $t\geq 0$ we have 
\begin{equation}
\label{eq:pfent1}
(\gl_t)_*\mm\leq e^{Nt}\mm
\end{equation}
and
\begin{equation}
\label{eq:pfent}
\frac{\d(\gl_t)_*\mm}{\d\mm}\circ\gl_t=e^{Nt},\qquad\mm-a.e.\ on\ B_\Rb(\O).
\end{equation}
\item[iv)] For every $t,s\geq 0$ we have
\begin{equation}
\label{eq:distgl}
\begin{split}
\gl_t(\gl_s(x))&=\gl_{t+s}(x),\\
\sfd(\gl_s(x),\gl_t(x))&=\sfd(x,\O)|e^{-s}-e^{-t}|,
\end{split}
\end{equation}
for $\mm$-a.e.\ $x\in \X$.
\end{itemize}
\end{corollary}
\begin{proof} The triangle inequality shows that $\lip(\b)(x)\leq \sfd(x,\O)$ and an interpolation with a geodesic that equality actually holds.
Point $(i)$ follows by direct computation. Concerning point $(ii)$, notice that the essential injectivity of $\gl_t$ follows from the one of $\fl_{e^{-t}}$ and the essential surjectivity is a consequence of \eqref{eq:misura}. The bound \eqref{eq:pfent1} follows from \eqref{eq:densinter} and \eqref{eq:pfent} is a restatement of \eqref{eq:misura}. Finally, for property $(iv)$ recall that $t\mapsto\fl_t(x)$ is a constant speed geodesic from $\O$ to $x$ for $\mm$-a.e.\ $x$ and take into account the reparametrization.
\end{proof}
Few basic, yet interesting, properties of $\b$ and $\gl_t$ can now be established.
\begin{corollary}\label{cor:rappresenta}
Let $T:(\X\setminus\mathcal N)\to C([0,1],\X)$ be the map sending $x$ to the curve $[0,1]\ni t\mapsto \gl_t(x)$ and $\mu\in\prob \X$ with bounded support and  such that $\mu\leq C\mm$ for some $C$. Then $\ppi:=T_*\mu$ represents the gradient of $-\b$.
\end{corollary}
\begin{proof}
The bound \eqref{eq:pfent1} ensures that $\ppi$ has bounded compression, while from point $(i)$ of Corollary \ref{cor:gl} we know that the metric speed of $T(x)$ is bounded by $\lip(\b)(x)=\sfd(x,\O)$, so that the fact that $\mu$ has bounded support ensures that $\ppi$ has finite kinetic energy.

Thus $\ppi$ is a test plan. Moreover, from $(i)$ of Corollary \ref{cor:gl} we have
\[
\begin{split}
\int \b(\gamma_0)-\b(\gamma_t)\,\d\ppi(\gamma)&=\frac12\iint_0^t\lip^2(\b)(\gamma_s)+|\dot\gamma_s|^2\,\d s\,\d\ppi(\gamma)\\
&\geq\frac12\iint_0^t|\D \b|^2(\gamma_s)+|\dot\gamma_s|^2\,\d s\,\d\ppi(\gamma).
\end{split}
\]
Dividing by $t$, letting $t\downarrow0$, noticing that the measures $(\e_s)_*\ppi$ have uniformly bounded support and that $|\D\b|\leq\lip(\b)$ is bounded on bounded sets, to conclude it is sufficient to prove that  $(\e_s)_*\ppi\weakto\mu$ as $s\downarrow0$ in duality with $L^1(\X)$. This follows from  $W_2$-convergence - which grants weak convergence - and the fact that these measures have uniformly bounded densities.
\end{proof}

\begin{corollary}\label{cor:db}
$|\D \b|^2=2\b$ $\mm$-a.e..
\end{corollary}
\begin{proof} Let $T:(\X\setminus\mathcal N)\to C([0,1],\X)$ be defined as in Corollary \ref{cor:rappresenta} and $\mu\in\prob \X$ be with bounded support such that $\mu\leq C\mm$ for some $C$. Then Corollary  \ref{cor:rappresenta} grants that $\ppi$ is a test plan and therefore, keeping in mind point $(i)$ of Corollary \ref{cor:gl}, we have
\[
\begin{split}
\frac12\iint_0^t|\dot\gamma_s|^2+\lip(\b)^2(\gamma_s)\,\d s\,\d\ppi(\gamma)&= \int\b(\gamma_0)-\b(\gamma_t)\,\d\ppi(\gamma)\\
&\leq\iint_0^t|\D \b|(\gamma_s)|\dot\gamma_s|\,\d s\,\d\ppi(\gamma)\\
&\leq \frac12\iint_0^t|\dot\gamma_s|^2+|\D \b|^2(\gamma_s)\,\d s\,\d\ppi(\gamma)
\end{split}
\]
which gives 
\[
\iint_0^t\lip(\b)^2(\gamma_s)\,\d s\,\d\ppi(\gamma)\leq \iint_0^t|\D \b|^2(\gamma_s)\,\d s\,\d\ppi(\gamma).
\]
 Dividing by $t$, letting $t\downarrow0$ and using the fact that $(\e_t)_*\ppi\weakto\mu$ as $t\downarrow0$ in duality with $L^1(\X)$ (like in Corollary \ref{cor:rappresenta}) we deduce that $\int\lip(\b)^2\,\d\mu\leq\int |\D \b|^2\,\d\mu$ which by the arbitrariness of $\mu$ and inequality \eqref{eq:lipwug} is sufficient to conclude, since \(\lip(\b)^2(x)=2\b(x)\).
\end{proof}

{
Finally we show as equality in the Bishop-Gromov inequality combined with the Laplacian comparison Theorem proved in \cite{Gigli12} allows to show that \(\Delta \b=N\) on $B_\Rb(\O)$:
\begin{proposition}\label{prop:laplB}
We have $\bd\b \restr{B_{\Rb}(\O)}=N\mm\restr{B_{\Rb}(\O)}$.
\end{proposition}
\begin{proof}
Let $\varphi:\R^+\to\R^+$ be smooth, non-increasing and with support in $[0,1)$  and let us consider 
\[
h_\varphi(r):=\frac{1}{r^N}\int \varphi\Big(\frac{2\b}{r^2}\Big)\,\d\mm.
\]
Differentiating with respect to \(r\), taking into account that \(|\D \b|^2=2\b\) and  that   \(\bd\b \le N\mm\), see \cite{Gigli12}, we obtain
\begin{equation}\label{1}
\begin{split}
r^{N+1}h_\varphi'(r)&=-N\int \varphi\Big(\frac{2\b}{r^2}\Big)\,\d\mm-4\int \varphi'\Big(\frac{2\b}{r^2}\Big)\frac{b}{r^2}\,\d\mm
\\
&= -N\int \varphi\Big(\frac{2\b}{r^2}\Big)\,\d\mm-2\int \varphi'\Big(\frac{2\b}{r^2}\Big)\frac{|\D \b|^2}{r^2}\,\d\mm
\\
&=-N\int \varphi\Big(\frac{2\b}{r^2}\Big)\,\d\mm+\int \varphi\Big(\frac{2\b}{r^2}\Big)\,\d \bd\b \le 0.
\end{split}
\end{equation}
So that \(h_\varphi\) is also non-increasing. On the other hand by \eqref{eq:misura} and the layer cake formula,  we see that for \(r\le \Rb\) and \(\varphi\) decreasing we have
\[
\begin{split}
h_\varphi(r)=\frac{1}{r^N}\int \varphi\Big(\frac{2\b}{r^2}\Big)\,\d\mm&=\frac{1}{r^N}\int_0^1\mm\Big(\Big\{\varphi\Big(\frac{2\b}{r^2}\Big)> c\Big\}\Big)\,\d c\\
&=\frac{1}{r^N}\int_0^r\mm\Big(\Big\{\varphi\Big(\frac{2\b}{r^2}\Big)> \varphi\Big(\frac{s^2}{r^2}\Big)\Big\}\Big)|\varphi'|\Big(\frac{s^2}{r^2}\Big)\frac{2s}{r^2}\,\d s\\
\text{(because $\varphi$ is non-increasing)}\qquad&=\frac{1}{r^N}\int_0^r\mm\big(B_{s}(\O)\big)|\varphi'|\Big(\frac{s^2}{r^2}\Big)\frac{2s}{r^2}\,\d s\\
\text{(by \eqref{eq:mispalle})}\qquad&=\frac{\mm(B_{\Rb}(\O))}{\Rb^N}\int_0^r\Big(\frac{s}{r}\Big)^N|\varphi'|\Big(\frac{s^2}{r^2}\Big)\frac{2s}{r^2}\,\d s\\
&=\frac{\mm(B_{\Rb}(\O))}{\Rb^N}\int_0^1t^N|\varphi'|(t^2){2t}\,\d t.
\end{split}
\]  

Hence \(h_\varphi\) is constant on \([0,R]\). Combining this with \eqref{1} we immediately deduce that for every  \(r\le \Rb\) we have
\[
\int \varphi\Big(\frac{2\b}{r^2}\Big)\,\d \bd\b=N\int \varphi\Big(\frac{2\b}{r^2}\Big).
\]
By  letting \(\varphi\to \chi_{[0,1]}\) and recalling that  \(\bd\b \le N\mm\), we obtain  $\bd\b \restr{B_{\Rb}(\O)}=N\mm\restr{B_{\Rb}(\O)}$, as desired.
\end{proof}
}

In the following corollary as well as in the foregoing discussion we shall denote by $\sfd_{\O}:\X\to\R^+$ the map sending $x$ to $\sfd(x,\O)$.
\begin{corollary}[Continuous disintegration of $\mm\restr{B_{\Rb}(\O)}$ along $\sfd_{\O}$]\label{cor:contdis}
We have 
\begin{equation}
\label{eq:pfdo}
\d (\sfd_{\O})_*(\mm\restr{B_\Rb(\O)})(r)=c\,\nchi_{[0,\Rb]}(r)\,r^{N-1}\d r,
\end{equation}
with $c:=N \mm(B_{\Rb}(\O))$ and there exists a weakly continuous family of measures $[0,\Rb]\ni r\mapsto \mm_r\in \prob{\X}$ such that
\begin{equation}
\label{eq:dismdo}
\int\varphi \,\d\mm=c\int_0^\Rb\int\varphi\,\d\mm_r\,r^{N-1}\,\d r,\qquad\forall \varphi\in C_c(B_{\Rb}(\O)).
\end{equation}
Moreover, for every $t\geq 0$ the measures $\mm_r$  satisfies 
\begin{equation}
\label{eq:glmeas}
(\gl_t)_*\mm_r=\mm_{e^{-t}r},\qquad a.e.\ r\in[0,\Rb].
\end{equation}
\end{corollary}
\begin{proof} The identity  \eqref{eq:pfdo} follows from \eqref{eq:misura}, which shows  that $\mm(B_r(\O))=(\frac r\Rb)^N\mm(B_\Rb(\O))$ for every $r\in[0,\Rb]$.

 Now let $\{\mm_r\}_{r\in[0,\Rb]}\subset\prob \X$ be a disintegration of $\mm\restr{B_\Rb(\O)}$ along $\sfd_{\O}$, hence such that \eqref{eq:dismdo} holds, and recall that a priori  existence and uniqueness of the $\mm_r$'s is only given for a.e.\ $r\in[0,\Rb]$. Fix  $t\geq 0$,  $\varphi\in C_c(B_{e^{-t}\Rb}(\O))$, and notice that
\[
\begin{split}
\int\varphi\,\d(\gl_t)_*\mm=\int\varphi\circ\gl_t\,\d\mm=N \mm(B_{\Rb}(\O))\int_0^\Rb\int\varphi\circ\gl_t\,\d\mm_r \, r^{N-1}\,\d r.
\end{split}
\]
On the other hand, since $\gl_t(B_{\Rb}(\O))\subset B_{e^{-t}\Rb}(\O)$, by \eqref{eq:pfent} we also have
\[
\begin{split}
\int\varphi\,\d(\gl_t)_*\mm= e^{Nt}\int\varphi\,\d\mm=e^{Nt}N \mm(B_{\Rb}(\O))\int_0^\Rb\int\varphi \,\d\mm_s \, s^{N-1}\,\d s.
\end{split}
\]
Thus the change of variable $s=e^{-t}r$ in the last integral shows that
\[
\int_0^\Rb\int\varphi\circ\gl_t\,\d\mm_r \, r^{N-1}\,\d r=\int_0^\Rb\int\varphi\,\d\mm_{e^{-t}r} \, r^{N-1}\,\d r,
\]
which by the arbitrariness of $\varphi\in  C_c(B_{e^{-t}\Rb}(\O))$ gives the claim \eqref{eq:glmeas}. 

It remains to prove that the $\mm_r$'s can be chosen to weakly depend on $r\in[0,\Rb]$ and to this aim, due to the existence of a countable set of Lipschitz functions dense in $C_c(B_{\Rb}(\O))$, it is sufficient to show that for a given  $\varphi:\X\to\R$  Lipschitz with support in $B_{\Rb}(\O)$, the map $r\mapsto I_\varphi(r):=\int\varphi\,\d\mm_r$ admits a continuous representative.

Thus fix such $\varphi$, put $J_\varphi(s):=I_\varphi(\Rb e^{-s})$ for $s\in\R^+$ and notice that \eqref{eq:glmeas} grants that for every $h\geq 0$ the identity 
\[
\begin{split}
\big|J_\varphi(s+h)-J_\varphi(s)\big|&=\Big|\int\varphi\,\d\mm_{\Rb e^{-h}e^{-s}}-\int\varphi\,\d\mm_{\Rb e^{-s}}\Big|\\
&=\Big|\int\varphi\circ\gl_h-\varphi \,\d\mm_{\Rb e^{-s}}\Big|\leq\Rb e^{-s}\Lip(\varphi)(1-e^{-h})
\end{split}
\]
holds for a.e.\ $s$, the inequality being a consequence of the second identity in \eqref{eq:distgl} and the fact that $\mm_{\Rb e^{-s}}$ is concentrated on $B_{\Rb e^{-s}}(\O)$. This is sufficient to show that the distributional derivative of $s\mapsto J_\varphi(s)$ is bounded by $\Rb\Lip(\varphi)e^{-s}$. Being $s\mapsto \Rb\Lip(\varphi) e^{-s}$ in $L^1(\R^+)$, we just proved that  $J_\varphi$ has an absolutely continuous representative admitting a limit when $s\to+\infty$. By construction, this is the same as to say that  $I_\varphi$ has a continuous representative on $[0,\Rb]$.
\end{proof}

For the purpose of the foregoing analysis it will be convenient to replace the function $\b$ with a smoother one. We therefore fix once and for all $\Rbt<\Rb$.  Later on $\Rbt$ will be sent to $\Rb$ but for the moment it is convenient to think it as fixed also in order to avoid mentioning the dependence on it of the various objects we are going to build.

Let $\varphi\in C^\infty(\R)$ be  a function with support contained in $(-\infty,\tfrac{\Rb^2}2)$ which is the identity on $(-\infty,\tfrac{\Rbt^2}2)$ and define $\bar\b:\X\to\R$ as
\[
\bar\b:=\varphi\circ\b.
\]
Notice  in particular that $\bar\b$ is Lipschitz, with support in $B_\Rb(\O)$ and equal to $\b$ on $B_{\Rbt}(\O)$. We then introduce the flow $\bar\gl$ as follows. First define the reparametrization function ${\rm rep}:(\R^+)^2\to\R+$ by requiring that 
\begin{equation}
\label{eq:defrep}
\partial_t{\rm rep}_t(r)=\varphi'\Big(\frac{r^2}{2}e^{-2{\rm rep}_t(r)}\Big),\qquad\qquad {\rm rep}_0(r)=0,
\end{equation}
for every $r\geq 0$, then we define $\bar\gl:[0,\infty)\times(\X\setminus\mathcal N)\to \X$ as
\[
\bar\gl_t(x):=\gl_{{\rm rep}_t(\sfd(x,\O))}(x).
\]
{The following proposition collects the basic properties of \(\bar\gl_t(x)\).
\begin{proposition}\label{prop:modifiedflow}
Let \(\bar\b\) and \(\bar\gl_t(x)\) as above, then:
\begin{itemize}
\item[a)] for every $x\in \X\setminus\mathcal N$ the curve $t\mapsto\gamma_t:=\bar\gl_t(x)$ satisfies
\[
\bar\b(\gamma_t)=\bar \b(\gamma_s)+\frac12\int_t^s|\dot\gamma_r|^2+\lip(\bar\b)^2(\gamma_r)\,\d r,\qquad\forall 0\leq t\leq s.
\]
In particular, the speed of $t\mapsto \bar\gl_t(x)$ is equal to $\lip(\bar\b)(\bar\gl_t(x))$ for a.e.\ $t$, thus granting that $t\mapsto\bar\gl_t(x)$ is $\Lip(\bar\b)$-Lipschitz for every $x\in \X\setminus\mathcal N$.
\item[b)]  $\bar\gl_t$ is the identity on $\X\setminus(B_{\Rb}(\O)\cup\mathcal N)$ and sends $B_{\Rb}(\O)\setminus\mathcal N$ into $B_{\Rb}(\O)$ for every $t\geq 0$.
\item[c)] $\bar\gl_t$ coincides with $\gl_t$ in $B_{\Rbt}(\O)\setminus \mathcal N$.
\item[d)] $\bar\gl_t:\X\to \X$ is essentially invertible for every $t\geq 0$.
\item[e)] $(\bar\gl_t)_*\mm\ll\mm$ for every $t\geq 0$, more precisely 
\begin{equation}
\label{eq:bc}
c(t)\mm\leq (\bar\gl_t)_*\mm\leq C(t)\mm,\qquad\forall t\in\R,
\end{equation}
for some  continuous functions $c,C:\R^+\to(0,+\infty)$.
\item[f)] The maps $\bar\gl_t$ form a semigroup, i.e.\ $\bar\gl_0=\Id$ $\mm$-a.e.\ and 
\begin{equation}
\label{eq:group}
\bar\gl_t\circ\bar\gl_s=\bar\gl_{t+s},\quad\mm-a.e.
\end{equation}
for every $t,s\geq 0$.
\end{itemize}
\end{proposition}
In particular, from the essential invertibility of $\bar\gl_t:\X\to \X$ we see that for $t>0$ the map $\bar\gl_{-t}:=(\bar\gl_t)^{-1}$ is well defined $\mm$-a.e.\ and with this definition property \eqref{eq:group} holds for every $t,s\in\R$.
\begin{proof}[Proof of Proposition \ref{prop:modifiedflow}] Points $(b)$, $(c)$, $(f)$ follows from the definition and  $(a)$ can be obtained by the same argument of \ref{cor:gl}. The injectivity part of point $(d)$  is a consequence of the same property for \(\gl_t\), while the surjectivity follows once we prove $(e)$. 

 We are thus left to prove $(e)$. Start observing that for $r\geq\Rb$ we have  ${\rm rep}_t(r)=0$ for any $t\geq 0$ and thus $\bar\gl_t\restr{X\setminus B_\Rb(\O)}=\Id$ so that
 \begin{equation}
\label{eq:partefuori}
(\bar\gl_t)_*\mm\restr{X\setminus B_\Rb(\O)}=\mm\restr{X\setminus B_\Rb(\O)},\qquad\forall t\geq 0.
\end{equation}
 To control the behaviour inside $B_\Rb(\O)$ notice that by the very definition of $\bar\gl_t$ and from \eqref{eq:glmeas} we have that for any $t\geq 0$ it holds
 \[
(\bar\gl_t)_* ( \mm_r)=\mm_{e^{-{\rm rep}_t(r)}r},\qquad a.e.\ r\in[0,\Rb].
\]
Put $f_t(r):=e^{-{\rm rep}_t(r)}r$. We claim that there are continuous functions $c,C:\R^+\to(0,+\infty)$ such that
\begin{equation}
\label{eq:claimft}
c(t)\leq \partial_rf_t(r)\leq C(t),\qquad\forall r\in[0,\Rb],\ t\geq 0.
\end{equation}
The bound from above is obvious by smoothness, for the one from below observe that 
\begin{equation}
\label{eq:derf}
\partial_r f_t(r)=e^{-{\rm rep}_t(r)}(1-r\partial_r{\rm rep}_t(r)),
\end{equation} 
put $g_t(r):=r\partial_r{\rm rep}_t(r)$ and differentiate \eqref{eq:defrep} in $r$ to deduce that $g_t(r)$ solves
\begin{equation}
\label{eq:perconfronto}
\partial_tg_t(r)=\varphi''\Big(\tfrac{r^2}{2}e^{-2{\rm rep}_t(r)}\Big)r^2e^{-2{\rm rep}_t(r)}\big(1-g_t(r)\big)
\end{equation}
with the initial condition $g_0(r)=0$ for every $r\in[0,\Rb]$. Noticing that the function identically  1  solution of \eqref{eq:perconfronto}, taking into account the initial condition by comparison we deduce that
\[
 g_t(r)\leq 1-\tilde c(t),\qquad\forall r\in[0,\Rb],\ t\geq 0,
\]
for some continuous function $\tilde c:\R^+\to(0,\infty)$. Plugging this bound in \eqref{eq:derf} we deduce the first inequality in \eqref{eq:claimft}.

Then for every $t\geq0$ consider the function $f_t^{-1}:[0,\Rb]\to [0,\Rb]$, notice that $f_t^{-1}(0)=0$, $f_t^{-1}(\Rb)=\Rb$ and that \eqref{eq:claimft} grants that
\begin{equation}
\label{eq:finally}
\frac1{C(t)}\leq\partial_s(f_t^{-1})(s)\leq \frac1{c(t)},\qquad\forall s\in[0,\Rb],\ t\geq 0
\end{equation}
and thus it also holds
\begin{equation}
\label{eq:finally2}
\frac s{C(t)}\leq f_t^{-1}(s)\leq\frac{s}{c(t)}
\end{equation}
Let now \(\varphi\ge 0\) be a Borel function identically  zero in  in  \(X\setminus B_{\Rb}(\O)\), according to Corollary \ref{cor:contdis} we have 
\[
\begin{split}
\int \varphi\, \d (\bar\gl_t )_*\mm&=\int \varphi\circ \bar\gl_t \, \d\mm=c\int_0^{\Rb} \int \varphi\circ \bar\gl_t \,\d\mm_r r^{N-1} \d r\\
&=c\int_0^{\Rb} \int \varphi \,\d(\bar\gl_t)_* ( \mm_r)\, r^{N-1} \d r= c\int_0^{\Rb} \int \varphi \,\d \mm_{f_t(r)}\, r^{N-1} \d r\\
&=c\int_0^{\Rb} \int \varphi \,\d \mm_{s}\, (f_t^{-1}(s))^{N-1}\partial_sf_t^{-1}(s) \d s,
\end{split}
\]
thus from the bounds \eqref{eq:finally} and \eqref{eq:finally2} and using again Corollary \ref{cor:contdis} we obtain
\[
 \frac{c}{C(t)^N}\int \varphi\,\d\mm\leq \int \varphi\, \d (\bar\gl_t )_*\mm\leq \frac{c}{c(t)^N}\int \varphi\,\d\mm,\qquad\forall t\geq 0,
\]
which together with \eqref{eq:partefuori} gives the claim \eqref{eq:bc}.
\end{proof}
}

{
Proposition \ref{prop:laplB} easily implies the following useful corllary:

\begin{corollary}\label{cor:btest}
$\bar\b\in\fsm \X$ and $\Delta\bar\b\in L^\infty(\X)$.
\end{corollary}
\begin{proof}
%
Clearly $\supp(\bar\b)\subset B_\Rb(\O)$, hence the identity $|\D \b|^2=2\b$ $\mm$-a.e.\ given by Corollary \ref{cor:db},   Proposition \ref{prop:laplB} and  the chain rule for the distributional Laplacian (see Proposition 4.11 in \cite{Gigli12})  yields
\[
\begin{split}
\bd\bar \b&=\varphi''\circ\b|\D \b|^2\mm+\varphi'\circ\b\bd\b=\big(2\b\varphi''\circ\b +N\varphi'\circ\b\big)\mm,
\end{split}
\]
Given that $2\b\varphi''\circ\b +N\varphi'\circ\b\in L^\infty\cap W^{1,2}(\X)$, this identity shows that $\bar\b\in D(\Delta)$ with $\Delta\bar\b\in L^\infty\cap W^{1,2}(\X)$, yielding the desired conclusion.
\end{proof}
}
We conclude the section with the following useful lemma, which highly depends on the first-order differentiation formula \eqref{eq:firstorder}.
\begin{lemma}[Basic properties of right composition with $\bar\gl_t$]\label{le:compgl}
Let $f\in L^p(\X)$, $p<\infty$. Then $t\mapsto f\circ\bar\gl_t\in L^p(\X)$ is continuous. 

If $f\in W^{1,2}(\X)$, then $t\mapsto f\circ\bar\gl_t\in L^2(\X)$ is $C^1$ and its derivative is given by
\begin{equation}
\label{eq:calcoloder}
\frac{\d}{\d t}f\circ\bar\gl_t=-\la\nabla f,\nabla\bar \b\ra\circ\bar\gl_t.
\end{equation}
\end{lemma}
\begin{proof} Property \eqref{eq:bc}  grants that the linear operators from $L^p(\X)$ into itself given by the right composition with $\bar\gl_t$ are, locally in $t$, uniformly continuous and thus it is sufficient to check that $t\mapsto f\circ\bar\gl_t\in L^p(\X)$ is continuous for a dense set of $f$'s. We then consider Lipschitz functions and notice that  the uniform Lipschitz bound granted by point $(a)$ of Proposition \ref{prop:modifiedflow} gives
\[
\sqrt[p]{\int|f\circ\bar\gl_{t_1}-f\circ\bar\gl_{t_0}|^p\,\d\mm}\leq\Lip(f)\sqrt[p]{\int\sfd^p(\bar\gl_{t_1}(x),\bar\gl_{t_0}(x))\,\d\mm(x)}\leq \Lip(f)\Lip(\bar\b)|t_1-t_0|,
\]
thus yielding the first claim.

Now let $\bar\mm\in\prob{\X}$ be such that $\mm\ll\bar\mm\leq C\mm$ for some $C>0$, let $T:\X\to C([0,1],\X)$ be the map sending $x$ to the curve $t\mapsto\bar\gl_t(x)$ and define $\ppi:=T_*\bar\mm$. The same arguments used for Corollary \ref{cor:rappresenta} grant that  $\ppi$ is a test plan. Thus  Proposition \ref{prop:sobf} grants that for any $t_0,t_1\in[0,1]$, $t_0<t_1$, for $\mm$-a.e.\ $x$  we have
\[
|f(\bar\gl_{t_1}(x))-f(\bar\gl_{t_0}(x))|\leq\Lip(\bar\b)\int_{t_0}^{t_1}|\D f|(\bar\gl_t(x))\,\d t.
\]
Squaring and integrating we get
\[
\begin{split}
\int |f\circ\bar\gl_{t_1}-f\circ\bar\gl_{t_0}|^2\,\d\mm&\leq |t_1-t_0|\Lip^2(\bar\b)\iint_{t_0}^{t_1}|\D f|^2(\bar\gl_t(x))\,\d t\,\d\mm\\
&= |t_1-t_0|\Lip^2(\bar\b)\int_{t_0}^{t_1}\int|\D f|^2\,\d(\bar\gl_t)_*\mm\,\d t\\
&\leq C|t_1-t_0|\Lip^2(\bar\b)\int|\D f|^2\,\d\mm,
\end{split}
\]
for some constant $C$, where in the last step we used the bound \eqref{eq:bc}. This shows that $t\mapsto f\circ\bar\gl_t\in L^2(\X)$ is Lipschitz on $[0,1]$ and it is then clear that Lipschitz continuity holds on the whole $\R$.

We now claim that
\begin{equation}
\label{eq:claimweakder}
\frac{f\circ\bar\gl_t-f}{t}\quad\to\quad-\la\nabla f,\nabla\bar\b\ra\qquad\qquad\text{ weakly in }L^2(\X).
\end{equation}
Since the incremental ratios are, by what we just proved, uniformly bounded in $L^2(\X)$, thanks to a simple density and linearity argument to get the claim it is sufficient to prove that for any bounded probability density $\rho$ it holds
\begin{equation}
\label{eq:dahorver}
\lim_{t\to 0}\int\frac{f\circ\bar\gl_t-f}{t}\rho\,\d\mm=-\int\la\nabla f,\nabla\bar\b\ra\rho\,\d\mm
\end{equation}
Putting $\tilde\ppi:=T_*(\rho\mm)$, where $T$ is defined as in Corollary \ref{cor:rappresenta}, and arguing as in the proof of the same corollary we see that $\tilde\ppi$ represents the gradient of $-\bar\b$ and thus \eqref{eq:dahorver} follows from \eqref{eq:firstorder}. Hence \eqref{eq:claimweakder} is proved. 

Since the $\bar\gl_t$'s form a group, we then have proved \eqref{eq:calcoloder} for any $t\in\R$, provided we intend the left hand side as the weak limit in $L^2(\X)$ of the incremental ratios. 

Now notice that since $t\mapsto f\circ\bar\gl_t\in L^2(\X)$ is Lipschitz and Hilbert spaces have the Radon-Nikodym property, the curve is a.e.\ differentiable. Together with the weak convergence just established, this is sufficient to get  \eqref{eq:claimweakder}  for a.e.\ $t$. To conclude it is now enough to notice that the first part of the statement ensures that the right hand side of \eqref{eq:calcoloder} depends  $L^2$-continuously on $t$.
\end{proof}

\subsection{Gradient flow of $\b$: effect on the Dirichlet energy}\label{se:gfdir}
In the previous section we established the link between $\gl_t$ and the reference measure. To get informations about the link between the flow and the distance we shall, as in  \cite{Gigli13}, first look at what happens to the Dirichlet energy along the flow. The metric information will later be recovered via the Sobolev-to-Lipschitz property.

Much like in the smooth case, the starting point for any metric-related information on the flow is Bochner inequality, which here is used to obtain the following Euler equation for $\bar\b$:
\begin{proposition}[Euler equation for $\bar \b$]\label{prop:eul} Let  $f,g\in\fsm \X$ with $\supp(f)\subset B_{\Rbt}(\O)$. Then
\begin{equation}
\label{eq:euler}
\int \Delta f\la\nabla g,\nabla \bar \b\ra\,\d\mm=\int f\big(\la\nabla\Delta g,\nabla\bar \b\ra+2\Delta g\big)\,\d\mm.
\end{equation}
\end{proposition}
\begin{proof} Let $\eps\in\R$ and write the Bochner inequality \eqref{eq:Bochner} for $\bar\b+\eps g\in\fsm \X$ to get
\[
\bd\frac{|\D(\bar\b+\eps g)|^2}2\geq\Big( \frac{|\Delta(\bar\b+\eps g)|^2}N+\la\nabla(\bar\b+\eps g),\nabla\Delta(\bar\b+\eps g)\ra\Big)\mm.
\]
Expand the formula and use the fact that $|\D\bar\b|^2=2\bar\b$ on $B_{\Rbt}(\O)$ (by Corollary \ref{cor:db} and  $\bar\b=\b$ on $B_{\Rbt}(\O)$) and  $\Delta\bar\b=N$ $\mm$-a.e.\ on $B_{\Rbt}(\O)$ (Proposition \ref{prop:laplB} and Corollary \ref{cor:btest}) to get
\[
\bd\Big(\eps\la\nabla g,\nabla\bar\b\ra+\frac{\eps^2}2|\D g|^2\Big)\restr{B_{\Rbt}(\O)}\geq \Big(\eps \big(\la\nabla\Delta g,\nabla\bar \b\ra+2\Delta g\big)+\eps^2\big(\frac{|\Delta g|^2}N+\la\nabla g,\nabla\Delta f\ra\big)\Big)\mm\restr{B_{\Rbt}(\O)}.
\]
Divide by $\eps>0$ (resp. $\eps<0$) and let $\eps\downarrow0$ (resp. $\eps\uparrow0$) to obtain
\[
\bd\la\nabla g,\nabla\bar\b\ra\restr{B_{\Rbt}(\O)}=\big(\la\nabla\Delta g,\nabla\bar \b\ra+2\Delta g\big)\mm\restr{B_{\Rbt}(\O)}.
\]
Multiplying by $f$ and integrating we conclude.
\end{proof}
Our control of $\bar\b$ is good on the ball $B_{\Rbt}(\O)$, but in our applications, we shall need an analogous of formula \eqref{eq:euler} with the function $f$ replaced by the `mollified' function $\h_sf$ for $s>0$ small. To get a control of the error terms that appear we need first to build appropriate cut-off functions (the construction is the same used in Lemma 6.7 in \cite{AmbrosioMondinoSavare13-2} with the finite dimensionality allowing for more precise quantitative estimates):
\begin{lemma}[`Smooth' cut-offs]\label{le:cutoff}
For every $r>0$ there exists a constant $C(r)>0$ such that the following holds. Given $K\subset\Omega$ with $K$ compact and $\Omega$ open such that $\inf_{x\in K,y\in\Omega^c}\sfd(x,y)\geq r$, there exists a  test function $\nchi$ with values in $[0,1]$, which is  1 on $K$, with support in   $\Omega$ and such that
\[
\Lip(\nchi)+\|\Delta\nchi\|_{L^\infty}\leq C(r).
\]
\end{lemma}
\begin{proof}
We shall use the moment estimate proved in \cite{Erbar-Kuwada-Sturm13} (see Theorem 3)
\[
W_2^2({\sf H}_t\delta_x,\delta_x)\leq Nt,\qquad\forall x\in\X,
\]
where $({\sf H_t})$ is the heat flow at the level of probability measures, and the Bakry-\'Emery estimate proved in the same reference (see Theorem 4):
\[
|\D\h_tf|^2+\frac{2t}{N}|\Delta\h_tf|^2\leq\h_t(|\D f|^2),\qquad\mm-a.e.,
\]  
valid for any $f\in L^2(\X)$.

Pick $f(x):=\max\{0,1-\frac1r\sfd(x,K)\}$  and notice that for every $x\in\X$ and $t>0$ we have
\begin{equation*}
\label{eq:linftycutoff}
\begin{split}
|f(x)-\h_t f(x)|&=\Big| \int f(x)- f(y)\,\d{\sf H}_t\delta_x(y)\Big|\leq \Lip(f)\int\sfd(x,y)\,\d{\sf H}_t\delta_x(y)\\
&\leq \Lip(f)W_2(\delta_x,{\sf H}_t\delta_x)\leq t\frac Nr.
\end{split}
\end{equation*}
Thus the function $\h_{\frac r{3N}}f$ takes values in $[0,\frac13]$ on $\Omega^c$ and in $[\frac23,1]$ on $K$. 

Now let $\psi:\R\to[0,1]$ be $C^\infty$, identically 0 on $(-\infty,\frac12]$ and identically 1 on $[\frac23,+\infty)$ and define $\nchi:=\psi\circ\h_{\frac r{3N}} f$. It is then clear that $\nchi$ is 1 on $K$ and with support in  $\Omega$. The fact that $\nchi\in\test\X$ follows from the chain rules for minimal weak upper gradients and Laplacians (see \cite{Gigli12}) and the Bakry-\'Emery estimate grants that
\[
\||\D\h_{\frac r{3N}} f|\|_{L^\infty}+\||\Delta\h_{\frac r{3N}} f|\|_{L^\infty}\leq c(r),
\]
for some constant $c(r)$. Thus using the Sobolev-to-Lipschitz property to replace $\||\D\h_{\frac r{3N}} f|\|_{L^\infty}$ by $\Lip(\h_{\frac r{3N}} f)$ and the chain rules for the Lipschitz constant and the Laplacian we conclude.
\end{proof}
The cut-offs just built allow the following estimates:
\begin{lemma}[Tail estimates]\label{le:tail2}
For every $r>0$ there is a constant $C'(r)>0$ such that the following holds. Given $K\subset\Omega$ with $K$ compact and $\Omega$ open such that $\inf_{x\in K,y\in\Omega^c}\sfd(x,y)\geq r$  and $f\in L^2(\mm)$ with $\supp(f)\subset K$, for every $t> 0$ the quantities
\begin{equation}
\label{eq:altrefuori}
\begin{split}
&\int_{\Omega^c}|\h_tf|^2\,\d\mm
,\quad\int_0^t\!\!\int_{\Omega^c}|\h_sf|^2\,\d\mm\,\d s,\quad\int_0^t\!\!\int_{\Omega^c}|\D\h_sf|^2\,\d\mm\,\d s,\\
&\int_0^t\!\!\int_{\Omega^c}|\Delta\h_sf|^2\,\d\mm\,\d s,\quad\int_0^t\!\!\int_{\Omega^c}|\D\Delta\h_sf|^2\,\d\mm\,\d s
\end{split}
\end{equation}
are all bounded above by $t^2C'(r)\|f\|_{L^2}^2$.
\end{lemma}
\begin{proof} In the course of the proof the value of the constant $C'(r)$ may change in its various occurrences. Let $K_\eps$ be the $\eps$-neighbourhood of $K$ and use repeatedly  Lemma \ref{le:cutoff} above to find test functions $\nchi_i:\X\to[0,1]$, $i=0,\ldots,3$, identically 1 on $\bar K_{\frac{(i+1)r}5}$ and with support in $K_{\frac{(i+2)r}5}$. Put $\eta_i:=1-\nchi_i$.  Also, notice that  the quantities we have to bound depend continuously on $f\in L^2(\X)$, thus taking into account the density property \eqref{eq:densetest} we can, and will, assume that $f\in\test\X$ with support in $K_{\frac r6}$.

Start noticing that
\[
\begin{split}
\frac{\d}{\d s}\frac12\int \eta^2_1|\h_sf|^2\,\d\mm& =\int \eta^2_1\h_sf\Delta\h_sf\,\d\mm=\int\eta^2_1\big(\Delta\frac{|\h_sf|^2}2-|\D\h_sf|^2\big)\,\d\mm\leq\int\eta^2_1\Delta\frac{|\h_sf|^2}2\,\d\mm,
\end{split}
\]
and use the fact that $\eta_1^2=1-2\nchi_1+\nchi_1^2$, that the integral of a Laplacian is 0 and that $\supp(\nchi_1)\subset \{\eta_0=1\}$ to deduce that
\begin{equation}
\label{eq:eta}
\begin{split}
\frac{\d}{\d s}\frac12\int \eta^2_1|\h_sf|^2\,\d\mm&\leq\frac12\int (-2\Delta\nchi_1+\Delta\nchi_1^2)|\h_sf|^2\,\d\mm\\
&\leq\|-2\Delta\nchi_1+\Delta\nchi_1^2\|_{L^\infty}\frac12\int \eta_0^2|\h_sf|^2\,\d\mm\leq C'(r)\frac12\int\eta_0^2|\h_sf|^2\,\d\mm,
\end{split}
\end{equation}
having used the identity $\Delta \nchi^2_1=2\nchi_1\Delta\nchi_1+2|\D\nchi_1|^2$ to get the uniform bound on $\Delta \nchi^2_1$. Analogously, we have
\[
\frac{\d}{\d s}\frac12\int \eta^2_0|\h_sf|^2\,\d\mm\leq C'(r)\frac12\int_{(K_{\frac r5})^c} |\h_sf|^2\,\d\mm\leq  C'(r) \|\h_sf\|_{L^2}^2\leq C'(r)\|f\|_{L^2}^2
\]
and since  $\int \eta^2_0|f|^2\,\d\mm=0$ (because $\eta_0$ and $f$ have disjoint supports), integrating in $s$ we obtain
\begin{equation}
\label{eq:tardi}
\int \eta^2_0|\h_sf|^2\,\d\mm\leq sC'(r)\|f\|_{L^2}^2,\qquad\forall s\geq 0,
\end{equation}
which plugged in \eqref{eq:eta}, integrating in $s$ and using the fact that $\int \eta^2_1|f|^2\,\d\mm=0$ yields the desired control
\[
\int_{\Omega^c}|\h_sf|^2\,\d\mm\leq\int \eta^2_1|\h_sf|^2\,\d\mm\leq s^2\,C'(r)\|f\|_{L^2}^2,\qquad\forall s\geq 0.
\]
The bound for the second quantity in \eqref{eq:altrefuori} follows directly.

 Now notice that integrating in $s$ the identity
\[
\begin{split}
\frac{\d}{\d s}\int\eta_1^2|\h_sf|^2\,\d\mm=2\int\eta_1^2\,\h_sf\Delta\h_sf\,\d\mm=-2\int 2\eta_1\la\nabla\eta_1,\nabla\h_s f\ra\h_sf+\eta_1^2|\D\h_sf|^2\,\d\mm
\end{split}
\]
and using the fact that $\int\eta_1^2|f|^2\,\d\mm=0$ we deduce
\[
\begin{split}
\iint_0^{t}\eta_1^2|\D\h_sf|^2\,\d s\,\d\mm&\leq \iint_0^t2\eta_1\la\nabla\eta_1,\nabla\h_s f\ra\h_sf\,\d r\,\d\mm\\
&\leq2 \sqrt{\iint_0^t \eta_1^2|\D\h_sf|^2\,\d\mm}\sqrt{\iint_0^t|\D\eta_1|^2|\h_sf|^2\,\d s\,\d\mm}.
\end{split}
\]
Taking into account  that $\supp(\eta_1)\subset \{\eta_0=1\}$ we conclude that
\[
\begin{split}
\iint_0^{t}\eta_1^2|\D\h_sf|^2\,\d s\,\d\mm\leq 4\Lip(\eta_1)^2\iint_0^t \eta^2_0|\h_sf|^2\,\d s\,\d\mm\stackrel{\eqref{eq:tardi}}\leq t^2\, C'(r)\|f\|_{L^2}^2.
\end{split}
\]

Analogous computations starting from the derivatives of $\int \eta_2^2|\D\h_sf|^2\,\d\mm$ and $\int\eta_3^2|\Delta\h_sf|^2\,\d\mm$ yield the bounds
\[
\begin{split}
\iint_0^{t}\eta_2^2|\Delta\h_sf|^2\,\d s\,\d\mm&\leq4\Lip(\eta_2)^2  \iint_0^t\eta_1^2|\D\h_sf|^2\,\d s\,\d\mm,\\
\iint_0^{t}\eta_3^2|\D\Delta\h_sf|^2\,\d s\,\d\mm&\leq4\Lip(\eta_3)^2  \iint_0^t\eta_2^2|\Delta\h_sf|^2\,\d s\,\d\mm,
\end{split}
\]
and the conclusion follows recalling that   $\Omega^c\subset \{\eta_i=1\}$ for  $i=1,2,3$.
\end{proof}
These estimates and the Euler equation for $\bar\b$ previously obtained yield the following:
\begin{corollary}\label{cor:meglio}
For every $r>0$ there is a constant  $C''(r)$  such that the following holds.   For $f\in L^2(\X)$ with $\supp(f)\subset B_{\Rbt-r}(\O)$ and every $s>0$ we have
\begin{equation}
\label{eq:daeul2}
\int \la \nabla \h_{2s}f,\nabla\bar\b \ra f\,\d\mm=\int-\frac {N}2|\h_sf|^2+2s|\D\h_sf|^2\,\d\mm+{\rm Rem}(f,s),
\end{equation}
where the reminder term ${\rm Rem}(f,s)$ can be estimated as
\begin{equation}
\label{eq:resto}
|{\rm Rem}(f,s)|\leq s^2\, C''(r)\|f\|_{L^2}^2.
\end{equation}
\end{corollary}
\begin{proof} It is readily verified that all the integrals appearing in \eqref{eq:daeul2} vary continuously as $f$ varies in $L^2(\X)$, hence recalling the density property \eqref{eq:densetest} we can, and will, assume that $f\in \fsm\X$ with support in ${B_{\Rbt-r}(\O)}$.

It is also easy to see that the function $t\mapsto\int \la \nabla \h_{s+t}f,\nabla\bar\b \ra \h_{s-t}f\,\d\mm $ is  $C^1$ on $[0,s]$, thus
\begin{equation}
\label{eq:step1}
\int \la \nabla \h_{2s}f,\nabla\bar\b \ra f\,\d\mm=\int \la \nabla \h_{s}f,\nabla\bar\b \ra \h_sf\,\d\mm+\int_0^s\frac{\d}{\d t}\int \la \nabla \h_{s+t}f,\nabla\bar\b \ra \h_{s-t}f\,\d\mm\,\d t.
\end{equation}
For the first addend on the right, notice that
\begin{equation}
\label{eq:step2}
\int \la \nabla \h_{s}f,\nabla\bar\b \ra \h_sf\,\d\mm=\int \la \nabla\tfrac{ |\h_{s}f|^2}2,\nabla\bar\b \ra \,\d\mm=-\frac N2\int |\h_sf|^2\,\d\mm+\int\frac{N-\Delta\bar\b}2|\h_sf|^2\,\d\mm,
\end{equation}
and that since $\Delta\bar\b=N$ on $B_{\Rbt}(\O)$, by Lemma \ref{le:tail2} we have
\begin{equation*}
\label{eq:step3}
\bigg|\int\frac{N-\Delta\bar\b}2|\h_sf|^2\,\d\mm\bigg|\leq(N+\|\Delta\bar\b\|_{L^\infty})\int_{X\setminus B_{\Rbt}(\O)}|\h_sf|^2\,\d\mm\leq s^2C'(r)(N+\|\Delta\bar\b\|_{L^\infty})\|f\|_{L^2}^2.
\end{equation*}
For the second added in the right of \eqref{eq:step1}, it is readily verified that the derivative  can be computed passing the limit inside the integral, thus obtaining
\begin{equation}
\label{eq:step4}
\int_0^s\frac{\d}{\d t}\int \la \nabla \h_{s+t}f,\nabla\bar\b \ra \h_{s-t}f\,\d\mm\,\d t=\iint_0^s \la \nabla\Delta \h_{s+t}f,\nabla\bar\b \ra \h_{s-t}f - \la \nabla \h_{s+t}f,\nabla\bar\b \ra \Delta f_{s-t}\,\d t\,\d\mm.
\end{equation}
Now let $\nchi$ be the cut-off given by Lemma \ref{le:cutoff} relative to the compact set $\bar B_{\Rbt-\frac r2}(\O)$ and the open set $  B_{\Rbt}(\O)$ and notice that
\begin{equation}
\label{eq:step5}
\begin{split}
\iint_0^s& \la \nabla\Delta \h_{s+t}f,\nabla\bar\b \ra f_{s-t} - \la \nabla \h_{s+t}f,\nabla\bar\b \ra \Delta \h_{s-t}f\,\d t\,\d\mm\\
&=\underbrace{\iint_0^s \la \nabla\Delta \h_{s+t}f,\nabla\bar\b \ra \nchi \h_{s-t}f - \la \nabla \h_{s+t}f,\nabla\bar\b \ra \Delta( \nchi \h_{s-t}f)\,\d t\,\d\mm}_A\\
&\qquad+\underbrace{\iint_0^s \la \nabla\Delta \h_{s+t}f,\nabla\bar\b \ra (1-\nchi) \h_{s-t}f - \la \nabla \h_{s+t}f,\nabla\bar\b \ra \Delta \big((1-\nchi) \h_{s-t}f\big)\,\d t\,\d\mm}_B,
\end{split}
\end{equation}
thus using the fact that $\nchi  f_{s-t}$ is a test function with support in $B_{\Rbt}(\O)$, by Proposition \ref{prop:eul} we get
\[
A=2\iint_0^s - \nchi\h_{s-t}f\Delta \h_{s+t}f\,\d t\,\d\mm=2s\int |\D \h_sf|^2\,\d\mm+2\iint_0^s(1-\nchi) \h_{s-t}f\Delta \h_{s+t}f\,\d t\,\d\mm.
\]
Since $1-\nchi$ is identically 0 on $B_{\Rbt-r/2}(\O)$, by  Lemma \ref{le:tail2} we obtain
\begin{equation}
\label{eq:step6}
\begin{split}
\bigg|A-2s\int |\D \h_sf|^2\,\d\mm\bigg|&\leq\sqrt{\int_0^{2s}\int_{B_{\Rbt-r/2}^c(\O)} |\h_tf|^2\,\d\mm\,\d t}\sqrt{\int_0^{2s}\int_{B_{\Rbt-r/2}^c(\O)} |\Delta\h_tf|^2\,\d\mm\,\d t}\\
&\leq 4s^2 C'(r/2)\|f\|_{L^2}^2.
\end{split}
\end{equation}
For the same reason,  letting $S:=\max\{1,\Lip(\bar\b),\Lip(\nchi),\|\Delta\nchi\|_{L^\infty}\}$ we have
\begin{equation}
\label{eq:step7}
\begin{split}
|B|&\leq S^2\int_0^{s}\int_{B_{\Rbt-r/2}^c(\O)}|\h_{s-t}f||\nabla\Delta \h_{s+t}f|+|\nabla \h_{s+t}f|(|\h_{s-t}f|+2|\nabla\h_{s-t}f|+|\Delta\h_{s-t}f|)\,\d\mm\,\d t\\
&\leq 10s^2\,S^2C(r/2)\|f\|_{L^2}^2.
\end{split}
\end{equation}
The conclusion comes collecting the informations in \eqref{eq:step1}, \eqref{eq:step2},  \eqref{eq:step4}, \eqref{eq:step5}, \eqref{eq:step6}, \eqref{eq:step7}.
\end{proof}
We are now ready to prove the main result of this section.
\begin{theorem}\label{thm:key2}
Let $f\in L^2(\mm)$ and $T>0$ be such that $\supp(f)\subset B_{e^{-T}\Rbt}(\O)$. Then 
\[
\mathcal E(f\circ \bar\gl_t)=e^{(N-2)t}\mathcal E(f),\qquad\forall t\in[0,T].
\]
\end{theorem}
\begin{proof} 
Let $f_t:=f\circ\bar\gl_t$ and notice that since $\supp(f_t)\subset B_{\Rbt}(\O)$ for every $t\in[0,T]$, from \eqref{eq:pfent} we have $\int|f_t|^2\,\d\mm=\int|f|^2\,\d(\bar\gl_t)_*\mm=e^{Nt}\int|f|^2\,\d\mm$ and thus
\begin{equation}
\label{eq:noh}
\frac{\d}{\d t}\frac12\int|f_t|^2\,\d\mm=\frac N2\int|f_t|^2\,\d\mm,\qquad\forall t\in[0,T].
\end{equation} 
Now pick $s>0$ and notice that since by Lemma \ref{le:compgl} we have that $t\mapsto f_t\in L^2(\X)$ is Lipschitz, we also have that $t\mapsto\h_sf_t\in L^2(\X)$ is Lipschitz. Then for a.e.\ $t\in[0,T]$ we have
\[
\begin{split}
\frac{\d}{\d t}\frac12\int|\h_sf_t|^2\,\d\mm&= \lim_{h\to 0}\int \h_sf_t\frac{\h_s(f_t\circ \bar\gl_h)-\h_sf_t}{h}\,\d\mm\\
&=\lim_{h\to 0}\int \h_{2s}f_t\frac{f_t\circ \bar\gl_h-f_t}{h}\,\d\mm\\
&=\lim_{h\to 0}\int \frac{e^{Nh}\h_{2s}f_t\circ \bar\gl_h^{-1}-\h_{2s}f_t}hf_t\,\d\mm\\
&=\int (N\h_{2s}f_t+\la\nabla\h_{2s}f_t,\nabla \bar\b\ra) f_t\,\d\mm,\\
&=\int N|\h_{s}f_t|^2+\la\nabla\h_{2s}f_t,\nabla \bar\b\ra f_t\,\d\mm,
\end{split}
\]
having used Lemma \ref{le:compgl} in the penultimate step. Thus from Corollary \ref{cor:meglio} we deduce
\begin{equation}
\label{eq:sih}
\frac{\d}{\d t}\frac12\int|\h_sf_t|^2\,\d\mm=\int\frac {N}2|\h_sf_t|^2+2s|\D\h_sf_t|^2\,\d\mm+{\rm Rem}(f_t,s),
\end{equation}
with ${\rm Rem}(f_t,s)$ satisfying the bound \eqref{eq:resto}. 

Consider the quantity
\[
G(t,s):=\int\frac{|f_t|^2-|\h_sf_t|^2}{4s}\,\d\mm,
\]
and notice that  \eqref{eq:noh} and \eqref{eq:sih} give that for any $s>0$ the map $t\mapsto G(t,s)$ is Lipschitz with
\begin{equation}
\label{eq:dergts}
\frac{\d}{\d t}G(t,s)=NG(t,s)-2\mathcal E(\h_sf_t)+\frac{{\rm Rem}(f_t,s)}s.
\end{equation}
Now assume for the moment that $f\in W^{1,2}(\X)$, let $c:=\sup_{s\in(0,1)}\frac{{\rm Rem}(f_t,s)}s<\infty$, use the trivial bound $G(0,s)\leq\mathcal E(f_0)$ to deduce from the last identity that
\[
G(t,s)\leq \mathcal E(f_0)+cT+N\int_0^tG(r,s)\,\d r, \qquad\forall t\in[0,T],\ s\in (0,1).
\]
By the Gronwall inequality in the integral form we get the uniform bound
\[
G(t,s)\leq \big( \mathcal E(f_0)+cT\big)e^{N t},\qquad\forall t\in[0,T], \ \forall s\in(0,1).
\]
Noticing that $G(t,s)\uparrow\mathcal E(f_t)$ as $s\downarrow 0$, this last bound ensures that $\mathcal E(f_t)$ is uniformly bounded in $t\in[0,T]$. We can therefore pass to the limit as $s\downarrow0$ in \eqref{eq:dergts} noticing that the right hand side is uniformly bounded and pointwise converges to $(N-2)\mathcal E(f_t)$ to conclude that
\[
\frac{\d}{\d t}\mathcal E(f_t)=(N-2)\mathcal E(f_t),\qquad a.e.\ t\in[0,T],
\]
and the conclusion follows.

To  remove the assumption that $f\in W^{1,2}(\X)$ it is sufficient to show that if $f_T\in W^{1,2}(\X)$ then $f\in W^{1,2}(\X)$ as well: this follows from the very same arguments just used replacing $\bar\gl_t$ with its inverse $\bar\gl_t^{-1}=\bar\gl_{-t}$.
\end{proof}

The statement of the previous theorem can be easily `localized' thanks to the chain and Leibniz rules \eqref{eq:calcrules}. Notice that in the following Corollary we will also come back to the original flow $(\gl_t)$ in place of the regularized one $(\bar\gl_t)$.
\begin{corollary}\label{cor:enid}
Let $f\in L^2(\mm)$ and $T>0$ be with  $\supp(f)\subset B_{e^{-T}\Rb}(\O)$. Then $f\in W^{1,2}(\X)$ if and only if $f\circ\gl_T\in W^{1,2}(\X)$ and in this case 
\begin{equation*}
\label{eq:key}
|\D (f\circ \gl_T)|=e^{-T}|\D f|\circ\gl_T,\qquad \mm-a.e.
\end{equation*}
\end{corollary}
\begin{proof} Fix $f,T$ as in the assumptions and notice that by compactness there exists $\Rbt<\Rb$ such that $\supp(f)\subset B_{e^{-T}\Rbt}(\O)$. Choosing such $\Rbt$ and building a corresponding function $\bar\b$ and its gradient flow $\bar\gl$, we see that $f\circ\gl_T=f\circ\bar\gl_t$ $\mm$-a.e.\  and thus by Theorem \ref{thm:key2} above we deduce that $f\in W^{1,2}(\X)$ if and only if $f\circ\gl_T\in W^{1,2}(\X)$. 

Now assume that $f\in W^{1,2}(\X)$ and notice that by the locality property \eqref{eq:locwug} we also have $|\D (f\circ\gl_T)|=|\D (f\circ\bar\gl_T)|$ $\mm$-a.e.\ so that our conclusion becomes
\[
|\D (f\circ\bar \gl_T)|=e^{-T}|\D f|\circ\bar \gl_T,\qquad \mm-a.e.
\]
and with an approximation argument based on the density of $L^\infty\cap W^{1,2}(\X)$ in $W^{1,2}(\X)$ (easy to establish from the definitions),  we can, and will, assume that $f\in W^{1,2}\cap L^\infty(\X)$.

Observe that for any two functions $f_1,f_2\in W^{1,2}(\X)$ with $\supp(f_i)\subset  B_{e^{-T}\Rbt}(\O)$, Theorem \ref{thm:key2} gives, by polarization, that
\[
\int\la\nabla (f_1\circ\bar\gl_T),\nabla (f_2\circ\bar\gl_T)\ra\,\d\mm=e^{(N-2)T}\int\la\nabla f_1,\nabla f_2\ra\,\d\mm,
\]
then pick  an arbitrary Lipschitz function $g$ with support in $ B_{e^{-T}\Rbt}(\O)$, notice that $f,g,f^2,fg$ are all in $W^{1,2}(\X)$ with support in $ B_{e^{-T}\Rbt}(\O)$, put for brevity $f_T:=f\circ\bar\gl_T$, $g_T:=g\circ\bar\gl_T$ and use this last identity to get
\[
\begin{split}
\int|\D f_T|^2g_T\,\d\mm&=\int\la\nabla (f_Tg_T),\nabla f_T\ra-\la \nabla \tfrac{f_T^2}2,\nabla g_T\ra\d\mm\\
&=e^{(N-2)T}\int\la\nabla (fg),\nabla f\ra-\la \nabla \tfrac{f^2}2,\nabla g\ra\d\mm=e^{(N-2)T}\int|\D f|^2g\,\d\mm.
\end{split}
\]
Since by \eqref{eq:pfent}  we have $\int|\D f_T|^2g_T\,\d\mm=e^{NT}\int|\D f_T|^2\circ \bar\gl^{-1}_Tg\,\d\mm$, we conclude that
\[
\int|\D f_T|^2\circ \bar\gl^{-1}_Tg\,\d\mm=e^{-2T}\int|\D f|^2g\,\d\mm,
\]
which by the arbitrariness of $g$ is the conclusion.
\end{proof}

\subsection{Gradient flow of $\b$: precise representative and first metric informations}
We shall now use the Sobolev-to-Lipschitz property of $\X$ to obtain information about the behaviour of the metric under the flow $(\gl_t)$:
\begin{theorem}\label{thm:rapprcont}
The map $\gl$, seen as a map from $\R^+\times  B_{\Rb}(\O)$ to $ B_{\Rb}(\O)$ admits a continuous representative w.r.t.\ the measure $(\mathcal L^1\times \mm)\restr{\R^+\times  B_{\Rb}(\O)}$. Still denoting such representative  by $\gl$, we have:
\begin{itemize}
\item[i)] for every $t,s\in\R^+$ and $x\in  B_{\Rb}(\O)$ it holds
\begin{equation}
\label{eq:glpunt}
\begin{split}
\gl_t(\gl_s(x))&=\gl_{t+s}(x),\\
\sfd(\gl_s(x),\gl_t(x))&=|e^{-s}-e^{-t}|\sfd(x,\O).
\end{split}
\end{equation}
\item[ii)] for every $t\in\R^+$, $\gl_t$ is an invertible locally Lipschitz map from $ B_{\Rb}(\O)$ to  $ B_{e^{-t}\Rb}(\O)$ whose inverse is also locally Lipschitz. Moreover, given a curve $\gamma$ with values in $ B_\Rb(\O)$, putting $\tilde\gamma:=\gl_t\circ\gamma$ we have 
\begin{equation}
\label{eq:locspeeds}
|\dot{\tilde\gamma}_s|=e^{-t}|\dot\gamma_s|\qquad\text{ for a.e.\ }s\in[0,1],
\end{equation}
meaning that one of the curves is absolutely continuous if and only if the other is and in this case their metric speeds are related by the stated identity.
\end{itemize}
\end{theorem}
\begin{proof} Fix $t\in\R^+$ and notice that by construction the image of $B_{\Rb}(\O)\setminus\mathcal N$ under $\gl_t$ is contained in $B_{e^{-t}\Rb}(\O)$. Fix $x_0\in B_{e^{-t}\Rb}(\O)$, let $r>0$ be such that $B_{3r}(x_0)\subset B_{e^{-t}\Rb}(\O)$ and let $\mathcal D$  be a countable set of 1-Lipschitz functions, all with support in $B_{3r}(x_0)$ dense in the space of 1-Lipschitz functions with support in $B_{3r}(x_0)$ w.r.t.\ uniform convergence. It is then clear that
\[
\sfd(y_0,y_1)=\sup_{f\in \mathcal D} |f(y_1)-f(y_0)|,\qquad\forall y_0,y_1\in B_r(x_0).
\]
Pick $f\in \mathcal D$  and apply Corollary \ref{cor:enid} to deduce that $f\circ\gl_t\in W^{1,2}(\X)$ with
\[
|\D (f\circ\gl_t)|=e^{-t}|\D f|\circ\gl_t\leq e^{-t},\qquad\mm-a.e..
\]
Since $\X$ has the Sobolev-to-Lipschitz property, $f\circ\gl_t$ has a $e^{-t}$-Lipschitz representative and since $\mathcal D$ is countable, we deduce that there exists a $\mm$-negligible Borel set $\mathcal N'$ such that the restriction of $f\circ\gl_t$ to $\X\setminus(\mathcal N\cup\mathcal N')$ is $e^{-t}$-Lipschitz for every $f\in\mathcal D$. 

Therefore for $x_0,x_1\in \gl_t^{-1}(B_r({x_0}))\setminus(\mathcal N\cup\mathcal N')$ we have
\[
\sfd\big(\gl_t(x_0),\gl_t(x_1)\big)=\sup_{f\in\mathcal D}\big|f(\gl_t(x_0))-f(\gl_t(x_1))\big|\leq e^{-t}\sfd(x_0,x_1),
\]
showing that $\gl_t$ has a $e^{-t}$-Lipschitz representative on the preimage of $B_r(x_0)$. Then the arbitrariness of $x_0$, the Lindelof property of $B_{e^{-t}\Rb}(\O)$ and the essential surjectivity of \(\gl_t\) ensure that $\gl_t:B_\Rb(\O)\to B_{e^{-t}\Rb}(\O)$ has a representative which is locally $e^{-t}$-Lipschitz and from now on we shall identify $\gl_t$ with such representative.

It then follows from the arbitrariness of $t\in\R^+$ and the uniform continuity in $t$ granted by the second identity in \eqref{eq:distgl}  that  $\gl$, seen as a map from $\R^+\times B_{\Rb}(\O)$ to $B_{\Rb}(\O)$, admits a continuous representative w.r.t.\ the measure $(\mathcal L^1\times \mm)\restr{\R^+\times B_{\Rb}(\O)}$, as claimed. 

The identities \eqref{eq:glpunt} then follow directly from \eqref{eq:distgl} recalling that $(\gl_t)_*\mm\restr{B_\Rb(\O)}\ll\mm$.

Inequality $\leq$ in \eqref{eq:locspeeds}  is a direct consequence of the fact that $\gl_t$ is locally $e^{-t}$-Lipschitz. To conclude it is therefore sufficient to prove that $\gl_t^{-1}:B_{e^{-t}\Rb}(\O)\to B_\Rb(\O) $, a priori well defined only $\mm$-a.e.\ (recall  $(ii)$ of Corollary \ref{cor:gl}) has a representative which is locally $e^t$-Lipschitz, as then it is clear that such representative is the inverse of the continuous representative of $\gl_t$. Such property of $\gl_t^{-1}$ can be proved by the very same means used to prove the local $e^{-t}$-Lipschitzianity of $\gl_t$. Just notice that if $f$ is a 1-Lipschitz function with support in $B_\Rb(\O)$, then $f\circ\gl_t^{-1}$ has support in $B_{e^{-t}\Rb}(\O)$ and, by Corollary \ref{cor:enid}, it belongs to $W^{1,2}(\X)$ with $|\D (f\circ\gl_t^{-1})|=|\D f|\circ\gl_t^{-1}\leq e^t$ $\mm$-a.e.. The conclusion then follows arguing as above.
\end{proof}

\begin{quote}
\underline{From now on, when considering the maps $\gl_t$ on $B_\Rb(\O)$ we shall always refer}\\
\underline{to their continuous versions.}
\end{quote}
Furthermore, for $t<0$ and $x\in \X$  such that $x\in B_{e^t\Rb}(\O)$, we put
\[
\gl_t(x):=\gl_{-t}^{-1}(x).
\]
\subsection{Basic properties of the sphere $S_{\Rb/2}(\O)$}
We consider the sphere
\[
S_{\Rb/2}(\O):=\big\{x\in \X\ :\ \sfd(x,\O)=\tfrac{\Rb}2\big\}
\]
and the projection map $\pr: B_{\Rb}(\O)\setminus\{\O\}\to S_{\Rb/2}(\O)$ given by
\[
\pr(x):=\gl_{\log(\frac{2\sfd(x,\O)}\Rb)}(x).
\]
Notice that by Theorem \ref{thm:rapprcont} the map $\pr$ is well defined and  locally Lipschitz.
\begin{proposition}\label{prop:sfera}
$S_{\Rb/2}(\O)$ is either one point, or two points or a Lipschitz-path connected subset of $\X$, i.e.\ a subset such that for every $x,y\in S_{\Rb/2}(\O)$ there is a Lipschitz curve with values in $S_{\Rb/2}(\O)$ connecting $x$ to $y$.
\end{proposition}
\begin{proof}
It will be convenient to work with the sphere $S_{\Rb/4}(\O)$: notice that by Theorem \ref{thm:rapprcont} it has the same cardinality of $S_{\Rb/2}(\O)$ and it is Lipschitz path connected if and only if $S_{\Rb/2}(\O)$ is.

Thus we shall assume that $S_{\Rb/4}(\O)$ contains at least 3 points $x_1,x_2,x_3$ and notice that to conclude it is sufficient to show that there are Lipschitz curves contained in $S_{\Rb/4}(\O)$ connecting these. We argue by contradiction and assume for a moment that there are no Lipschitz curves whose image is contained in $S_{\Rb/4}(\O)$ joining $x_1$ to $x_2$ and similarly no such curve joining $x_1$ to $x_3$. 

Let $x\in B_{\Rb/4}(x_1)$ and notice that any given geodesic from $x_1$ to $x$ does not pass through $\O$ and that the triangle inequality ensures that any given geodesic connecting $x$ to $x_2$ must stay in $B_\Rb(\O)$. Concatenate these two geodesics and assume that the resulting curve does not pass through $\O$. Then composing it with the locally Lipschitz map $\gl_{\log 2}\circ\pr$ we would obtain a Lipschitz curve from $x_1$ to $x_2$ lying entirely on $S_{\Rb/4}(\O)$, which contradicts our assumption. Thus the concatenation passes through $\O$, which forces the geodesic from $x$ to $x_2$ to pass through $\O$. Hence $\sfd(x,x_2)=\sfd(x,\O)+\frac\Rb4$ and arguing symmetrically we deduce that
\begin{equation}
\label{eq:distuguali}
\sfd(x,x_2)=\sfd(x,x_3),\qquad\forall x\in B_{\Rb/4}(x_1).
\end{equation}
Now consider the probability measures
\[
\mu_0:=\frac{1}{\mm(B_{\Rb/4}(x_1))}\mm\restr{B_{\Rb/4}(x_1)}\qquad\qquad \mu_1:=\frac12\big(\delta_{x_2}+\delta_{x_3}\big)
\]
and notice that the identity \eqref{eq:distuguali} gives that every admissible transport plan between them is optimal. In particular, this is the case for the plan $\mu_0\times\mu_1$. But being it not induced by a map, we found a contradiction with the fact that optimal plans on $\RCD^*(0,N)$ spaces must be induced by maps (see \cite{GigliRajalaSturm13}).

Hence there is a Lipschitz curve with image contained in $S_{\Rb/4}(\O)$ connecting $x_1$ to either $x_2$ or $x_3$, say $x_2$.  Repeating the argument swapping the roles of $x_1$ and $x_3$ we conclude.
\end{proof}

\begin{corollary}[Conclusion in the easy cases]\label{cor:easy} The following holds.
\begin{itemize} 
\item[i)] Assume that   $S_{\Rb/2}(\O)$ consists of one point. Then  $(\X,\sfd)$ is isometric to $[0,{\rm diam}(\X)]$ ($[0,\infty)$ if $\X$ is unbounded) with an isometry which sends $\O$ in $0$ and the measure $\mm\restr{B_\Rb}(\O)$ to  the measure $c\, x^{N-1}\d x$ for  $c:=N \mm(B_\Rb(\O))$.
\item[ii)] Assume that $S_{\Rb/2}(\O)$ consists of two points. Then   $(\X,\sfd)$ is a 1-dimensional Riemannian manifold, possibly with boundary,  and there is a bijective local isometry (in the sense of distance-preserving maps) from $B_{\Rb}(\O)$ to $(-\Rb,\Rb)$ sending $\O$ to $0$ and the measure  $\mm\restr{B_\Rb(\O)}$ to the measure  $c\,|x|^{N-1} \d x$ for $c:=\frac12N\mm(B_\Rb(\O))$. Moreover, such local isometry is an  isometry when restricted to $\bar B_{\Rb/2}(\O)$.
\end{itemize}
\end{corollary}
\begin{proof} 

\noindent{\bf (i)} By Theorem \ref{thm:rapprcont} we know that $\gl_t$ is a bijection of $S_r(\O)$ into $S_{e^{-t}r}(\O)$ for every $r<\Rb$, thus the assumption grants that for every $r<\Rb$ there is exactly one point at distance $r$ from $\O$. This shows that $B_\Rb(\O)$ is isometric to $[0,\Rb)$. The claim about the measure follows from Corollary \ref{cor:contdis}. To conclude that  the isometry can be extended to the whole $\X$ it is sufficient to show that the map $x\mapsto\sfd(x,\O)$ is injective. This follows from the very same argument by contradiction based on existence and uniqueness of optimal maps used in Proposition \ref{prop:sfera} above, using the fact that the restriction of $\sfd(\cdot,\O)$ to $B_\Rb(\O)$ is injective.

\noindent{\bf (ii)} Let $x_L,x_R$ be the two points in $S_{\Rb/2}(\O)$ and let $L:=\pr^{-1}(x_L)$, $R:=\pr^{-1}(x_R)$. Then $\{L,R\}$ is a partition of $B_\Rb(\O)\setminus \{\O\}$ and the fact that $\X$ is geodesic ensures that both $L$ and $R$ are isometric to $(0,\Rb)$. Now notice that a curve joining  $x_0\in B_{\Rb/2}(\O)\cap L$ to  $x_1\in B_{\Rb/2}(\O)\cap R$ either passes through $\O$ or through a point at distance $\Rb$ from $\O$. Given that the length of a curve of the second kind is at least $\Rb$, the metric claims all follow. The claim about the measure follow instead from Corollary \ref{cor:contdis} as before. 

{
Finally let us prove that $(\X,\sfd)$ is a 1-dimensional  manifold, possibly with boundary. For let us define 
\[
L'=\big\{x\in X\setminus\{\O\}\textrm{ such that \(\gamma \cap B_{\Rb}(\O)\subset L\) for every geodesic \(\gamma\) between \(x\) and \(\O\)}\big\}
\]
and 
\[
R'=\big\{x\in  X\setminus\{\O\}\textrm{ such that \(\gamma \cap B_{\Rb}(\O)\subset R\) for every geodesic \(\gamma\) between \(x\) and \(\O\)}\big\}.
\]
We claim that $\sfd(\cdot,O):L'\to\R^+$ is injective. To prove this start observing that for $x'\in L'\setminus L=L'\setminus B_\Rb(\O)$, any geodesic $\gamma$ from $x'$ to $O$ must satisfy $\gamma\supset L$. It follows that $\sfd(x,x')=\sfd(x',\O)-\sfd(x,\O)$ for any $x'\in L'$ and $x\in L$, thus if there where $x_1,x_2\in L'\setminus L$ with $\sfd(x_1,\O)=\sfd(x_2,\O)$ we would have, much like in Proposition \ref{prop:sfera}, that any transport plan between the measures
\[
\mu_0:=\mm(L)^{-1}\mm\restr L,\qquad\qquad \mu_1:=\frac12(\delta_{x_1}+\delta_{x_2}),
\]
would be optimal and thus induced by a map (see \cite{GigliRajalaSturm13}), contradicting the fact that $\mu_0\times\mu_1$ is not induced by a map.

Thus $\sfd(\cdot,O):L'\to\R^+$ is indeed injective and it is then clear that it is also an isometry. The same holds for $R'$, hence the conclusion follows by elementary topology if we show that  \(\X=\overline{L'}\cup \overline{R'}\).  This is the same as to say that the open set $\Omega:=\X\setminus(\overline{L'}\cup \overline{R'})$ has measure 0. By definition of $L$ and $R$, we know that for every $x\in\Omega$ there is a geodesic from $x$ to $\O$ passing through $x_L$ and another passing through $x_R$. Hence $\sfd(x,x_L)=\sfd(x,x_R)$ for any $x\in\Omega$ and if $\mm(\Omega)>0$ we can consider the measures
\[
\mu_0:=\mm(\Omega)^{-1}\mm\restr\Omega,\qquad\qquad \mu_1:=\frac12(\delta_{x_L}+\delta_{x_R}),
\]
and obtain a contradiction as before noticing that the plan $\mu_0\times\mu_1$ would be optimal and not induced by a map.
%
}
\end{proof}
\begin{quote}
\underline{From now on, we shall always assume that $S_{\Rb}(\O)$ contains at least 3 points.}
\end{quote}

\subsection{The sphere equipped with the induced distance and measure}\label{se:xp}

\begin{definition} We put $\X':=S_{\Rb/2}(\O)$. For $x',y'\in \X'$ we define $\sfd'(x',y')$ as
\[
\sfd'(x',y')^2:=\inf \int_0^1|\dot\gamma_t|^2\,\d t,
\]
where the infimum is taken among all Lipschitz curves $\gamma:[0,1]\to \X'\subset \X$ and the metric speed is computed w.r.t.\ the distance $\sfd$.

The measure $\mm'$ on $\X'$ is defined as
\[
\mm':=\mm_{\Rb/2},
\]
where $\mm_{\Rb/2}$ is obtained disintegrating $\mm$ along $\sfd(\cdot,\O)$ (recall Corollary \ref{cor:contdis}).
\end{definition}
Proposition \ref{prop:sfera} and the fact that we assumed $S_{\Rb/2}(\O)$ to contain at least 3 points grant that $\sfd'$ is finite and it is then easy to see that it is indeed a distance on $\X'$ inducing  the same topology coming from the inclusion $\X'\subset \X$. In particular,  $\mm'$ is a Borel measure on $(\X',\sfd')$. It is also clear from \eqref{eq:glmeas} that
\begin{equation}
\label{eq:relmisure}
\pr_*(\mm\restr{B_{\Rb}(\O)})=\mm(B_{\Rb}(\O))\mm'.
\end{equation}
Moreover, since for any Lipschitz curve $\gamma$ with values in $\X'$ we have - as it is easy to check - that $\sfd(\gamma_t,\gamma_s)\leq\sfd'(\gamma_t,\gamma_s)\leq \int_t^s|\dot\gamma_r|\,\d r$ for any $t<s$, $t,s\in[0,1]$, we deduce that
\begin{equation}
\label{eq:veluguali}
\begin{split}
&\text{a  curve $\gamma$ with values  in $\X'$ is absolutely continuous w.r.t.\ $\sfd'$ if and only if it is so w.r.t.\ $\sfd$}\\
&\text{and in this case the metric speeds computed w.r.t.\ the two distances are the same.}
\end{split}
\end{equation}
At this stage of the paper we begin considering Sobolev functions on different spaces. Although a priori there can be no confusion, for better  clarity we shall denote the minimal weak upper gradient of the Sobolev function $f$ defined, say, on the space $\X$ by $|\D f|_\X$ rather than by $|\D f|$.

Another notation that we introduce is ${\rm ms}_t(\gamma)$ for the metric speed $|\dot\gamma_t|$ of the absolutely continuous curve $\gamma$ at time $t$.

With that said, the following link between Sobolev functions on $\X$ and on $\X'$ is easily established: 
\begin{proposition}\label{prop:linksob}
Let $[a,b]\subset \,(0,\Rb)$, $h\in \Lip(\R)$ with support in $(0,\Rb)$ and identically 1 on $[a,b]$ and $f\in L^2(\X)$ of the form $f(x)=g(\pr(x))h(\sfd(x,\O))$ for some $g\in L^2(\mm')$. 

Assume that $f\in W^{1,2}(\X)$. Then  $g\in W^{1,2}(\X')$ and 
\begin{equation}
\label{eq:sobxp1}
|\D f|_\X(x)\geq \frac{\Rb}{2\sfd(x,\O)}|\D g|_{\X'}(\pr(x)) ,\qquad\text{for $\mm$-a.e.\ x such that }\sfd(x,\O)\in[a,b].
\end{equation}
\end{proposition}
\begin{proof}
Fix $f\in W^{1,2}(\X)$, let $\ppi'$ be a test plan on $\X'$ and pick $[a',b']\subset[a,b]$ with $a'<b'$. Consider the map $P:\X'\times[a',b']\to \X$ given by $P(x,d):=\gl_{\log(\frac\Rb{2d})}(x)$, the induced map  $\hat P:C([0,1],\X')\times [a',b']\to C([0,1],\X)$ defined as  $\hat P(\gamma,d)_t:=P(\gamma_t,d)$ and consider the plan
\[
\ppi:=\hat P_*(\ppi'\times(|b'-a'|^{-1}\mathcal L^1\restr{[a',b']}))\in \prob{C([0,1],\X)}.
\]
Since $\sfd'\geq\sfd$ on $\X'$, we see from the fact that $\gl_t$ is locally Lipschitz ($(ii)$ in Theorem \ref{thm:rapprcont}) and the compactness of $\X'\times[a',b']$ that  $P$ is Lipschitz, thus since  $\ppi'$ has  finite kinetic energy, we conclude that $\ppi$ has also finite kinetic energy. Moreover, from \eqref{eq:pfdo}, the definition of $\mm'$ and the fact that $\ppi'$ has bounded compression, we deduce that $\ppi$ has also bounded compression. In summary: $\ppi$ is a test plan on $\X$.

Notice that by construction, for $\ppi$-a.e.\ $\gamma$ we have $f(\gamma_t)=g(\pr(\gamma_t))$ and that from \eqref{eq:veluguali} and \eqref{eq:locspeeds} we see that ${\rm ms}_t(\hat P(\gamma,d))=\frac{2d}{\Rb}|\dot \gamma_t|$ for a.e.\ $t$. Then we have:
\[
\begin{split}
\int |g(\gamma_1)-g(\gamma_0)|\,\d\ppi'(\gamma)&=\int |f(\gamma_1)-f(\gamma_0)|\,\d\ppi(\gamma)\\
&\leq \iint_0^1|\D f|_\X(\gamma_t)|\dot\gamma_t|\,\d t\,\d\ppi(\gamma)\\
&=\frac1{b'-a'}\iint_0^1\int_{a'}^{b'}|\D f|_\X(P(\gamma,d)_t){\rm ms}_t(P(\gamma,d))\,\d d\,\d t\,\d\ppi'(\gamma)\\
&=\iint_0^1\left(\frac1{b'-a'}\int_{a'}^{b'}\frac{2d}{\Rb}|\D f|_\X(P(\gamma,d)_t)\,\d d\right)|\dot\gamma_t| \,\d t\,\d\ppi'(\gamma),
\end{split}
\]
which by the arbitrariness of $\ppi'$ shows that $g\in W^{1,2}(\X')$ and 
\[
|\D g|_{\X'}(x')\leq \frac1{b'-a'}\int_{a'}^{b'}\frac{2d}{\Rb}|\D f|( \gl_{\log(\frac\Rb{2d})}(x')) \,\d d,\qquad \mm'-a.e.\ x'.
\]
Then the arbitrariness of $a',b'$ yields  \eqref{eq:sobxp1}.
\end{proof}
In fact, also the inequality opposite of \eqref{eq:sobxp1} holds, the proof being based on the following proposition:
\begin{proposition}\label{prop:key}
Let $[a,b]\subset(0,\Rb)$ and  $\ppi$ be a test plan on $\X$ such that $\sfd(\gamma_t,\O)\in[a,b] $ for every $t\in[0,1]$ and $\ppi$-a.e.\ $\gamma$. Then for $\ppi$-a.e.\ $\gamma$ the curve $\tilde\gamma:=\pr\circ\gamma$ is absolutely continuous and satisfies
\[
|\dot{\tilde\gamma}_t|\leq \frac{\Rb}{2\sfd(\gamma_t,\O)}|\dot\gamma_t|,\qquad a.e.\ t.
\]
\end{proposition}
The proof of this proposition is technically quite involved, as it heavily relies on the first and second order differential calculus recently developed in \cite{Gigli14}. We postpone it to the next section, see Proposition \ref{prop:keyagain}, where all the necessary ingredients will be recalled and discussed. Here we show how to use this proposition to get the equality in \eqref{eq:sobxp} and then other basic informations about the structure of $(\X',\sfd',\mm')$:
\begin{theorem}\label{thm:linksob}
Let $[a,b]\subset \,(0,\Rb)$, $h\in \Lip(\R)$ with support in $(0,\Rb)$ and identically 1 on $[a,b]$ and $f\in L^2(\X)$ of the form $f(x)=g(\pr(x))h(\sfd(x,\O))$ for some $g\in L^2(\mm')$. 

Then $f\in W^{1,2}(\X)$ if and only if   $g\in W^{1,2}(\X')$ and in this case we have
\begin{equation}
\label{eq:sobxp}
|\D f|_\X(x)= \frac{\Rb}{2\sfd(x,\O)}|\D g|_{\X'}(\pr(x)) ,\qquad\text{for $\mm$-a.e.\ x such that }\sfd(x,\O)\in[a,b].
\end{equation}
\end{theorem} 
\begin{proof}
The `only if' and the inequality $\geq$ are the content of Proposition \ref{prop:linksob}, so we turn to the `if' and the inequality $\leq$.

Let $[a',b']\subset (0,\Rb)$ be such that $\supp(h)\subset (a',b')$ and notice that by construction we have  $\supp(f)\subset B_{b'}(\O)\setminus B_{a'}(\O)$, thus arguing as in the proof of Theorem 4.19 in  \cite{AmbrosioGigliSavare11-2} to conclude it is sufficient to check the weak upper gradient property for test plans $\ppi$ such that $\sfd(\gamma_t,O)\in[a',b']$ for every $t\in[0,1]$ and  $\gamma\in\supp(\ppi)$.

Fix such $\ppi$, let $G:\X\to\R$ be given by
\[
G(x):= \frac{\Rb}{2\sfd(x,\O)}|\D g|_{\X'}(\pr(x ))h(\sfd(x,\O))+g(\pr(x)) |h'|(\sfd(x,\O)),
\]
and notice that $G$ is in $L^2(\mm)$ and equal to $ \frac{\Rb}{2\sfd(x,\O)}|\D g|_{\X'}(\pr(x)) $ for $x$ such that $\sfd(x,\O)\in[a,b]$. 

Therefore, taking into account Proposition \ref{prop:sobf}, to conclude it is sufficient to prove that  for $\ppi$-a.e.\ $\gamma$ the function $t\mapsto f(\gamma_t)$ is equal a.e.\ on $[0,1]$ and in $\{0,1\}$ to an absolutely continuous map $f_\gamma$ such that 
\begin{equation}
\label{eq:perf}
|\partial_tf_\gamma|(t)\leq G(\gamma_t)|\dot\gamma_t|,\qquad a.e.\ t\in[0,1].
\end{equation}
Notice  that $\pr:B_{b'}(\O)\setminus B_{a'}(\O)\to S_{\Rb/2}(\O)$ is Lipschitz, thus recalling \eqref{eq:relmisure} we deduce that the plan $\ppi':=\pr_*\ppi$ is a test plan on $\X'$  (here we are abusing a bit the notation as we are  interpreting $\pr$ as the map from $C([0,1],\X)$ to $C([0,1],\X')$ sending $\gamma$ to $\pr\circ\gamma$).

Since $g\in W^{1,2}(\X')$, by Proposition \ref{prop:sobf} we deduce that for $\ppi$-a.e.\ $\gamma$ we have that the map $t\mapsto g(\pr(\gamma_t))$ is equal a.e.\ on $[0,1]$ and in $\{0,1\}$ to an absolutely continuous map $g_{\pr\circ\gamma}$ such that $|g'_{\pr\circ\gamma}|(t)\leq |\D g|_{\X'}(\pr(\gamma_t)){\rm ms}_t(\pr\circ\gamma)$. 

Here we use the key Proposition \ref{prop:key}  to obtain that for $\ppi$-a.e.\ $\gamma$ it holds
\[
|g'_{\pr\circ\gamma}|(t)\leq  \frac{\Rb}{2\sfd(\gamma_t,\O)} |\D g|_{\X'}(\pr(\gamma_t))|\dot\gamma_t|.
\]
Since for any absolutely continuous curve $\gamma$ the map $t\mapsto h(\sfd(\gamma_t,\O)$ is absolutely continuous with $|\partial_th(\sfd(\gamma_t,\O)|\leq |h'|(\sfd(\gamma_t,\O)|\dot\gamma_t|$, we deduce  that for $\ppi$-a.e.\ $\gamma$ the map $t\mapsto f(\gamma_t)=g(\pr(\gamma_t))h (\sfd(\gamma_t,\O))$ is equal a.e.\ on $[0,1]$ and in $\{0,1\}$ to the absolutely continuous map $t\mapsto f_\gamma(t):=g_{\pr\circ\gamma}(t)h(\sfd(\gamma_t,\O))$ and that \eqref{eq:perf} holds. By the arbitrariness of $\ppi$, this is sufficient to conclude.
\end{proof}
In \cite{GH15} the notion of `measured-length space' has been introduced as key tool, in conjunction with some doubling property, to establish the Sobolev-to-Lipschitz property of a space and its warped products with an interval. We recall the definition, which consists in a variant of the well-known length property which takes into account the reference measure:
\begin{definition}[Measured-length space]
We say that a metric measure space $(\Z,\d_\Z,\mm_\Z)$ is measured-length if there exists a Borel set $A\subset \Z$ whose complement is $\mm_\Z$-negligible with the following property. For every $x_0,x_1\in A$  there exists $ \eps>0$ such that for every  $\eps_0,\eps_1\in(0,\eps]$ there is a test plan $\pi^{\eps_0,\eps_1}$ with:
\begin{itemize}
\item[a)] the map $(0,\eps]^2\ni(\eps_0,\eps_1)\mapsto \pi^{\eps_0,\eps_1}$ is weakly Borel in the sense that for any $\varphi\in C_b(C([0,1],\Z))$ the map
\[
(0,\eps]^2\ni(\eps_0,\eps_1)\qquad\mapsto \qquad \int \varphi\,\d\pi^{\eps_0,\eps_1},
\]
is Borel.
\item[b)] We have 
\[
(e_0)_*\pi^{\eps_0,\eps_1}= \frac{1_{B_{\eps_0}(x_0)}}{\mm_\Z(B_{\eps_0}(x_0))}\,\mm_\Z,\qquad\text{ and }\qquad (e_1)_*\pi^{\eps_0,\eps_1}= \frac{1_{B_{\eps_1}(x_1)}}{\mm_\Z(B_{\eps_1}(x_1))}\,\mm_\Z
\]
 for every $\eps_0,\eps_1\in(0,\eps]$,
\item[c)] We have
\begin{equation*}
\label{eq:limitenergy}
\lims_{\eps_0,\eps_1\downarrow 0}\iint_0^1|\dot\gamma_t|^2\,\d t\,\d\pi(\gamma)\leq \d_\Z^2(x_0,x_1).
\end{equation*}
\end{itemize}
\end{definition}
We then have the following result:
\begin{proposition}\label{prop:basexp}
$(\X',\sfd',\mm')$ is infinitesimally Hilbertian, doubling and a measured-length space.
\end{proposition}
\begin{proof}

\noindent{\bf Infinitesimal Hilbertianity.} Direct consequence of Theorem \ref{thm:linksob} and the infinitesimal Hilbertianity of $\X$ (recall also property \eqref{eq:hilbloc}). 

\noindent{\bf Doubling.} We shall denote by $B^\X$, $B^{\X'}$ balls in $\X$, $\X'$ respectively. Start noticing that being $(\X',\sfd')$ compact it is sufficient to prove that for some $c>0$ we have
\begin{equation}
\label{eq:claimdoub}
\mm'(B^{\X'}_{2r}(x'))\leq c\, \mm'(B^{\X'}_{r}(x')),\qquad\forall x'\in \X',\ r< \Rb/8.
\end{equation}
Then for $x'\in \X'\subset \X$ and $r\in(0,\Rb/8)$ define $A(x',r)\subset \X$ as
\begin{equation}
\label{eq:axr}
A(x',r):=\big\{x\in \X\ :\ \sfd(x,\O)\in[\Rb/2-r,\Rb/2+r],\ \sfd'(\pr(x),x')\leq r\big\}
\end{equation}
and notice that from Corollary \ref{cor:contdis} we see that 
\[
\mm(A(x',r))=N\mm(B_\Rb(\O))\mm'(B^{\X'}_r(x'))\int_{\Rb/2-r}^{\Rb/2+r}s^{N-1}\,\d s
\]
and therefore for some constants $c_1,c_2>0$ we have
\begin{equation}
\label{eq:pallaquadrata}
c_1r\,\mm'(B^{\X'}_r(x'))\leq\mm(A(x',r))\leq c_2r \,\mm'(B^{\X'}_r(x')),\qquad\forall x'\in \X',\ r< \Rb/8,
\end{equation}
while the construction ensures that
\begin{equation}
\label{eq:pallafuori}
A(x',r)\subset B^\X_{2r}(x')\qquad\forall x'\in \X',\ r< \Rb/8.
\end{equation}
{ Recall that  $\pr:B_{3\Rb/4}(\O)\setminus B_{\Rb/4}(\O)\to S_{\Rb/2}(\O)$ is Lipschitz and let $L$ be a bound on its Lipschitz constant}. Observe also that the triangle inequality ensures that a geodesic with endpoints in $B_{2r}(x')$ for some $x'\in \X'$ and $r< \Rb/8$ never leaves $B_{3\Rb/4}(\O)\setminus B_{\Rb/4}(\O)$, which is sufficient to deduce that
\[
\sfd'(\pr(x_1),\pr(x_2))\leq L\sfd(x_1,x_2),\qquad\forall x_1,x_2\in B_{2r}(x'),\ x'\in \X',\ r<\Rb/8.
\]
This fact together with \eqref{eq:pallafuori} grants that
\begin{equation}
\label{eq:palladentro}
B_{r/L}^\X(x')\subset A(x',r),\qquad\forall x'\in \X',\ r<\Rb/8.
\end{equation}
Then the claim \eqref{eq:claimdoub} follows from \eqref{eq:pallaquadrata}, \eqref{eq:pallafuori} and \eqref{eq:palladentro} taking into account that $(\X,\sfd,\mm)$ is doubling.

\noindent{\bf  Measured-length property.} Fix $x^0,x^1\in \X'$, put $\eps:=\min\{\tfrac19,\tfrac{\Rb^2}{16(1+\sfd'(x^0,x^1))^2}\}$ and pick $\eps_0,\eps_1\in(0,\eps)$. Put $\eps_{01}:=\max\{\eps_0,\eps_1\}$, let $\gamma$ be a geodesic in $\X'$ connecting $x^0$ to $x^1$, let $n$ be the integer part of $1+\frac1{\sqrt{\eps_{01}}}$ and for $i=0,\ldots,n$ put $x_{\eps_{01},i}:=\gamma_{\frac in}$. Notice that
\[
\sum_{i=0}^{n-1}\sfd(x_{\eps_{01},i},x_{\eps_{01},i+1})\leq\sfd'(x^0,x^1),\qquad\forall n\in\N.
\]
For $x\in \X'$ and $r>0$ consider the sets $A(x,r)$ as defined in \eqref{eq:axr}, then put $\eps_{i}:=\eps_0+\tfrac{i}{n}(\eps_1-\eps_0)$ and define the measures 
\[
\mu^{\eps_0,\eps_1}_{i}:=\frac1{\mm(A(x_{\eps_{01},i},\eps_i))}\mm\restr{A(x_{\eps_{01},i},\eps_i)}\in\prob{\X}.
\]
From Corollary \ref{cor:contdis} it follows that
\begin{equation}
\label{eq:bla}
\pr_*\mu^{\eps_0,\eps_1}_{i}=\frac1{\mm'(B_{\eps_i}(x_{\eps_{01},i}))}\mm'\restr{B_{\eps_i}(x_{\eps_{01},i})},
\end{equation}
the balls considered in the right hand side being in the space $(\X',\sfd')$.

For $i=0,\ldots,n-1$ let $\ppi^{\eps_0,\eps_1}_{i}$ be  the only optimal geodesic plan from $\mu^{\eps_0,\eps_1}_{i}$ to $\mu^{\eps_0,\eps_1}_{i+1}$ (recall \cite{GigliRajalaSturm13}). Taking into account that the distance between a point in $A(x_{\eps_{01},i},\eps_i)$ and a point in $A(x_{\eps_{01},i+1},\eps_{i+1})$ is bounded above by $4\eps_{01}+\tfrac{\sfd'(x^0,x^1)}{n}$ we have
\begin{equation}
\label{eq:perml}
\iint_0^1|\dot\gamma_t|^2\,\d t\,\d\ppi^{\eps_0,\eps_1}_{i}(\gamma)=W_2^2(\mu^{\eps_0,\eps_1}_{i},\mu^{\eps_0,\eps_1}_{i+1})\leq \big(4\eps_{01}+\tfrac{\sfd'(x^0,x^1)}{n}\big)^2
\end{equation}
and from the construction and the choice of $\eps,\eps_0,\eps_1$ it is also easy to see that
\begin{equation}
\label{eq:distcentro}
\sfd(\gamma_t,\O)\in \big[\tfrac \Rb2-\sqrt{\eps_{01}}(1+\sfd'(x^0,x^1)),\tfrac \Rb2+\sqrt{\eps_{01}}(1+\sfd'(x^0,x^1))\big]\subset \big[\tfrac\Rb4,\tfrac{3\Rb}4\big],
\end{equation}
for every $t\in[0,1]$ and $\ppi^{\eps_0,\eps_1}_{i}$-a.e.\ $\gamma$.

 With a gluing argument we can then build a plan $\ppi^{\eps_0,\eps_1}\in\prob{C([0,1],\X)}$ (which is in fact unique, being the $\ppi^{\eps_0,\eps_1}_{i}$'s induced by maps) such that
\[
\Big({\rm Restr}_{\frac in}^{\frac{i+1}n}\Big)_*\ppi^{\eps_0,\eps_1}_n=\ppi^{\eps_0,\eps_1}_{n,i},\qquad\forall i=0,\ldots, n-1,
\]
and property \eqref{eq:perml}, taking into account the rescaling factor, gives
\[
\begin{split}
\iint_0^1|\dot\gamma_t|^2\,\d t\,\d\ppi^{\eps_0,\eps_1}(\gamma)&=n\sum_{i=0}^n\iint_0^1|\dot\gamma_t|^2\,\d t\,\d\ppi^{\eps_0,\eps_1}_{i}(\gamma)\\
&\leq(4n\eps_{01}+\sfd'(x^0,x^1))^2\leq(8\sqrt{\eps_{01}}+\sfd'(x^0,x^1))^2,
\end{split}
\]
while the construction ensures that \eqref{eq:distcentro} holds also $\ppi^{\eps_0,\eps_1}$-a.e.\ $\gamma$ for every $t\in[0,1]$.

We now put
\[
\bar\ppi^{\eps_0,\eps_1}:=\pr_*\ppi^{\eps_0,\eps_1}\in\prob{C([0,1],\X')}
\]
and notice that:
\begin{itemize}
\item[-] Since $\ppi^{\eps_0,\eps_1}$ is a test plan on $\X$ for which \eqref{eq:distcentro} holds, by Corollary \ref{cor:contdis}, property \eqref{eq:relmisure} and the fact that $\pr$ is Lipschitz from $B_{3\Rb/4}(\O)\setminus B_{\Rb/4}(\O)$ to $\X'$ we deduce that $\bar\ppi^{\eps_0,\eps_1}$ is a test plan on $\X'$.
\item[-] By \eqref{eq:bla} it follows that
\[
(\e_0)_*\bar\ppi^{\eps_0,\eps_1}=\frac1{\mm(B_{\eps_0}(x^0))}\mm\restr{B_{\eps_0}(x^0)},\qquad\qquad(\e_0)_*\bar\ppi^{\eps_0,\eps_1}=\frac1{\mm(B_{\eps_1}(x^1))}\mm\restr{B_{\eps_1}(x^1)}.
\]
\item[-] From Proposition \ref{prop:key} we have that 
\[
\iint_0^1|\dot\gamma_t|^2\,\d t\,\d\bar\ppi^{\eps_0,\eps_1}(\gamma)\leq\frac{\Rb^2}{4}\iint_0^1\frac{|\dot\gamma_t|^2}{\sfd^2(\gamma_t,\O)}\,\d t\,\d\ppi^{\eps_0,\eps_1}(\gamma)
\]
and therefore using \eqref{eq:perml} and \eqref{eq:distcentro} we obtain
\[
\lims_{\eps_0,\eps_1\downarrow0}\iint_0^1|\dot\gamma_t|^2\,\d t\,\d\bar\ppi^{\eps_0,\eps_1}(\gamma)\leq
\lims_{\eps_0,\eps_1\downarrow0}\frac{\Rb^2(8\sqrt{\eps_{01}}+\sfd'(x^0,x^1))^2}{4(\frac{\Rb}2+\sqrt{\eps_{01}}(1+\sfd'(x^0,x^1)))^2}=\sfd'(x^0,x^1)^2.
\]
\end{itemize}
Since the Borel dependency of $\bar\ppi^{\eps_0,\eps_1}$ from $\eps_0,\eps_1$ is clear from the construction, the proof is achieved.
\end{proof}

\subsection{Estimate on the speed of the projection}\label{se:speedproj}
This section is devoted to the proof of Proposition \ref{prop:key}, which, as already mentioned, relies on some definitions and results contained in \cite{Gigli14}. In order to keep the presentation at reasonable length, we shall assume the reader familiar with the language developed in \cite{Gigli14} and in particular with the concept of $L^2$-normed $L^\infty$-module and related objects. We shall use the next subsection to recall the basic results we shall need, also in order to fix the notation. The subsequent  one will then be devoted to the proof of Proposition \ref{prop:key}.
\subsubsection{Tools for differential calculus on metric measure spaces}\label{se:tools}
\paragraph{(co)tangent vectors and speed of test plans.}
The tangent and cotangent modules of the metric measure space $(\X,\sfd,\mm)$ are denoted as $L^2(T\X)$ and $L^2(T^*\X)$ respectively. The pointwise norm on both spaces will be denoted by $|\cdot|$.

The differential of a function $f\in W^{1,2}(\X)$ is denoted by $\d f$ and is an element of $L^2(T^*\X)$ which, among other properties, satisfies
\[
|\d f|=|\D f|,\quad\mm-a.e..
\]
Seen as unbounded operator from $L^2(\X)$ to $L^2(T^*\X)$, the differential is a linear and closed operator. See Section 2.2 of \cite{Gigli14}  for details. In case $\X$ is infinitesimally Hilbertian, the gradient $\nabla f\in L^2(T\X)$ of $f\in W^{1,2}(\X)$ is the element associated to the differential $\d f$ via the Riesz isomorphism for modules.

Given another metric measure space $(\Y,\sfd_\Y,\mm_\Y)$, a map ${\rm p}:\Y\to \X$ is said of bounded compression provided ${\rm p}_*\mm_\Y\leq C\mm$ for some $C>0$ and for such map it is possible to introduce the pullback $L^2(T\X,{\rm p},\mm_\Y)$ of the tangent module $L^2(T\X)$ and the pullback operator ${\rm p}^*:L^2(T\X)\to L^2(T\X,{\rm p},\mm_\Y)$. These are characterised, up to isomorphism, by the fact that ${\rm p}^*:L^2(T\X)\to L^2(T\X,{\rm p},\mm_\Y)$ is linear, satisfying
\[
|{\rm p}^*v|=|v|\circ{\rm p},\quad\mm_\Y-a.e.\qquad\forall v\in L^2(T\X),
\]
and with image which generates, in the sense of modules, the whole $L^2(T\X,{\rm p},\mm_\Y)$. A similar construction can be made for the cotangent module and there is a (unique) natural duality relation  $L^2(T^*\X,{\rm p},\mm_\Y)\times L^2(T\X,{\rm p},\mm_\Y)\to L^1(\Y)$ which is $L^\infty(\Y)$-bilinear, continuous and such that
\begin{equation}
\label{eq:dualpb}
{\rm p}^*\omega({\rm p}^*v)=\omega(v)\circ{\rm p},\quad\mm_\Y-a.e.\qquad \forall v\in L^2(T\X),\ \omega\in L^2(T^*\X).
\end{equation}
See Section 1.6 of \cite{Gigli14} for details.

Given a test plan $\ppi$, we shall consider such pullback construction for $\Y=C([0,1],\X)$ equipped with the sup distance and  $\ppi$ as reference measure. The maps of bounded compression of interest for us are the evaluation maps $\e_t$.

If $(\X,\sfd,\mm)$ is infinitesimally Hilbertian, as in our case, it turns out that for a.e.\ $t\in[0,1]$ there exists a unique element $\ppi'_t\in L^2(T\X,\e_t,\ppi)$, called  speed of $\ppi$ at time $t$, having the property that
\[
\lim_{h\to 0}\frac{f\circ\e_{t+h}-f\circ\e_t}h=(\e_t^*\d f)(\ppi_t'),\qquad\forall f\in W^{1,2}(\X),
\]
the limit being intended in the strong topology of $L^1(\ppi)$. See Theorem 2.3.18 in \cite{Gigli14}. The same theorem also provides a direct and tight link between the pointwise norm of the vector fields $\ppi'_t$ and the metric speed of curves,  as for a.e.\ $t\in[0,1]$ it holds
\begin{equation}
\label{eq:linkspeedmetric}
|\ppi'_t|(\gamma)=|\dot\gamma_t|,\qquad\ppi-a.e.\ \gamma.
\end{equation}

\paragraph{Maps of bounded deformation and their differential.} Maps between metric measure spaces which are both Lipschitz and of bounded compression will be called of bounded deformation. 

The right composition with a map of bounded deformation $F:\X\to \Y$ provides a linear and continuous map from $W^{1,2}(\Y)$ to $W^{1,2}(\X)$ and for any $\varphi\in W^{1,2}(\Y)$ the bound
\begin{equation}
\label{eq:normdiffdef}
|\d(\varphi\circ F)|\leq \Lip(F)|\d\varphi|\circ F,\quad\mm_\X-a.e.,
\end{equation}
holds. 

If $F$ is invertible with  inverse  also bounded deformation, then the differential $\d F$ is a well defined linear continuous map from $L^2(T\X)$ to $L^2(T\Y)$: for $v\in L^2(T\X)$ the vector field $\d F(v)\in L^2(T\Y)$ is charaterized by
\begin{equation}
\label{eq:defdiff}
\Big(\d\varphi\big(\d F(v)\big)\Big)\circ F=\d(\varphi\circ F)(v),\quad\mm_\X-a.e.\ \qquad \forall \varphi\in W^{1,2}(\Y),
\end{equation}
and the bound \eqref{eq:normdiffdef} yields
\begin{equation}
\label{eq:normdiffdef2}
|\d F(v)|\circ F\leq\Lip(F) |v|,\quad\mm_\X-a.e.,
\end{equation}
and in particular the differential $\d F$ is local in the sense that
\begin{equation}
\label{eq:locdiff}
\d F(v)=\d F(w),\quad\mm_\Y-a.e.\ on\ F(\{v=w\}).
\end{equation}
See Section 2.4 of \cite{Gigli14} for details.

The left composition with $F$ provides a Lipschitz map from $C([0,1],\X)$ to $C([0,1],\Y)$ which, abusing a bit the notation, we shall continue to denote by $F$. Hence for a given test plan $\ppi$ on $\X$ we can consider the measure $F_*\ppi$ on $C([0,1],\Y)$ and it is trivial to see that the fact that $F$ is of bounded deformation ensures that $F_*\ppi$ is a test plan on $\Y$.

To clarify the notation in the foregoing discussion, we shall put $\bar\ppi:=F_*\ppi$ and  denote by $\bar\e_t$ the evaluation maps from $C([0,1],\Y)$ to $\Y$. We shall assume that $F$ is of bounded deformation, invertible and with inverse of bounded deformation.

Notice that that for every $t\in[0,1]$, the differential $\d F:L^2(T\X)\to L^2(T\Y)$ naturally induces a map, which we shall still denote $\d F$, from $L^2(T\X,\e_t,\ppi)$ to $L^2(T\X,\bar\e_t,\bar\ppi)$: it is the unique linear and continuous map such that
\begin{equation}
\label{eq:liftdF}
\begin{split}
\d F(\e_t^*v)&=\bar \e_t^*(\d F(v)),\qquad \forall v\in L^2(T\X),\\
\d F(g V)&=g\circ F^{-1}\,\d F(V),\qquad\forall V\in L^2(T\X,\e_t,\ppi),\ g\in L^\infty(\ppi).
\end{split}
\end{equation}
Recalling that in the classical smooth setting we have the chain rule
\[
(F\circ\gamma)'_t=\d F(\gamma_t'),
\]
we are now going to show that the language just discussed allows to state and prove an analogous of this chain rule in the context of metric measure spaces. As we shall see, the  proof is being just based on keeping track of the various definitions.
\begin{proposition}[Chain rule for speeds]
Let $F:\X\to \Y$ be of bounded deformation, invertible and with inverse of bounded deformation. Then for every test plan  $\ppi$  on $\X$ we have
\begin{equation}
\label{eq:speedaftercompos}
(F_*\ppi)'_t=(\d F)(\ppi'_t),\qquad a.e.\ t\in[0,1].
\end{equation}
\end{proposition}
\begin{proof}
Put $\bar\ppi:=F_*\ppi$ as before and fix $f\in W^{1,2}(\X)$. We claim that
\begin{equation}
\label{eq:pullb2}
(\bar\e_t^*\d f)\big((\d F)(V)\big)=\Big(\big(\e_t^*\d(f\circ F)\big)(V)\Big)\circ F^{-1},\qquad\forall V\in  L^2(T\X,\e_t,\ppi).
\end{equation}
For $V$ of the form $\e_t^*v$ for some $v\in L^2(T\X)$ such identity comes from the chain of equalities

\begin{align*}
(\bar\e_t^*\d f)\big((\d F)(\e_t^*v)\big)&=(\bar\e_t^*\d f)\big(\bar \e_t^*(\d F(v))\big)&&\qquad\text{by the first in }\eqref{eq:liftdF}\\
&=\big(\d f(\d F(v))\big)\circ \bar\e_t&&\qquad\text{by }\eqref{eq:dualpb}\\
&=\d(f\circ F)(v)\circ F^{-1}\circ \bar\e_t&&\qquad\text{by }\eqref{eq:defdiff}\\
&=\big(\d(f\circ F)(v)\big)\circ\e_t\circ F^{-1}&&\qquad\text{because }F^{-1}(\gamma_t)=(F^{-1}(\gamma))_t\\
&=\Big(\big(\e_t^*\d(f\circ F)\big)(\e_t^*v)\Big)\circ F^{-1}&&\qquad\text{by }\eqref{eq:dualpb}.
\end{align*}
Therefore noticing that both sides of \eqref{eq:pullb2} are, as functions of $V$, linear and continuous from $ L^2(T\X,\e_t,\ppi)$ to $L^1(\bar \ppi)$, and since the vector space generated by elements of the form $g\e_t^*v$ for $v\in L^2(T\X)$ and $g\in L^\infty(\ppi)$ is dense in $ L^2(T\X,\e_t,\ppi)$, the claim \eqref{eq:pullb2} follows if we show that 
\[
\begin{split}
(\bar\e_t^*\d f)\big((\d F)(gV)\big)&=g\circ F^{-1}\,(\bar\e_t^*\d f)\big((\d F)(V)\big),\\
\Big(\big(\e_t^*\d(f\circ F)\big)(gV)\Big)\circ F^{-1}&=g\circ F^{-1}\Big(\big(\e_t^*\d(f\circ F)\big)(V)\Big)\circ F^{-1}.
\end{split}
\]
The second of these is obvious from the $L^\infty(\ppi)$-linearity of $\e_t^*\d(f\circ F)$ as a map from $ L^2(T\X,\e_t,\ppi)$ to $L^1(\ppi)$. For the first we use the second in \eqref{eq:liftdF} and the $L^\infty(\bar\ppi)$-linearity of $\bar\e_t^*\d f$ as a map from $ L^2(T\Y,\bar \e_t,\bar \ppi)$ to $L^1(\bar\ppi)$:
\[
\begin{split}
(\bar\e_t^*\d f)\big((\d F)(gV)\big)=(\bar\e_t^*\d f)\big(g\circ F^{-1}(\d F)(V)\big)=g\circ F^{-1}\,(\bar\e_t^*\d f)\big((\d F)(V)\big).
\end{split}
\]
Thus \eqref{eq:pullb2} is proved and writing it for $\ppi'_t$ in place of $V$ we obtain
\[
(\bar\e_t^*\d f)\big((\d F)(\ppi'_t)\big)=\Big(\big(\e_t^*\d(f\circ F)\big)(\ppi'_t)\Big)\circ F^{-1},\qquad a.e.\ t\in[0,1].
\]
To conclude, recall that $\big(\e_t^*\d(f\circ F)\big)(\ppi'_t)$ is the strong limit in $L^1(\ppi)$ as $h\to 0$ of the maps 
\[
\gamma\quad\mapsto\quad \frac{f(F(\gamma_{t+h}))-f(F(\gamma_t))}h
\]
and therefore the change of variable $\gamma=F^{-1}(\bar\gamma)$ and the fact that $\ppi=F^{-1}_*\bar\ppi$ show that $\Big(\big(\e_t^*\d(f\circ F)\big)(\ppi'_t)\Big)\circ F^{-1}$ is the strong limit in $L^1(\bar\ppi)$ as $h\to 0$ of the maps 
\[
\bar\gamma\quad\mapsto\quad\frac{f(\bar\gamma_{t+h})-f(\bar\gamma_t)}h,
\]
but this latter limit is, by the very definition of $\bar\ppi'_t$, equal to $\bar\e_t^*\d f(\bar\ppi'_t)$. We therefore proved that
\[
(\bar\e_t^*\d f)\big((\d F)(\ppi'_t)\big)=(\bar \e_t^*\d f)(\bar\ppi)'_t,\qquad a.e.\ t\in[0,1],
\]
which, by the arbitrariness of $f\in W^{1,2}(\X)$,  is the thesis.
\end{proof}
\begin{remark}{\rm
In all this discussion the assumption that $F$ was invertible with inverse of bounded deformation is not really necessary, being the differential $\d F$ of a map of bounded deformation  $F:\X\to \Y$ always well defined as map from $L^2(T\X)$ to the pullback $L^2(T\Y,F,\mm_\X)$ of the tangent module $L^2(T\Y)$ of $\Y$ via $F$ (see Section 2.4 in \cite{Gigli14}). 

We made  this further assumption because it simplifies the exposition and will be present in our applications.
}\fr\end{remark}

\paragraph{Bits of second order calculus}
Here we come back to our assumption that  $(\X,\sfd,\mm)$ is a $\RCD^*(0,N)$ space.

Test functions and the language of $L^\infty$-modules allow to introduce the second order Sobolev space $W^{2,2}(\X)$ as follows.  First of all, we recall that being $L^2(T^*\X)$ an Hilbert module, it is possible to consider the Hilbert tensor product of $L^2(T^*\X)$ with itself, which we shall denote by $L^2((T^*)^{\otimes 2}\X)$ (see Section 1.5 in \cite{Gigli14} for the definition). We remark that in the smooth case the pointwise norm in $L^2((T^*)^{\otimes 2}\X)$ is the Hilbert-Schmidt one.

Then we say that a function $f\in W^{1,2}(\X)$ belongs to $W^{2,2}(\X)$ provided there is an element of $L^2((T^*)^{\otimes 2}\X)$, called the Hessian of $f$ and denoted by $\H f$, such that for any $g_1,g_2,h\in\test \X$ it holds
\begin{equation}
\label{eq:defhess}
\begin{split}
2\int h\H f(\nabla g_1,\nabla g_2)\,\d\mm=\int- \la\nabla f,\nabla g_1\ra{\rm div}(h\nabla g_2)- &\la\nabla f,\nabla g_2\ra{\rm div}(h\nabla g_1)\\
&- h\la\nabla f,\nabla(\la\nabla g_1,\nabla g_2)\ra\,\d\mm.
\end{split}
\end{equation}
The density of $\test \X$ in $W^{1,2}(\X)$ grants that the above uniquely characterises  $\H f$. It is then possible to see that $W^{2,2}(\X)$ equipped with the norm
\[
\|f\|_{W^{2,2}}^2:=\int |f|^2+|\d f|^2+|\H f|^2\,\d\mm,
\]
is a separable Hilbert space.

An important inequality is
\[
\int |\H f|^2\,\d\mm\leq \int |\Delta f|^2-K|\d f|^2\,\d\mm,
\]
which shows in particular that $D(\Delta)\subset W^{2,2}(\X)$ and thus that $W^{2,2}(\X)$ is dense in $W^{1,2}(\X)$.

The natural chain and Leibniz rules for the Hessian are in place. In particular, if $f\in W^{2,2}(\X)$ is bounded and Lipschitz and $\varphi\in C^{2}_c(\R)$ then $\varphi\circ f\in W^{2,2}(\X)$ and
\[
\H{\varphi\circ f}=\varphi''\circ f\d f\otimes\d f+\varphi'\circ f\H f,
\] 
and if $f_1,f_2\in W^{2,2}(\X)$ are bounded and Lipschitz, then $f_1f_2\in W^{2,2}(\X)$ and
\[
\H{f_1f_2}=f_2\H{f_1}+f_1\H{f_2}+\d f_1\otimes\d f_2+\d f_2\otimes\d f_1.
\]
Moreover, for $f_1,f_2\in\test \X$ we have that $\la\nabla f_1,\nabla f_2\ra\in W^{1,2}(\X)$ with
\begin{equation}
\label{eq:diff2}
\d\la\nabla f_1,\nabla f_2\ra=\H{f_1}(\nabla f_2)+\H{f_2}(\nabla f_1).
\end{equation}
See Section 3.3 of \cite{Gigli14} for more details and more general versions of these calculus rules.

Given an open set $\Omega$  we introduce the space $W^{2,2}(\Omega)$ as the subspace of $W^{1,2}_{loc}(\X)$ made of functions $f$ for which there is $\H f\in L^2((T^*)^{\otimes 2}\X)$ such that  \eqref{eq:defhess} holds for any $g_1,g_2,h\in\test \X$ with support in $\Omega$. In this case, evidently, $\H f$ is uniquely characterized only $\mm$-a.e.\ on $\Omega$. 

We conclude this preliminary section showing that 
\begin{equation}
\label{eq:hessid}
\b\in W^{2,2}(B_\Rb(\O))\qquad\text{with}\qquad\H\b=\Id\quad \mm-a.e.\ on\ B_\Rb(\O),
\end{equation} 
where $\Id\in L^\infty((T^*)^{\otimes 2}\X)$ is defined by $\Id(v,w)=\la v,w\ra$ $\mm$-a.e.\ for any $v,w\in L^2(T\X)$. { For let us fix \(\bar \Rb<\Rb\) and  recall the that  the function \(\bar \b=\varphi\circ\b\) introduced in  Section \ref{se:gf} belongs to \(\test \X\) and that  $\bd \bar \b=N\mm$ on $B_{\bar \Rb}(\O)$.  Let  $g,h\in\test \X$ with support in $B_{\bar \Rb}(\O)$, using the Euler equation \eqref{eq:euler} and few integration by parts we obtain:
}
\[
\begin{split}
\int&-{\rm div}(h\nabla g)\la\nabla\bar\b,\nabla g\ra\,\d\mm\\
&=\int-\Delta(hg)\la\nabla\bar\b,\nabla g\ra+{\rm div}(g\nabla h)\la\nabla\bar\b,\nabla g\ra\,\d\mm\\
&=\int -hg\la\nabla\Delta g,\nabla\bar\b\ra-2hg\Delta g+\la\nabla h,\nabla g\ra\la\nabla\bar \b,\nabla g\ra+g\Delta h\la\nabla\bar\b,\nabla g\ra\,\d\mm\\
&=\int h\Delta g\la\nabla g,\nabla\bar \b\ra+g\Delta g\la\nabla h,\nabla \bar\b\ra+(N-2)hg\Delta g+\la\nabla h,\nabla g\ra\la\nabla \bar\b,\nabla g\ra+\Delta h\la\nabla\bar\b,\nabla\tfrac{ g^2}2\ra\,\d\mm\\
&=\int {\rm div}(h\nabla g)\la\nabla\bar\b,\nabla g\ra+g\Delta g\la\nabla h,\nabla\bar \b\ra+(N-2)hg\Delta g+\Delta h\la\nabla\bar\b,\nabla\tfrac{ g^2}2\ra\,\d\mm.
\end{split}
\]
Similarly, using \eqref{eq:euler} in the third equality we get
\[
\begin{split}
\int&-h\la \nabla\bar \b,\nabla\tfrac{|\D g|^2}{2}\ra\,\d\mm\\
&=\int\tfrac{|\D g|^2}2\la\nabla h,\nabla\bar\b\ra+Nh\tfrac{|\D g|^2}2\,\d\mm\\
&=\int\Delta\tfrac{g^2}2\la\nabla h,\nabla\bar\b\ra-\tfrac{|\D g|^2}2\la\nabla h,\nabla\bar\b\ra-g\Delta g\la\nabla h,\nabla\bar \b\ra+Nh\tfrac{|\D g|^2}2\,\d\mm\\
&=\int\tfrac{g^2}2\la\nabla\Delta h,\nabla\bar \b\ra+g^2\Delta h-\tfrac{|\D g|^2}2\la\nabla h,\nabla\bar\b\ra-g\Delta g\la\nabla h,\nabla \b\ra+Nh\tfrac{|\D g|^2}2\,\d\mm\\
&=\int-\Delta h\la\nabla\tfrac{g^2}2,\nabla\bar\b\ra-N\tfrac{g^2}2\Delta h+g^2\Delta h-\tfrac{|\D g|^2}2\la\nabla h,\nabla\bar\b\ra-g\Delta g\la\nabla h,\nabla \bar\b\ra+Nh\tfrac{|\D g|^2}2\,\d\mm\\
&=\int-\Delta h\la\nabla\tfrac{g^2}2,\nabla\bar\b\ra+(-\tfrac N2+1)g^2\Delta h+h\la\nabla\bar\b,\nabla\tfrac{|\D g|^2}2\ra-g\Delta g\la\nabla h,\nabla \bar\b\ra+Nh|\D g|^2\,\d\mm.
\end{split}
\]
Adding up  we get
\[
\begin{split}
\int&-{\rm div}(h\nabla g)\la\nabla\bar\b,\nabla g\ra-h\la \nabla \bar\b,\nabla\tfrac{|\D g|^2}{2}\ra\,\d\mm\\
&=\int{\rm div}(h\nabla g)\la\nabla\bar\b,\nabla g\ra+h\la \nabla\bar \b,\nabla\tfrac{|\D g|^2}{2}\ra+ (N-2)hg\Delta g+(-\tfrac N2+1)g^2\Delta h+Nh|\D g|^2\,\d\mm\\
&=\int{\rm div}(h\nabla g)\la\nabla\bar\b,\nabla g\ra+h\la \nabla\bar \b,\nabla\tfrac{|\D g|^2}{2}\ra+2h|\D g|^2\,\d\mm,
\end{split}
\]
or equivalently
\[
\int-{\rm div}(h\nabla g)\la\nabla\bar\b,\nabla g\ra-h\la \nabla\bar \b,\nabla\tfrac{|\D g|^2}{2}\ra\,\d\mm=\int h|\D g|^2\,\d\mm,
\]
which by a polarization argument gives {\(\H{\bar\b}=\Id\) in \(B_{\bar \Rb}(\O)\) from which \eqref{eq:hessid} follows.
}

\subsubsection{Proof of the estimate}

Fix $\rb',\Rb'\in(0,\Rb)$ so that $\rb'<\frac\Rb2<\Rb'$. Later on $\rb',\Rb'$ will be sent to 0 and $\Rb$ respectively, but for the moment it is convenient to keep them fixed and to avoid mentioning the dependence on them of the various objects we are going to build.

Pick a function $\psi\in C^\infty(\R)$ with support in $(0,\tfrac{\Rb^2}2)$ so that 
\[
\psi(z)=\frac12\big(\sqrt{2z}-\frac\Rb2\big)^2=z-\frac{\Rb}{\sqrt 2}\sqrt z+\frac{\Rb^2}8\qquad\text{ for }z\in[\tfrac{\rb'^2}2,\tfrac{\Rb'^2}2]
\]
and define the reparametrization function ${\rm rep}:(\R^+)^2\to\R+$ by requiring that 
\[
\partial_t{\rm rep}_t(r)=\psi'\Big(\frac{r^2}{2}e^{-2{\rm rep}_t(r)}\Big),\qquad\qquad {\rm rep}_0(r)=0,
\]
for every $r\geq 0$. Then define the function $\hat \b:\X\to\R$ and the flow $\hat\gl:\R^+\times \X\to \X$ as
\[
\begin{split}
\hat\b&:=\psi\circ\b,\\
\hat\gl_s(x)&:=\gl_{{\rm rep}^x_s}(x),\qquad\forall s\in\R^+,\ x\in \X.
\end{split}
\]
The following properties of $\hat\b$ and $\hat\gl$ are direct consequence of the definitions, the analogous ones of $\b,\gl$ and minor algebraic manipulation, see also the proof of Proposition \ref{prop:modifiedflow}:
\begin{itemize}
\item[a)] From the chain rules for the gradient, Laplacian and Hessian we see that $\hat\b\in\test \X$ and 
\[
\begin{split}
\nabla\hat\b&=\psi'\circ\b\,\nabla\b,\\
\Delta\hat\b&=\psi'\circ\b\,\Delta\b+\psi''\circ\b|\nabla\b|^2,\\
\H{\hat\b}&=\psi'\circ\b\,\H\b+\psi''\circ\b\nabla\b\otimes\nabla\b.
\end{split}
\]
In particular, recalling that $\bd\b=N\mm$ and $\H\b=\Id$ on $B_\Rb(\O)$ we see that $\hat\b$ has bounded gradient, Laplacian and Hessian. Moreover, on $B_{\Rb'}(\O)\setminus B_{\rb'}(\O)$ it holds
\begin{equation}
\label{eq:formulehatb}
\begin{split}
\hat\b&=\frac12\big(\sfd(\cdot,\O)-\frac\Rb2\big)^2=\b-\frac{\Rb}{\sqrt 2}\sqrt \b+\frac{\Rb^2}8,\\
\nabla\hat\b&=\nabla \b\Big(1-\frac{\Rb}{2\sqrt {2\b}}\Big),\\
\H{\hat\b}&=\Id \Big(1-\frac{\Rb}{2\sqrt{2\b}}\Big)+\frac{\Rb}{4\sqrt 2}\frac{1}{\b\sqrt\b}\nabla\b\otimes\nabla\b,\\
\end{split}
\end{equation}
\item[b)] for any $x\in \X$ the curve $[0,1]\ni s\mapsto \hat\gl_s(x)\in \X$ is Lipschitz and satisfies
\[
\b(\gamma_0)=\b(\gamma_t)+\frac12\int_0^t|\dot\gamma_s|^2+\lip^2(\hat\b)(\gamma_s)\,\d s,\qquad\forall t\in[0,1].
\]
In particular, using the fact that $\lip^2(\hat\b)=2\hat\b$, which follows from the anlogous property of $\b$, we see that $\partial_t\hat\b(\hat\gl_t(x))=2\hat\b(\hat\gl_t(x))$. Therefore for $x\in B_{\Rb'}(\O)\setminus B_{\rb'}(\O)$  the quantity $|\sfd(\hat\gl_t(x),\O)-\tfrac\Rb2|$ decreases exponentially as $t\to\infty$ and in particular
\begin{equation}
\label{eq:expr3}
\hat\gl_s\to\pr\qquad\text{ uniformly on $B_{\Rb'}(\O)\setminus B_{\rb'}(\O)$  as $s\to\infty$.}
\end{equation}
Finally, the fact that $|\dot\gamma_t|=\lip(\hat\b)(\gamma_t)\leq \Lip(\hat \b)$ shows that $t\mapsto\hat\gl_t(x)$ is $\Lip(\hat\b)$-Lipschitz.
\item[c)] For any $s\geq 0$, the map $\hat\gl_s$ is  invertible map from $\X$ into itself, it is the identity on a neighbourhood of $\{\O\}\cup(\X\setminus B_\Rb(\O))$ and sends $B_{\Rb'}(\O)\setminus B_{\rb'}(\O)$ into itself.
\item[d)] for any $t,s\in\R^+$ we have 
\[
\hat\gl_s\circ\hat\gl_t=\hat\gl_{s+t}
\]
and in particular putting $\hat\gl_{-s}:=\hat\gl_s^{-1}$ we obtain a one parameter group of maps of $\X$ into itself.
\item[e)] The same arguments used to obtain the bound \eqref{eq:bc} grant that
\begin{equation*}
\label{eq:bcomphat}
c(s)\mm\leq (\hat\gl_s)_*\mm\leq C(s)\mm,\qquad\forall s\in\R,
\end{equation*}
for some positive continuous functions $c,C:\R\to(0,\infty)$.
\item[f)] Since by a compactness argument we have that $\gl_s$ restricted to compact subsets of $B_\Rb(\O)$ is Lipschitz, we deduce that $\hat\gl_s:\X\to \X$ is Lipschitz. Moreover the chain of inequalities
\[
\begin{split}
\sfd\big(\hat\gl_s(x),\hat\gl_s(y)\big)&=\sfd\big(\gl_{{\rm rep}^x_s}(x),\gl_{{\rm rep}^y_s}(y)\big)\\
&\leq \sfd\big(\gl_{{\rm rep}^x_s}(x),\gl_{{\rm rep}^x_s}(y)\big)+\sfd\big(\gl_{{\rm rep}^x_s}(y),\gl_{{\rm rep}^y_s}(y)\big)\\
&\leq \sfd\big(\gl_{{\rm rep}^x_s}(x),\gl_{{\rm rep}^x_s}(y)\big)+\sfd(y,\O)|e^{-{\rm rep}^x_s}-e^{-{\rm rep}^y_s}|\\
&\leq \sfd\big(\gl_{{\rm rep}^x_s}(x),\gl_{{\rm rep}^x_s}(y)\big)+\sfd(y,\O)|{\rm rep}^x_s-{\rm rep}^y_s|\\
\end{split}
\]
together with the fact that ${\rm rep}^y_s=y$ for any $s\in\R^+$ if $y\notin B_\Rb(\O)$, the bound
\[
\begin{split}
|{\rm rep}^x_s-{\rm rep}^y_s|&\leq\int_0^s|\psi'(\b(\gl_r(x)))-\psi'(\b(\gl_r(y)))|\,\d r\\
&\leq \Lip(\psi'\circ\b)\int_0^s\sfd(\gl_r(x),\gl_r(y))\,\d r,
\end{split}
\]
and the fact that $\lip(\gl_r)(x)=e^{-r}\leq 1$ for any $x\in B_\Rb(\O)$, grant that
\begin{equation}
\label{eq:sp}
\sfd\big(\hat\gl_s(x),\hat\gl_s(y)\big)\leq\sfd(x,y)\Big( 1+s\Rb\Lip(\psi'\circ\b)+o(\sfd(x,y))\Big),
\end{equation}
for every  $x,y\in \X$ and $s\in\R^+$. To get a similar control  for negative $s$, we start from
\[
\sfd\big(\hat\gl_s(x),\hat\gl_s(y)\big)\geq \sfd\big(\gl_{{\rm rep}^x_s}(x),\gl_{{\rm rep}^x_s}(y)\big)-\sfd\big(\gl_{{\rm rep}^x_s}(y),\gl_{{\rm rep}^y_s}(y)\big)
\]
and use  the bound ${\rm rep}^x_s\leq s\sup\psi'$, so that arguing as before we get 
\begin{equation}
\label{eq:sm}
\sfd\big(\hat\gl_s(x),\hat\gl_s(y)\big)\geq\sfd(x,y)\Big( 1-s\big(\sup\psi'+\Rb\Lip(\psi'\circ\b)\big)+o(\sfd(x,y))\Big),
\end{equation}
for every  $x,y\in \X$ and $s\in\R^+$. Putting $C:=2(\sup\psi'+\Rb\Lip(\psi'\circ\b))$, the bounds \eqref{eq:sp} and \eqref{eq:sm} give
\[
\lip(\hat\gl_s)(x)\leq 1+|s|C,\qquad\forall x\in \X,\ s\in[-\tfrac1{C},\tfrac1{C}].
\]
and since $\X$ is a geodesic space this further implies that $\Lip(\hat\gl_s)\leq 1+|s|C$ for every $s\in[-\tfrac1{C},\tfrac1{C}]$. Recalling \eqref{eq:normdiffdef} and \eqref{eq:normdiffdef2} we therefore deduce that
\begin{equation}
\label{eq:dhatgl2}
|\d (f\circ\hat\gl_s)|\leq (1+|s|C)|\d f|\circ\hat\gl_s,\qquad\text{and}\qquad |\d \hat\gl_s(v)|\circ\hat\gl_s\leq (1+|s|C)|v|
\end{equation}
$\mm$-a.e.\ for every $f\in W^{1,2}(\X)$, $v\in L^2(T\X)$ and $s\in[-\tfrac1{C},\tfrac1{C}]$.
\end{itemize}
The following lemma will be useful.
\begin{lemma}\label{le:utile}
Let $\varphi\in W^{1,2}(\X)$. Then the map $\R^+\ni s\mapsto\varphi\circ\hat\gl_s\in L^2(\X)$ is $C^1$ and its derivative is given by
\begin{equation}
\label{eq:derflusso}
\frac{\d}{\d s}\varphi\circ\hat\gl_s=-\la\nabla\varphi,\nabla\hat \b\ra\circ\hat\gl_s.
\end{equation}
If $\varphi$ is further assumed to be in $\test \X$, then the map $\R^+\ni s\mapsto\d(\varphi\circ\hat\gl_s)\in L^2(T\X)$ is also $C^1$ and its derivative is given by
\begin{equation}
\label{eq:derflusso2}
\frac{\d}{\d s}\big(\d(\varphi\circ\hat\gl_s)\big)=-\d\big(\la\nabla\varphi,\nabla\hat \b\ra\circ\hat\gl_s\big).
\end{equation}
\end{lemma}
\begin{proof} The first claim is proved exactly as Lemma \ref{le:compgl}, thus we pass to the second.

Start noticing that since $\varphi,\hat\b\in\test \X$, we have $\la\nabla\varphi,\nabla\hat \b\ra\in W^{1,2}(\X)$ and thus, since $\hat\gl_s$ is of bounded deformation, that $\la\nabla\varphi,\nabla\hat \b\ra\circ\hat\gl_s\in W^{1,2}(\X)$ as well. 

We now claim  that
\begin{equation}
\label{eq:contdiff}
\R\ni s\quad\mapsto\quad \d(\la\nabla\varphi,\nabla\hat \b\ra\circ\hat\gl_s)\in L^2(T\X)\qquad\text{ is continuous.}
\end{equation}
Putting for a moment $f:=\la\nabla\varphi,\nabla\hat\b\ra$, using the group property of $(\hat\gl_s)$ and  \eqref{eq:dhatgl2} we get
\begin{equation*}
\label{eq:normdiff}
\int|\d (f\circ\hat\gl_s)|^2\,\d\mm\leq (1+|s-s_0|C)^2\int|\d (f\circ\hat\gl_{s_0})|^2\circ\hat\gl_{s-s_0}\,\d\mm,\qquad\forall s_0\in\R.
\end{equation*}
Applying the first claim in Lemma \ref{le:compgl} with the flow $\hat\gl_s$ in place of $\bar\gl_s$ (the proof is the same) to the $L^1$ function $|\d (f\circ\hat\gl_{s_0})|^2$ we therefore deduce that
\begin{equation*}
\label{eq:normdiff2}
\lims_{s\to s_0}\int|\d (f\circ\hat\gl_s)|^2\,\d\mm\leq\int|\d (f\circ\hat\gl_{s_0})|^2\,\d\mm.
\end{equation*}
Since $L^2(T^*\X)$ is an Hilbert space, to get the claimed strong continuity of $s\mapsto\d(f\circ\hat\gl_s)$ it is now sufficient to prove the weak continuity. This follows recalling that vector fields in $L^2(T\X)$ with divergence in $L^2(\X)$ are dense in $L^2(T\X)$, that for any such vector field $v$ we have
\[
\int \d(f\circ\hat\gl_s)(v)\,\d\mm=-\int f\circ\hat\gl_s\,{\rm div}(v)\,\d\mm,
\]
and using the first claim in Lemma \ref{le:compgl} (again with the flow $\hat\gl_s$ in place of $\bar\gl_s$ ) for the $L^2$ function $f$ to get the continuity in $s$ of the right hand side. This settles property \eqref{eq:contdiff}.

To conclude, notice that the first part of the statement ensures that
\[
\varphi\circ\hat\gl_{s_1}-\varphi\circ\hat\gl_{s_0}=-\int_{s_0}^{s_1}\la\nabla\varphi,\nabla\hat\b\ra\circ\hat\gl_s\,\d s,\qquad\forall s_0<s_1,
\]
the integral being the Bochner one. Take  the differential on both sides and use  the fact that it is, as a map from $W^{1,2}(\X)$ to $L^2(T\X)$, linear and continuous to bring it inside the integral. Then divide by $s_1-s_0$, let $s_1\to s_0$ and use the continuity property \eqref{eq:contdiff} to get the thesis.
\end{proof}

\begin{proposition}\label{prop:derspeed}
Let $v\in L^2(T\X)$ and put $v_s:=\d\hat\gl_s(v)$. Then  the map $s\mapsto \frac12|v_s|^2\circ\hat\gl_s\in L^1(\X)$ is $C^1$ and its derivative is given by the formula
\begin{equation}
\label{eq:dervel}
\frac\d{\d s}\frac12|v_s|^2\circ\hat\gl_s=\H{\hat\b}(v_s,v_s)\circ\hat\gl_s,
\end{equation}
the incremental ratios being converging both in $L^1(\X)$ and $\mm$-a.e.. 

If $v$ is also bounded, then the curve $s\mapsto \frac12|v_s|^2\circ\hat\gl_s$ is $C^1$ also when seen with values in $L^2(\X)$ and in this case the incremental ratios in \eqref{eq:dervel} also converge in $L^2(\X)$ to the right hand side.
\end{proposition}
\begin{proof}

\noindent{\bf Step 1: $v$ is the gradient of a test function and $s=0$.} Assume for the moment that   $v=\nabla\varphi$ for some $\varphi\in \test \X$. Notice that for any $s\in\R$ we have
\begin{equation}
\label{eq:bsotto}
\frac12|v_s|^2\circ\hat\gl_s\geq \d\varphi(v_s)\circ\hat\gl_s-\frac12|\d\varphi|^2\circ\hat\gl_s\qquad\mm-a.e.,
\end{equation}
with equality $\mm$-a.e.\ for $s=0$. Recalling that \eqref{eq:defdiff} gives $\d\varphi(v_s)\circ\hat\gl_s=\d(\varphi\circ\hat\gl_s)(v)$, that $|\d\varphi|^2\in W^{1,2}(\X)$ and using Lemma \ref{le:utile} above, we see that the right hand side of this last inequality is $C^1$ when seen as a curve depending on $s$ with values in $L^2(\X)$. Formulas \eqref{eq:derflusso}, \eqref{eq:derflusso2} grant that its derivative at $s=0$ is given by
\[
\begin{split}
\frac{\d}{\d s}\Big(\d\varphi(v_s)\circ\hat\gl_s-\frac12|\d\varphi|^2\circ\hat\gl_s\Big)\restr{s=0}&=\d\big(\la\nabla\varphi,\nabla\hat \b\ra\big)(v)-\la\nabla\tfrac{|\d\varphi^2|}2,\nabla\hat\b\ra\\
&=\H{\hat\b}(v,v)\circ\hat\gl_{s_0},
\end{split}
\]
having used the fact that $v=\nabla\varphi$ and \eqref{eq:diff2}.  From \eqref{eq:bsotto} we then have
\begin{equation}
\label{eq:doppio}
\lims_{s\uparrow 0}\frac{|v_s|^2\circ\hat\gl_s-|v|^2}{2s}\leq \H{\hat\b}(v,v)\leq \limi_{s\downarrow 0}\frac{|v_s|^2\circ\hat\gl_s-|v|^2}{2s},
\end{equation}
where the $\limi$ and $\lims$ are intended in the $\mm$-essential sense.

\noindent{\bf Step 2: $v$ is locally the gradient of a test function and $s=0$.} 
From the locality property of the differential \eqref{eq:locdiff}  we deduce that \eqref{eq:doppio} holds for $v$ of the form $\sum_i\nchi_{A_i}\nabla\varphi_i\in L^2(T\X)$ for a given Borel partition $(A_i)_{i\in\N}$ of $\X$ and functions $\varphi_i\in \test \X$.

\noindent{\bf Step 3: generic $v\in L^2(T\X)$  and $s=0$.} Let $Q_s:L^2(T\X)\to L^1(\X)$ be the quadratic form defined by
\[
Q_s(v):=\frac{|v_s|^2\circ\hat\gl_s-|v|^2}{s},
\]
and notice that inequality \eqref{eq:dhatgl2} grants that
\begin{equation}
\label{eq:bquad}
|Q_s(v)|\leq C'|v|^2,\quad\mm-a.e.\qquad\forall v\in L^2(T\X),\ s\in[-\tfrac1{C},\tfrac1{C}],
\end{equation}
for some $C'>0$. We claim that the $Q_s$'s are locally uniformly continuous in $s\in[-\tfrac1{C},\tfrac1{C}]$ and to this aim we introduce the auxiliary quadratic forms
\[
\tilde Q_s(v):=Q_s(v)+C'|v|^2,
\]
and notice that \eqref{eq:bquad} yields
\begin{equation*}
\label{eq:bquad2}
0\leq \tilde Q_s(v)\leq 2C'|v|^2,\quad\mm-a.e.\qquad\forall v\in L^2(T\X),\ s\in[-\tfrac1{C},\tfrac1{C}].
\end{equation*}
Letting $\tilde B_s$ be the bilinear form associated to $\tilde Q_s$, the positivity of the latter and the Cauchy-Schwarz inequality - which is easily seen to be valid even in this context - yield
\[
|\tilde B_s(v,w)|^2\leq \tilde Q_s(v)\tilde Q_s(w),\quad\mm-a.e.\qquad\forall v,w\in L^2(T\X),\ s\in[-\tfrac1{C},\tfrac1{C}].
\]
Therefore we have
\[
\begin{split}
|Q_s(v+w)-Q_s(v)|&\leq |\tilde Q_s(v+w)-\tilde Q_s(v)|+C'\big||v+w|^2-|v|^2\big|\\
&=|\tilde Q_s(w)+2\tilde B_s(v,w)|+C'\big||w|^2+2\la v,w\ra\big|\\
&\leq 2C|w|^2+2\sqrt{\tilde Q_s(v)\tilde Q_s(w)}+C'|w|^2+2C'|v||w|\\
&\leq 6C|v||w|+3C'|w|^2,
\end{split}
\]
$\mm$-a.e.\ for every $v,w\in L^2(T\X)$ and $s\in [-\tfrac1{C},\tfrac1{C}]$. This estimate is the claimed local uniform continuity of the $Q_s$'s. 

Since the boundedness of $\H{\hat\b}$ grants that $v\mapsto\H{\hat\b}(v,v)$ is continuous from $L^2(T\X)$ to $L^1(\X)$ and in Step 2 we established \eqref{eq:doppio} for a set of vector fields dense in $L^2(T\X)$, we conclude that \eqref{eq:doppio} holds  for every $v\in L^2(T\X)$.

\noindent{\bf Step 4: conclusion.} Since the $\hat\gl_s$'s form a group, by what we just proved we know that for any $v\in L^2(T\X)$ and any $s_0\in\R$ we have
\begin{equation}
\label{eq:doppio2}
\lims_{s\uparrow s_0}\frac{|v_s|^2\circ\hat\gl_s-|v_{s_0}|^2\circ\hat\gl_{s_0}}{2(s-s_0)}\leq \H{\hat\b}(v_{s_0},v_{s_0})\circ\hat\gl_{s_0}\leq \limi_{s\downarrow s_0}\frac{|v_s|^2\circ\hat\gl_s-|v_{s_0}|^2\circ\hat\gl_{s_0}}{2(s-s_0)},
\end{equation}
$\mm$-a.e.. Now assume for a moment that $v\in  L^2\cap L^\infty(T\X)$, so that using again the group property and keeping in mind the bound \eqref{eq:dhatgl2} we have that $\R\ni s\mapsto \frac12|v_s|^2\circ\hat\gl_s\in L^2(\X)$ is Lipschitz and hence - since Hilbert spaces have the Radon-Nikodym property - it is a.e.\ differentiable. Letting $s_0$ be a point of differentiability, from \eqref{eq:doppio2} we deduce that
\begin{equation}
\label{eq:derpunt}
\lim_{s\to s_0}\frac{|v_s|^2\circ\hat\gl_s-|v_{s_0}|^2\circ\hat\gl_{s_0}}{2(s-s_0)}= \H{\hat\b}(v_{s_0},v_{s_0})\circ\hat\gl_{s_0},
\end{equation}
the limit being intended both in $L^2(\X)$ and $\mm$-a.e.. Since the right hand side of this last expression is, as a function of $s_0$ with values in $L^2(\X)$, continuous, we deduce that $s\mapsto \frac12|v_s|^2\circ\hat\gl_s\in L^2(\X)$ is $C^1$ and that \eqref{eq:derpunt} holds for any $s_0\in\R$.

Finally, let $v\in L^2(T\X)$ be arbitrary and notice that the bound \eqref{eq:dhatgl2} grants that for any $s_0\in\R$ the incremental ratios $\frac{|v_s|^2\circ\hat\gl_s-|v_{s_0}|^2\circ\hat\gl_{s_0}}{2(s-s_0)}$ are, for $s\in[s_0-1,s_0+1]$, dominated in $L^1(\X)$. 

To conclude, put  $v_n:=\nchi_{\{|v|\leq n\}}v\in L^2\cap L^\infty(T\X)$,  $n\in\N$, and $v_{n,s}:=\d\hat\gl_s(v_n)$ and notice that for every $s\in\R$ we have
\[
\begin{split}
\frac{|v_s|^2\circ\hat\gl_s-|v_{s_0}|^2\circ\hat\gl_{s_0}}{2(s-s_0)}&=\frac{|v_{n,s}|^2\circ\hat\gl_s-|v_{n,s_0}|^2\circ\hat\gl_{s_0}}{2(s-s_0)},\qquad\mm-a.e.\ on\ \{|v|\leq n\},\\
 \H{\hat\b}(v_{s_0},v_{s_0})\circ\hat\gl_{s_0}&= \H{\hat\b}(v_{n,s_0},v_{n,s_0})\circ\hat\gl_{s_0},\qquad\mm-a.e.\ on\ \{|v|\leq n\}.
\end{split}
\]
By what we already proved, we deduce that the limit in \eqref{eq:derpunt} holds in the $\mm$-a.e.\ sense on $\{|v|\leq n\}$. Since $n\in\N$ is arbitrary, we conclude that the same limit is true $\mm$-a.e.\ and given that the incremental ratio on the left is dominated in $L^1(\X)$, we obtain that \eqref{eq:dervel} holds, the derivative being intended in $L^1(\X)$.

The stated $C^1$ regularity is then a consequence of the continuity in $L^1(\X)$ of $s\mapsto \H{\hat\b}(v_s,v_s)\circ\hat\gl_s$ which in turn is a consequence of the $L^2(T\X)$ continuity of $s\mapsto v_s$ and the boundedness of $\H{\hat\b}$. 
\end{proof}

\begin{corollary}\label{cor:decr}
Let $v\in L^2(T\X)$ be concentrated on $B_{\Rb'}(\O)\setminus B_{\rb'}(\O)$ and put $v_s:=\d\hat\gl_s(v)$. Then for every $s_1>s_0\geq 0$ we have
\begin{equation}
\label{eq:endecr}
\frac{|v_{s_1}|^2}{\sfd^2(\cdot,\O)}\circ\hat\gl_{s_1}\leq\frac{|v_{s_0}|^2}{\sfd^2(\cdot,\O)}\circ\hat\gl_{s_0},\qquad\mm-a.e..
\end{equation}
\end{corollary}
\begin{proof} Up to replacing $v$ with $v_n:=\nchi_{\{|v|\leq n\}}v$, using the fact that $|\d\hat\gl_s(v_n)|\circ\hat\gl_s=|\d\hat\gl_s(v)|\circ\hat\gl_s$ on $\{|v|\leq n\}$ and letting $n\to\infty$ we can assume that $v$ is bounded.

Put for brevity $A:=B_{\Rb'}(\O)\setminus B_{\rb'}(\O)$ and notice that on the complement of $A$ both sides of \eqref{eq:endecr} are 0 $\mm$-a.e., so that we are reduced to prove that
\[
\frac{|v_{s_1}|^2}{2\b}\circ\hat\gl_{s_1}\nchi_A\leq\frac{|v_{s_0}|^2}{2\b}\circ\hat\gl_{s_0}\nchi_A,\qquad\mm-a.e..
\]
Let $\tilde\b\in W^{1,2}(\X)$ be a function bounded from below by a positive constant and agreeing with $\b$ on $A$. Then $\frac1{\tilde\b}\in W_{loc}^{1,2}(\X)$ and Lemma \ref{le:utile} grants that the derivative of $s\mapsto\frac1{\tilde\b}\circ\hat\gl_s$ in $L^2(\X)$ is $\frac{1}{\tilde\b^2}\la\nabla\tilde\b,\nabla\hat\b\ra\circ\hat\gl_s$. Since  $\hat\gl_s$ maps $A$ into itself for every $s\geq 0$ (point $(c)$ in the list of properties of $\hat\gl_s$) and $\b=\tilde\b$ on $A$, we deduce that $s\mapsto \frac1\b\circ\hat\gl_s\nchi_A\in L^2(\X)$ is $C^1$ with derivative equal to $\frac{1}{\b^2}\la\nabla\b,\nabla\hat\b\ra\circ\hat\gl_s\nchi_A$. Using then the second part of Proposition \ref{prop:derspeed} we conclude that $s\mapsto\frac{|v_s|^2}{2\b}\circ\hat\gl_s\nchi_A\in L^1(\X)$ is $C^1$ with derivative given by
\begin{equation}
\label{eq:caffe}
\begin{split}
\frac{\d}{\d s}\Big(\frac{|v_s|^2}{2\b}\circ\hat\gl_s\nchi_A\Big)\restr{s=0}=\frac{\nchi_A}{2\b^2}\Big(-2\b\,\H{\hat\b}(v,v)+|v|^2 \la\nabla\b,\nabla\hat\b\ra\Big).
\end{split}
\end{equation}
Using the expressions for $\nabla\hat\b$ and $\H{\hat\b}$ given in \eqref{eq:formulehatb} we have that $\mm$-a.e.\ on $A$ it holds
\[
\begin{split}
-2\b\,\H{\hat\b}(v,v)&+|v|^2 \la\nabla\b,\nabla\hat\b\ra\\
&=-2\b|v|^2\Big(1-\frac{\Rb}{2\sqrt{2\b}}\Big)-\frac{\Rb}{2\sqrt{2\b}}\la\nabla\b,v\ra^2+|v|^2|\nabla\b|^2\Big(1-\frac{\Rb}{2\sqrt{2\b}}\Big)\\
&=-\frac{\Rb}{2\sqrt{2\b}}\la\nabla\b,v\ra^2\leq 0,
\end{split}
\]
having used the fact that $|\nabla\b|^2=2\b$ (Corollary \ref{cor:db}). This computation together with \eqref{eq:caffe} gives the conclusion.
\end{proof}
We are now ready to prove our key Proposition \ref{prop:key}, which for convenience we restate:
\begin{proposition}\label{prop:keyagain}
Let $[a,b]\subset(0,\Rb)$ and $\ppi$ be a test plan such that $\sfd(\gamma_t,\O)\in[a,b]$ for every $t\in[0,1]$ and $\ppi$-a.e.\ $\gamma$. Then
\begin{equation*}
\label{eq:speedproj}
{\rm ms}_t(\pr\circ\gamma)\leq\frac{\Rb}{2\sfd(\gamma_t,\O)}{\rm ms}_t(\gamma),\qquad a.e.\ t\in[0,1],\ \ppi-a.e.\ \gamma.
\end{equation*}
\end{proposition} 
\begin{proof}
Build a function $\hat\b$ and the corresponding flow $\hat\gl$ as in the beginning of the current subsection for $[\rb',\Rb']=[a,b]$.

With a slight abuse of notation we shall denote  by $\hat\gl_s$ the map from $C([0,1],\X)$ into itself sending $\gamma$ to $\hat\gl_s\circ\gamma$. Then we put $\ppi_s:=(\hat\gl_s)_*\ppi$.

Recalling that for every $t\in[0,1]$ the differential of $\hat\gl_s$ induces a map, still denoted by $\d\hat\gl_s$ from $L^2(T\X,\e_t,\ppi)$ to $L^2(T\X,\e_t,\ppi_s)$ (recall the discussion in Section \ref{se:tools}), we claim that for any $s_1\geq s_0\geq 0$ and any $V\in L^2(T\X,\e_t,\ppi)$ it holds
\begin{equation}
\label{eq:endecr2}
\frac{|\d\hat \gl_{s_1}(V)|^2}{\sfd^2_{\O}\circ\e_t}\circ\hat\gl_{s_1}\leq\frac{|\d\hat \gl_{s_0}(V)|^2}{\sfd^2_{\O}\circ\e_t}\circ\hat\gl_{s_0},\qquad\ppi-a.e..
\end{equation}
Indeed, for $V$ of the form $\e_t^*v$ for some $v\in L^2(T\X)$ the claim follows directly from Corollary \ref{cor:decr} above via the computation:
\begin{equation}
\label{eq:pbdiff}
\begin{split}
\frac{|\d\hat \gl_{s_1}(\e_t^*v)|^2}{\sfd^2_{\O}\circ\e_t}\circ\hat\gl_{s_1}&=\frac{|\e_t^*(\d\hat \gl_{s_1}(v))|^2}{\sfd^2_{\O}\circ\e_t}\circ\hat\gl_{s_1}=\frac{|\d\hat \gl_{s_1}(v)|^2}{\sfd^2_{\O}}\circ\e_t\circ\hat\gl_{s_1}=\frac{|\d\hat \gl_{s_1}(v)|^2}{\sfd^2_{\O}}\circ\hat\gl_{s_1}\circ\e_t\\
\text{(by \eqref{eq:endecr})}\qquad\qquad&\leq\frac{|\d\hat \gl_{s_0}(v)|^2}{\sfd^2_{\O}}\circ\hat\gl_{s_0}\circ\e_t=\cdots=\frac{|\d\hat \gl_{s_0}(\e_t^*v)|^2}{\sfd^2_{\O}\circ\e_t}\circ\hat\gl_{s_0},\qquad\ppi-a.e..
\end{split}
\end{equation}
Then the locality property of $\d\hat\gl_s:L^2(T\X,\e_t,\ppi)\to L^2(T\X,\e_t,\ppi_s)$ expressed by the second in \eqref{eq:liftdF} ensures that for $V$ of the form $\sum_i\nchi_{A_i}\e_t^*v_i$ for some Borel partition $(A_i)_{i\in\N}$ of $C([0,1],\X)$ and $(v_i)\subset L^2(T\X)$ it holds
\[
\begin{split}
\frac{|\d\hat \gl_{s_1}(\sum_i\nchi_{A_i}\e_t^*v_i)|^2}{\sfd^2_{\O}\circ\e_t}\circ\hat\gl_{s_1}&=\frac{|\sum_i\nchi_{A_i}\circ\hat\gl_{s_1}^{-1}\d\hat \gl_{s_1}(\e_t^*v_i)|^2}{\sfd^2_{\O}\circ\e_t}\circ\hat\gl_{s_1}=\sum_i\nchi_{A_i}\frac{|\d\hat \gl_{s_1}(\e_t^*v_i)|^2}{\sfd^2_{\O}\circ\e_t}\circ\hat\gl_{s_1}\\
\text{(by \eqref{eq:pbdiff})}\qquad\qquad&\leq\sum_i\nchi_{A_i}\frac{|\d\hat \gl_{s_0}(\e_t^*v_i)|^2}{\sfd^2_{\O}\circ\e_t}\circ\hat\gl_{s_0}=\cdots=\frac{|\d\hat \gl_{s_0}(\sum_i\nchi_{A_i}\e_t^*v_i)|^2}{\sfd^2_{\O}\circ\e_t}\circ\hat\gl_{s_0}
\end{split}
\]
$\ppi$-a.e.. Then the claim \eqref{eq:endecr2} follows by the density of the set of $V$'s of the form just considered in $L^2(T\X,\e_t,\ppi)$ and the continuity of $\d\hat\gl_s:L^2(T\X,\e_t,\ppi)\to L^2(T\X,\e_t,\ppi_s)$.

Denoting by $(\ppi_s)_t'\in L^2(T\X,\e_t,\ppi_s)$ the speed at time $t$ of the test plan $\ppi_s$, applying \eqref{eq:endecr2} to $\ppi_t'$ and recalling \eqref{eq:speedaftercompos} we obtain that for $s_1\geq s_0\geq 0$ and a.e.\ $t\in[0,1]$ it holds
\[
\frac{|(\ppi_{s_1})'_t|^2}{\sfd^2_{\O}\circ\e_t}\circ\hat\gl_{s_1}\leq\frac{|(\ppi_{s_0})'_t|^2}{\sfd^2_{\O}\circ\e_t}\circ\hat\gl_{s_0},\qquad\ppi-a.e..
\]
Integrating in $t$ and recalling the link between pointwise norm and metric speed given in \eqref{eq:linkspeedmetric} we obtain that
\begin{equation}
\label{eq:mono}
\iint_0^1\frac{|\dot\gamma_t|^2}{\sfd^2(\gamma_t,\O)}\,\d t\,\d\ppi_{s_1}(\gamma)\leq \iint_0^1\frac{|\dot\gamma_t|^2}{\sfd^2(\gamma_t,\O)}\,\d t\,\d\ppi_{s_0}(\gamma).
\end{equation}
Recall that \(A=B_{\bar \Rb}(\O)\setminus B_{r}(\O)\) and let us consider the closure $\bar A$ of $A$ and the functional $E:C([0,1],\bar A)\to[0,+\infty]$ given by
\[
E(\gamma):=\int_0^1\frac{|\dot\gamma_t|^2}{\sfd^2(\gamma_t,\O)}\,\d t\quad\text{ if }\gamma\in AC_2([0,1],\bar A),\qquad\qquad\qquad E(\gamma)=+\infty,\quad\text{otherwise}.
\] 
It is readily verified that  $E$ is lower semicontinuous. Indeed, let $(\gamma_n)$ be with $\sup_nE(\gamma_n)<\infty$ and uniformly converging to $\gamma$. Then $\sup_n\int|\dot\gamma_{n,t}|^2\,\d t<\infty$ and the uniform convergence ensures that $\limi_nE(\gamma_n)=\limi_n\int\frac{|\dot\gamma_{n,t}|^2}{\sfd(\gamma_t,\O)^2}\,\d t$. Since $\gamma$ takes values in $A$, we have that $\frac{1}{\sfd(\gamma_t,\O)}$ is bounded from above and thus the functions $|\dot\gamma_{n,t}|$ are uniformly bounded in $L^2([0,1],\frac{1}{\sfd(\gamma_t,\O)^2}\d t)$. Up to pass to a subsequence, not relabeled, we can assume that $|\dot\gamma_{n,t}|\weakto G$ in $L^2([0,1],\frac{1}{\sfd(\gamma_t,\O)^2}\d t)$ and passing to the limit in $\sfd(\gamma_{n,{t_1}},\gamma_{n,t_0})\leq \int_{t_0}^{t_1}|\dot\gamma_{n,t}|\,\d t$ deduce that $|\dot\gamma_t|\leq G(t)$ for a.e.\ $t\in[0,1]$. Then the lower semicontinuity of $E$ follows from the lower semicontinuity of the $L^2([0,1],\frac{1}{\sfd(\gamma_t,\O)^2}\d t)$-norm w.r.t.\ weak convergence.

Since $E$ is lower semicontinuous and non-negative, we deduce that the functional 
\[
\prob{C([0,1],\bar A)}\ni \ssigma\qquad\mapsto\qquad \int E(\gamma)\,\d\ssigma(\gamma)\in [0,+\infty],
\]
is lower semicontinuous w.r.t.\ weak convergence in duality with continuous and bounded functions on  $C([0,1],\bar A)$.

Now consider the functions $\hat\gl_s$ as functions from $\bar A$ into itself and recall that they uniformly converge to $\pr$ as $s\to+\infty$ (by \eqref{eq:expr3}). It follows that $(\ppi_s)$ weakly converges to $\pr_*\ppi$ as $s\to\infty$ and thus that
\[
\int E(\gamma)\,\d \pr_*\ppi\leq\limi_{s\to+\infty} \int E(\gamma)\,\d \ppi_s.
\]
Then the monotonicity property \eqref{eq:mono} and  the definition of $E$ give
\begin{equation}
\label{eq:concint}
\iint_0^1 \frac{{\rm ms}_t^2(\pr\circ\gamma)}{|\frac\Rb2|^2}\,\d t\,\d \ppi\leq\iint_0^1 \frac{{\rm ms}_t^2(\gamma)}{\sfd^2(\gamma_t,\O)}\,\d t\,\d \ppi.
\end{equation}
To pass from this integral formulation to the conclusion, we start claiming that for every $[t_0,t_1]\subset [0,1]$ and $\Gamma\subset C([0,1],\bar A)$ it holds
\begin{equation}
\label{eq:concint2}
\iint_{[t_0,t_1]\times \Gamma}\frac{{\rm ms}_t^2(\pr\circ\gamma)}{|\frac\Rb2|^2}\,\d t\,\d \ppi\leq\iint_{[t_0,t_1]\times \Gamma} \frac{{\rm ms}_t^2(\gamma)}{\sfd^2(\gamma_t,\O)}\,\d t\,\d \ppi.
\end{equation}
Indeed, if $\ppi(\Gamma)=0$ or $t_0=t_1$ there is nothing to prove, otherwise simply write \eqref{eq:concint} for the plan $\frac1{\ppi(\Gamma)}({\rm Restr}_{t_0}^{t_1})_*(\ppi\restr{\Gamma})$, which is still a test plan, in place of $\ppi$ to get the claim.

The class of sets of the form $[t_0,t_1]\times\Gamma$ just considered is a $\pi$-system which generates the Borel $\sigma$-algebra of $[0,1]\times C([0,1],\bar A)$, hence by the $\pi$-$\lambda$ theorem we deduce that we can replace $[t_0,t_1]\times\Gamma$ in \eqref{eq:concint2} by an arbitrary Borel set $E\subset [0,1]\times C([0,1],\bar A)$. Choosing as $E$ the set where the integrand on the left is bigger than the one on the right, we conclude.
\end{proof}

{
\subsection{The rescaled sphere $\Z$ and the cone  $\Y$ built over it}

Let us define \((\Z,\sfd_\Z,\mm_\Z)\) as \((\X',2\sfd'/\Rb, \mm')\) and note that \(f:\Z=\X'\to \R\) belongs to  \(W^{1,2}(\X')\) if and only if it belongs to \(W^{1,2}(\Z)\) and in this case
\begin{equation}\label{eq:xprimoxsecondo}
|\D f|_{\X'}=\frac{2}{\Rb}|\D f|_{\Z}\qquad  \textrm{\(\mm'\) a.e. (equivalently \(\mm_\Z\) a.e.).}
\end{equation}
On  the set $\Z\times [0,+\infty)$ we put the semidistance $\sfd_\Y$ defined by
\[
\sfd_\Y^2\big((z_0,s_0),(z_1,s_1)\big):=\inf \int_0^1 \big|\frac{\d}{\d r}s(r)\big|^2+s^2(r)|\dot\gamma_r|^2\,\d r,
\]
the $\inf$ being taken among all Lipschitz curves $[0,1]\ni r\mapsto  s(r)\in \R^+$ and $[0,1]\ni r\mapsto\gamma_r\in \Z$ such that $(\gamma_i,s(i))=(z_i,s_i)$ for $i=0,1$.

We denote by $\Y$ the quotient of $\Z\times [0,+\infty)$ w.r.t. the equivalence relation  given by $(z_0,s_0)\sim (z_1,s_1)$ provided  $\sfd_\Y\big((z_0,s_0),(z_1,s_1)\big)=0$. In particular, $(z_0,0)=(z_1,0)$ for any $z_0,z_1\in \Z$ and their equivalence class will be denoted by $O_\Y$.

The semidistance $\sfd_\Y$ passes to the quotient and induces a distance on $\Y$ which we shall continue to denote as $\sfd_\Y$. It is easy to check that $(\Y,\sfd_\Y)$ is a locally  compact metric space whose topology is the same as the quotient topology. The typical element of $\Y$ will be denoted by $(z,r)$ with $z\in \Z=\X'$ and $r\in[0,+\infty)$.

We also endow $\Y$ with the measure $\mm_\Y$ defined by
\[
\int_\Y f(z,s)\,\d\mm_\Y:=\frac{N\mm(B_\Rb(\O))}{\Rb^N}\int_0^{+\infty}s^{N-1}\int_{\Z}f(z,s)\,\d\mm_\Z(z)\,\d s,
\]
for every non-negative Borel function $f$.

\bigskip

We shall now recall some results proved in \cite{GH15} about the structure of Sobolev functions on $\Y$. It is convenient to introduce the following notation: for  $f:\Y\to\R$ given and $z\in \Z$ we shall denote by $f^{(z)}$ the function on $\R^+$ given by $r\mapsto f(z,r)$, similarly, for $r\in[0,+\infty)$ the function $f^{(r)}$ on $\Z$ is defined as $z\mapsto f(z,r)$.

\begin{theorem}\label{thm:BG2}
Let $f\in W^{1,2}(\Y)$. Then:
\begin{itemize}
\item[i)] for $\mm_\Z$-a.e.\ $z\in\Z$ we have $f^{(z)}\in W^{1,2}([0,+\infty),\sfd_{\rm Eucl},r^{N-1}\d r)$
\item[ii)] for $\mathcal L^1$-a.e.\ $r\in[0,+\infty)$ we have $f^{(r)}\in W^{1,2}(\Z,\sfd_\Z,\mm_\Z)$
\item[iii)] the identity
\begin{equation}
\label{eq:warpedgrad}
|\D f|_\Y^2(r,x)=|\D f^{(z)}|^2_\R(r)+\frac{1}{r^2}|\D f^{(r)}|^2_{\Z}(z)
\end{equation}
holds for $\mm_\Y$-a.e.\ $(z,r)$.
\end{itemize}
Conversely, if $f\in L^2(\Y)$ is 0 on a neighbourhood of $\O_\Y$,  $(i)$ and $(ii)$ hold and the right hand side of \eqref{eq:warpedgrad} is in $L^2(\Y)$, then $f\in W^{1,2}(\Y)$.

Furthermore, $\mathcal \Y$ is infinitesimally Hilbertian and has the Sobolev-to-Lipschitz property.
\end{theorem}
\begin{proof}
The relation between $W^{1,2}(\Y)$ and $(i),(ii)$ is one of the main results of \cite{GH15}. Infinitesimal Hilbertianity then follows directly from the one of $\X'$ (Proposition \ref{prop:basexp}). The Sobolev-to-Lipschitz property  follows from the fact that $(\X',\sfd',\mm')$ and thus $(\Z,\sfd_\Z,\mm_\Z)$  is doubling and measured-length, as proved in Proposition \ref{prop:basexp}, and the results in the last section of \cite{GH15}.
\end{proof}

\subsection{From annuli in $\X$ to annuli in $\Y$ and viceversa}
We shall now adapt the strategy used in \cite{Gigli13} to prove that the natural map from $B_\Rb(\O_\Y)$ to $ B_{\Rb}(\O)$ preserves the Sobolev norm of functions defined in appropriate annuli. As the proofs closely follows those in \cite{Gigli13}, we shall mostly only sketch them, highlighting the key points and providing precise references for their completion.

We introduce, for $0<r<R<\Rb$, the annulus ${\rm Ann}^\Y_{r,R}:=\{y\in\Y:\sfd_\Y(y,O_\Y)\in(r,R)\}\subset\Y$, and similarly the annulus ${\rm Ann}^\X_{r,R}:=\{x\in \X:\sfd(x,\O)\in(r,R)\}\subset \X$.

We then introduce the map $\mau:B_\Rb(O_\Y)\to  B_{\Rb}(\O)$ as 
\[
\mau (z,r):=\gl_{\log(\frac\Rb{2r})}(z),\quad\text{if }r>0,\qquad\qquad \mau(O_\Y):=\O,
\]
and the map $\mad: B_{\Rb}(\O)\to B_\Rb(O_\Y)$ as
\[
\mad (x):=(\pr(x),\sfd(x,O)),\qquad\forall x\in B_{\Rb}(\O)\setminus\{\O\},\qquad\qquad\mad (\O):=O_\Y.
\]
It is clear that $\mau,\mad$ are one the inverse of the other, while the very definition of $\mm_\Y$ and Corollary \ref{cor:contdis} grant that
\begin{equation}
\label{eq:mespres}
\mau_*\mm_\Y\restr{B_{\Rb}(\O)}=\mm\restr{B_{\Rb}(\O)}\qquad\text{ and }\qquad \mad_*\mm\restr{B_{\Rb}(\O)}=\mm_\Y\restr{B_{\Rb}(\O)}.
\end{equation}
Moreover, noticing that point $(ii)$ of Theorem \ref{thm:rapprcont} gives that 
\[
\sfd(x_1,x_2)\leq \frac{{2d}}{\Rb}\sfd'(\pr(x_1),\pr(x_2))=d\, \sfd_\Z(\pr(x_1),\pr(x_2))
\]
for any $x_1,x_2\in \X$ with $\sfd(x_1,\O)=\sfd(x_2,\O)=d\in(0,\Rb)$, we have the estimate
\[
\begin{split}
\sfd(x,y)&\leq \sfd\Big(x,\gl_{\log\big(\tfrac{\sfd(y,\O)}{\sfd(x,\O)}\big)}(y)\Big)+\sfd\Big(\gl_{\log\big(\tfrac{\sfd(y,\O)}{\sfd(x,\O)}\big)}(y),y\Big)\\
&\leq \sfd(x,\O)\sfd_\Z(\pr(x),\pr(y))+|\sfd(x,\O)-\sfd(y,\O)|.
\end{split}
\]
Thus the very definition of $\sfd_\Y$ grants that
\begin{equation}
\label{eq:lip1}
\text{for every $\eps\in(0,\Rb)$, $\mau$ is Lipschitz from ${\rm Ann}^\Y_{\eps,\Rb}$ to ${\rm Ann}^\X_{\eps,\Rb}$}.
\end{equation}
Similarly, the fact that $\sfd(\cdot,\O):\X\to \R$ is Lipschitz and that for every $\eps,\eps'\in(0,\Rb/2)$ the map $\pr:{\rm Ann}^\X_{\eps,\Rb-\eps'}\to\X'$ is also Lipschitz (by local Lipschitzianity and a compactness argument) grant, together with the definition of $\sfd_\Y$, that
\begin{equation}
\label{eq:lip2}
\text{for every $\eps,\eps'\in(0,\Rb/2)$, $\mad$ is Lipschitz from ${\rm Ann}^\X_{\eps,\Rb-\eps'}$ to ${\rm Ann}^\Y_{\eps,\Rb-\eps'}$ }.
\end{equation}

Finally, we define the following classes of functions:
\[
\begin{split}
\mathcal G&:=\Big\{g:\Y\to\R\ :\ g(x',r)=\tilde g(x')\text{ for some }\tilde g\in W^{1,2}\cap L^\infty(\Z)\Big\},\\
\mathcal H&:=\Big\{h:\Y\to\R\ :\ h(x',r)=\tilde h(r)\text{ for some }\tilde h\in \Lip([0,\Rb])\text{ with }\supp(h)\subset\ (0,\Rb)\Big\},\\
\mathcal A&:=\Big\{\sum_{i=1}^ng_ih_i\ :\ i\in\N,\ \ g_i\in\mathcal G,\ h_i\in\mathcal H\ \ \forall i=1,\ldots,n\Big\}.
\end{split}
\]
In the foregoing discussion,  given a metric measure space $(\Z,\sfd_\Z,\mm_\Z)$ and an open set $\Omega\subset \X$, we shall denote by $W^{1,2}_0(\Omega)\subset W^{1,2}(\Z)$  the $W^{1,2}(\X)$-completion of the space of functions in $W^{1,2}(\X)$ with support in $\Omega$. Using Theorem \ref{thm:BG2} we then see  that  every function in $\mathcal A$ belongs to $W^{1,2}_0({\rm Ann}^\Y_{\eps,\Rb})$ for any $\eps\in(0,\Rb)$ sufficiently small. In particular, the minimal weak upper gradient of such functions is well defined $\mm_\Y$-a.e..
\begin{proposition}\label{prop:1} For every $\eps\in(0,\Rb)$,  $\mathcal A\cap W^{1,2}_0({\rm Ann}^\Y_{\eps,\Rb})$ is a dense subset of $W^{1,2}_0({\rm Ann}^\Y_{\eps,\Rb})$. 
\end{proposition}
\begin{proof}
It follows from the very same arguments used to prove the analogous statement Proposition 6.6 in \cite{Gigli13}  keeping in mind Theorem \ref{thm:BG2}  (see also Proposition 4.16 in \cite{Gigli13over} and \cite{GH15}).
\end{proof}

\begin{proposition}\label{prop:2} For every $\eps,\eps'\in(0,\Rb/2)$, the map $f\mapsto f\circ \mad$ is a homeomorphism from $W^{1,2}_0({\rm Ann}^\X_{\eps,\Rb-\eps'})$ to $W^{1,2}_0({\rm Ann}^\Y_{\eps,\Rb-\eps'})$.
\end{proposition}
\begin{proof}
Direct consequence of the measure preservation property \eqref{eq:mespres}, the Lipschitz properties \eqref{eq:lip1}, \eqref{eq:lip2} and the fact that the Sobolev norm changes in a bi-Lipschitz way under a bi-Lipschitz change of the metric (see also Proposition 6.7 in \cite{Gigli13}, Proposition 4.16 in \cite{Gigli13over} and the arguments used in \cite{GH15}).
\end{proof}

\begin{proposition}\label{prop:3}
For every $f\in \mathcal A$ we have  $f\circ\mad\in W^{1,2}_0({\rm Ann}^\X_{\eps,\Rb-\eps})$ for every $\eps>0$ sufficiently small. Moreover
\[
|\D f|_\Y\circ\mad=|\D(f\circ\mad)|_\X,\qquad\mm-a.e.\ on\ B_{\Rb}(\O).
\]
\end{proposition}
\begin{proof} The first claim follows from the fact, already noticed, that and $f\in\mathcal A$ belongs to $ W^{1,2}_0({\rm Ann}^\Y_{\eps,\Rb-\eps})$ for $\eps>0$ sufficiently small and Proposition \ref{prop:2} above.
 
Now  let  $g\in\mathcal G$, $\eps\in(0,\Rb/2)$ be arbitrary and pick $h\in\mathcal H$ identically 1 on ${\rm Ann}^\Y_{\eps,\Rb-\eps}$. Theorem \ref{thm:linksob}   grant that $(gh)\circ\mad \in W^{1,2}({\rm Ann}^\X_{\eps,\Rb})$ and using also  Theorem \ref{thm:BG2} and \eqref{eq:xprimoxsecondo} we see that  $|\D g|_\Y\circ\mad=|\D(g\circ\mad)|_\X$ $\mm$-a.e.\ on ${\rm Ann}^\X_{\eps,\Rb-\eps}$. As $\eps>0$ was chosen arbitrarily, we deduce that  $|\D g|_\Y\circ\mad=|\D(g\circ\mad)|_\X$ $\mm$-a.e.\ on $B_\Rb(\O)$.

Similarly, given that a function $h\in\mathcal H$ is Lipschitz with support in ${\rm Ann}^\Y_{\eps,\Rb-\eps'}$ for some $\eps,\eps'\in(0,\Rb/2)$, it is clear that $h\in W^{1,2}(\Y)$ and, recalling \eqref{eq:lip2}, that $h\circ\mad \in W^{1,2}( B_{\Rb}(\O))$. The fact that $|\D h|_\Y\circ\mad=|\D(h\circ\mad)|_\X$ $\mm$-a.e.\ on $B_{\Rb}(\O)$ then follows from the very same arguments used in Proposition 6.3 in \cite{Gigli13} (see also Proposition 4.14 in \cite{Gigli13over} and \cite{GH15}).

The conclusion for general $f\in \mathcal A$ then comes using the very same arguments of the proof of Proposition 6.5 in \cite{Gigli13} (see also Proposition 4.15 in \cite{Gigli13over}), keeping in mind the infinitesimal Hilbertianity of $\X,\Y$, the characterisation of Sobolev functions on $\Y$ given by Theorem \ref{thm:BG2} and the first order differentiation formula \eqref{eq:calcoloder}.
\end{proof} 

\begin{theorem}\label{thm:sobug} For any $\eps,\eps'\in(0,\Rb/2)$ we have   $f\in W_0^{1,2}({\rm Ann}^\Y_{\eps,\Rb-\eps'})$ if and only if $f\circ\mad \in W_0^{1,2}({\rm Ann}^\X_{\eps,\Rb-\eps'})$ and in this case
\[
|\D f|_\Y\circ\mad=|\D(f\circ\mad)|_\X,\qquad\mm-a.e.\ on\ {\rm Ann}^\X_{\eps,\Rb-\eps'}.
\]
\end{theorem}
\begin{proof}
The fact that   $f\in W_0^{1,2}({\rm Ann}^\Y_{\eps,\Rb-\eps'})$ if and only if $f\circ\mad \in W_0^{1,2}({\rm Ann}^\X_{\eps,\Rb-\eps'})$ has already been proved in Proposition \ref{prop:2}. Now let $f\in W^{1,2}_0({\rm Ann}^\Y_{\eps,\Rb-\eps'})$ and use Proposition \ref{prop:1} to find a sequence $(f_n)\subset \mathcal A$ converging to it in $W^{1,2}_0({\rm Ann}^\Y_{\eps,\Rb-\eps'})$. By Proposition \ref{prop:2} again we deduce that $f_n\circ\mad\to f\circ\mad$ in $W^{1,2}({\rm Ann}^\X_{\eps,\Rb-\eps'})$ and since by Proposition \ref{prop:3} we know that
\[
|\D f_n|_\Y\circ\mad=|\D(f_n\circ\mad)|_\X,\qquad\mm-a.e.\ on\ {\rm Ann}^\X_{\eps,\Rb-\eps'}
\]
for every $n\in\N$, passing to the limit (recall \eqref{eq:mespres} for the left hand side) we conclude.
\end{proof}
}

\subsection{Back to the metric properties and conclusion}\label{metprop}
We can now state and prove our main result.

\begin{theorem}[Main result]\label{thm:main}
Let $N\in [1,\infty)$, $(\X,\sfd,\mm)$ a $\RCD^*(0,N)$ space with $\supp(\mm)=\X$, $\O\in \X$ and $\Rb>\rb>0$ such that
\[
\mm(B_{\Rb}(\O))=\Big(\frac{\Rb}{\rb}\Big)^N\mm(B_{\rb}(\O)).
\]
Then exactly one of the following holds:
\begin{itemize}
\item[1)] $S_{\Rb/2}(\O)$ contains only one point. In this case $(\X,\sfd)$ is isometric to $[0,{\rm diam}(\X)]$ ($[0,\infty)$ if $\X$ is unbounded) with an isometry which sends $\O$ in $0$ and the measure $\mm\restr{B_\Rb}(\O)$ to  the measure $c\, x^{N-1}\d x$ for  $c:=N \mm(B_\Rb(\O))$.
\item[2)] $S_{\Rb/2}(\O)$ contains  two points. In this case $(\X,\sfd)$ is a 1-dimensional Riemannian manifold, possibly with boundary, and there is a bijective local isometry (in the sense of distance-preserving maps) from $B_{\Rb}(\O)$ to $(-\Rb,\Rb)$ sending $\O$ to $0$ and the measure  $\mm\restr{B_\Rb(\O)}$ to the measure  $c\,|x|^{N-1} \d x$ for $c:=\frac12N\mm(B_\Rb(\O))$. Moreover, such local isometry is an  isometry when restricted to $\bar B_{\Rb/2}(\O)$.
\item[3)] $S_{\Rb/2}(\O)$ contains more than two points. In this case:
\begin{itemize}
\item[-] $N\geq 2$ and the metric measure space $(\Z,\sfd_\Z,\mm_\Z)$ is a $\RCD^*(N-2,N-1)$ space.
\item[-] The map $\mad:B_\Rb(\O)\to\Y$ is a measure preserving local isometry which, when restricted to $\bar B_{\Rb/2}(\O)$, is an isometry.
\end{itemize}
\end{itemize} 
\end{theorem}
\begin{proof}
Cases (1) and (2) have already been handled in Corollary \ref{cor:easy}, thus we assume that  $S_{\Rb/2}(\O)$ contains more than two points and notice that we already proved that $\mau,\mad$ are measure preserving.

We now claim that for any $\eps,\eps'\in(0,\Rb)$,  the maps $\mau,\mad$ are locally isometries from ${\rm Ann}^{\Y}_{\eps,\Rb-\eps'}$ to ${\rm Ann}^{\X}_{\eps,\Rb-\eps'}$ and viceversa.

To this aim, pick $y\in {\rm Ann}^{\Y}_{\eps,\Rb-\eps'}$ and let $r>0$ be such that $B_{5r}(y)\subset  {\rm Ann}^{\Y}_{\eps,\bar R-\eps'}$. Pick $y_1,y_2\in B_r(y)$ and consider the function $f:=\min\{\sfd_\Y(y_1,\cdot),4r-\sfd_\Y(y_1,\cdot)\}$, which is supported in $B_{5r}(y)$.

By Theorem \ref{thm:sobug} we deduce that $\tilde f:=f\circ\mad$ is in $W^{1,2}({\rm Ann}^{\X}_{\eps,\Rb-\eps'})$ with $|\D \tilde f |_\X=|\D f|_\Y\circ\mad\leq 1$ $\mm$-a.e.\ on ${\rm Ann}^{\X}_{\eps,\Rb-\eps'}$, the inequality being a consequence of the fact that $f$ is 1-Lipschitz. Notice that being the support of $\tilde f$ contained in ${\rm Ann}^\X_{\eps,\Rb-\eps'}$, we can extend it to the whole $\X$ setting it to be 0 outside ${\rm Ann}^\X_{\eps,\Rb-\eps'}$ and the new function, which we will continue to denote $\tilde f$, will still be in $W^{1,2}(\X)$ with $|\D \tilde f|\leq 1$ $\mm$-a.e..  By the Sobolev-to-Lipschitz property of $\X$ - recall the discussion in the  Section \ref{se:prel} -, we then deduce that $\tilde f$ has a 1-Lipschitz representative, but being $\tilde f$ continuous such representative must be equal to $\tilde f$ itself. In particular we have
\[
\sfd_\Y(y_1,y_2)=|f(y_1)-f(y_2)|=|\tilde f(\mau(y_1))-\tilde f(\mau(y_2))|\leq \sfd(\mau(y_1),\mau(y_2)).
\]
Recalling that $y_1,y_2\in B_r(y)$ were chosen arbitrarily and reversing the roles of ${\rm Ann}^{\Y}_{\eps,\Rb-\eps'}$, ${\rm Ann}^{\X}_{\eps,\Rb-\eps'}$ in the argument (the Sobolev-to-Lipschitz property of $\Y$ being ensured by Theorem \ref{thm:BG2}), we conclude.

Now let $x_0,x_1\in \bar B_{\Rb/2}(\O)\setminus \{\O\}$ and, recalling that $(\Z,\sfd_\Z)$ is a geodesic space, notice that their distance can be realized as limit of the lengths of a sequence of curves with range in $B_\Rb(\O)\setminus\{\O\}$. Any such curve has range in ${\rm Ann}^\X_{\eps,\Rb-\eps}$ for some $\eps>0$ and therefore its length  is equal to the length in $\Y$ of its composition with $\mad$. This shows that the restriction of $\mad$ to $\bar B_{\Rb/2}(\O)\setminus \{\O\}$ is 1-Lipschitz. In particular, it can be extended to a 1-Lipschitz map sending $\bar B_{\Rb/2}(\O)$ to $\bar B_{\Rb/2}(\O_\Y)$ and it is obvious that such extension sends $\O$ in $\O_\Y$ and thus agrees with $\mad (\O)$ as previously defined. Reversing the roles of $\X,\Y$ we see that $\mad:\bar B_{\Rb/2}(\O)\to \bar B_{\Rb/2}(\O_\Y)$ is an isometry.

The fact that $N\geq 2$ now follows by considerations about the Hausdorff dimension. Indeed, the fact that $\Z$ is geodesic forces $B_\Rb(\O_\Y)$ to contain a bi-Lipschitz copy of the square $[0,1]^2$. Thus the Hausdorff dimension of $B_\Rb(\O_\Y)$ is at least 2 and since it is locally isometric to $B_\Rb(\O)$, the same holds for $\X$. The claim then follows recalling that $N$ bounds the Hausdorff dimension of $\X$ from above (see Corollary 2.5 in  \cite{Sturm06II}).

It remains to prove that $(\Z,\sfd_\Z,\mm_\Z)$ is a $\RCD^*(N-2,N-1)$ space. {Thanks to the result of Ketterer \cite[Theorem 5.28]{Ketterer13}, to this aim it is sufficient to prove that \(\Y\), which by construction is a cone, is also a \(\RCD^*(0,N)\) space. 

This can  be seen as follows. First of all, we recall that  \(\bar B_{\Rb/2}(\O)\) is isometric \(\bar B_{\Rb/2}(\O_Y)\), then we observe that since $\bar B_{\Rb/4}(\O_\Y)$ is a totally geodesic subset of $\Y$, we must have that $\bar B_{\Rb/4}(\O)$  is a totally geodesic subset of $\X$ (since geodesics with endpoints in $\bar B_{\Rb/4}(\O)$ cannot leave $\bar B_{\Rb/2}(\O)$). 

Now we use the Global-to-Local result established in \cite{AmbrosioGigliSavare11-2} and \cite{AmbrosioMondinoSavare13-2} to deduce that \(\bar B_{\Rb/4}(\O)\) is \(\RCD^*(0,N)\) and again the isomorphism of metric measure structures to see that  \(\bar B_{\Rb/4}(\O_\Y)\) is $\RCD^*(0,N)$ as well. 

Finally define the spaces $(\Y_r,\sfd_{\Y_r},\mm_{\Y_r})$ as $(\Y,\sfd_\Y/r,\mm_\Y/r^N)$ and note that $\Y_r$ is isomorphic to $\Y$, the isomorphism being given by $(z,s)\mapsto (z,rs)$, so that trivially $\Y$ is the pointed-measured-Gromov-Hausdorff limit of the net $\{(\Y_r,\sfd_{\Y_r},\mm_{\Y_r},\O_{\Y_r})\}_{r>0}$ as $r\downarrow0$. It is then clear that also the balls $B_{\Rb/(4r)}(\O_{\Y_r})\subset \Y_r$ converge to $\Y$ and since the former are, by what said previously, $\RCD^*(0,N)$ spaces, the same is true for $\Y$, as desired.
}
\end{proof}
\section{Variants}
There is nothing special about the choice $K=0$ in the discussion we did. Here we, very briefly, discuss general lower bounds on the Ricci and the case of `volume annulus implies metric annulus'.

For $K\in\R$, $N>1$ define $s_{K,N}:\R^+\to\R$ as
\[
s_{K,N}(r):=\left\{
\begin{array}{ll}
{\frac{1}{\sqrt{K}}}\sin(r\sqrt{\frac{K}{N-1}}),&\qquad\text{ if }K>0,\\
r,&\qquad\text{ if }K=0,\\
{\frac{1}{\sqrt{|K|}}}\sinh(r\sqrt{\frac{|K|}{N-1}}),&\qquad\text{ if }K<0,\\
\end{array}
\right.
\]
and $v_{K,N}:\R^+\to\R$ as
\[
v_{K,N}(r):=\int_0^r \big|s_{K,N}(t)\big|^{N-1}\,\d t.
\]
The following result is the analogous of Theorem \ref{thm:mainthm} in the case \(\X\) as a $\RCD^*(K,N)$ space, see \cite{Ketterer13} for the definition of \((K,N)\)-cone.
\begin{theorem}\label{thm:mainvar}
Let $K\in\R$, $N\in [0,\infty)$, $(\X,\sfd,\mm)$ a $\RCD^*(K,N)$ space, $\O\in \X$ and $\Rb>\rb>0$ such that
\[
\frac{\mm(B_{\Rb}(\O))}{v_{K,N}(\Rb)}=\frac{\mm(B_{\rb}(\O))}{v_{K,N}(\rb)}.
\]
Then exactly one of the following holds:
\begin{itemize}
\item[1)] $S_{\Rb/2}(\O)$ contains only one point. In this case $\X$ is isometric to $[0,{\rm Diam}(\X)]$ ($[0,\infty)$ if $\X$ is unbounded) with an isometry sending $\O$ to 0 and the measure $\mm\restr{B_\Rb(\O)}$ to the measure $c\, s_{K,N}(x)\d x$ for  $c:=\frac{\mm(B_\Rb(\O))}{v_{K,N}(\Rb)}$.
\item[2)] $S_{\Rb/2}(\O)$ contains  two points. In this case $(\X,\sfd)$ is a 1-dimensional Riemannian manifold, possibly with boundary, and there is a bijective local isometry (in the sense of distance-preserving maps) from $B_\Rb(\O)$ to $(-\Rb,\Rb)$ sending $\O$ to $0$ and the measure $\mm\restr{B_\Rb(\O)}$ to the measure  $c\, s_{K,N}(x) \d x$ for $c:=\frac{\mm(B_\Rb(\O))}{2v_{K,N}(\Rb)}$. Moreover, such local isometry  is an isometry when restricted to $\bar B_{\Rb/2}(\O)$.
\item[3)] $S_{\Rb/2}(\O)$ contains more than two points. In this case there exists a $\RCD^*(N-2,N-1)$ space \(\Z\) such that if \(\Y\) is the \((K,N)\) cone built over \(\Z\) with ``origin'' \(\O_Y\), then  $B_\Rb(\O)$ is locally isometric to \(B_\Rb(\O_\Y)\), the corresponding metric ball in \(\Y\). Moreover  the local isometry is a measure-preserving bijection, which, when restricted to $\bar B_{\Rb/2}(\O)$, is an isometry.
\end{itemize} 
\end{theorem}

Exactly as in the case of Theorem \ref{thm:main} the space \(\Z\) is given by a suitable rescaling of the sphere \(S_{\Rb/2}(\O)\) seen with its induced distance. The \emph{proof} of this result follows along the very same lines used to obtain Theorem \ref{thm:main}, the major difference being that here the `Busemann' function is 
\begin{align*}
\b(x)&:=\cos\Big(\sfd(x,\O)\sqrt{\frac{K}{N-1}}\Big)&&\text{if }K>0,\\
\b(x)&:=\cosh\Big(\sfd(x,\O)\sqrt{\frac{|K|}{N-1}}\Big)&&\text{if }K<0.
\end{align*}
Then the computations can be carried over with the same modifications one would do in the smooth case. {Eventually in order to apply Ketterer result in \cite{Ketterer13} to show that \(\Z\) is a $\RCD^*(N-2,N-1)$ one has to perform a blow-up analysis as the one done at the end of the proof of Theorem \ref{thm:main}}
\bigskip

In fact, the very same arguments can also be used to obtain the non-smooth version of the ``volume annulus implies metric annulus'' theorem. The main assumption in this case is that for our $\RCD^*(K,N)$ pointed space there are $0<\rb_1<\rb_2<\rb_3$  such that
\[
\frac{\mm(B_{\rb_3}(\O))-\mm(B_{\rb_2}(\O))}{\int_{\rb_2}^{\rb_3}s^{N-1}_{K,N}(r)\,\d r}=\frac{\mm(B_{\rb_2}(\O))-\mm(B_{\rb_1}(\O))}{\int_{\rb_1}^{\rb_2}s^{N-1}_{K,N}(r)\,\d r}.
\]
Then the same conclusions of the above theorem hold, with point $(3)$ being replaced by:
\begin{itemize}
\item[3')]$S_{\Rb/2}(\O)$ contains more than two points.  In this case, $N\geq 1$. and the annulus ${\rm Ann}^\X_{\rb_1,\rb_3}$ is locally isometric to the corresponding one in the  $(K,N)$-cone built over a suitable rescaling of  $S_{\Rb/2}(\O)$ and the local isometry is a measure-preserving bijection. 
\end{itemize}

The proof is the same. Note however that  there is no ready-to-use result that would grant that the  ``sphere'' of the cone  is, with the induced distance and measure, a $\RCD^*(N-1,N)$ space. The problem is that we only have informations about the annulus ${\rm Ann}^\X_{\rb_1,\rb_3}$, which, being not convex, is certainly not a $\RCD^*$ space. A look to the proof of the aforementioned Ketterer's result, which is needed in Theorems \ref{thm:main} and \ref{thm:mainvar}, seems to suggest that the statement can be modified to be usable in our context, thus showing that even in this case the sphere is a $\RCD^*(N-2,N-1)$ space, but this is outside the scope of this paper.

\bibliographystyle{siam}
\bibliography{biblio}

\end{document}